\documentclass{amsart}

\usepackage{slashed}
\usepackage{graphicx}
\usepackage{amsfonts}
\usepackage{amssymb}
\usepackage{amsmath}
\usepackage{amsxtra}
\usepackage{latexsym}

\newtheorem{theorem}{Theorem}[section]
\newtheorem{lemma}[theorem]{Lemma}

\newtheorem{proposition}{Proposition}

\theoremstyle{definition}

\theoremstyle{remark}

\numberwithin{equation}{section}

\begin{document}

\title[Improved breakdown criterion for Einstein vacuum
equations]
{Improved breakdown criterion for Einstein vacuum
equations in CMC gauge}

\author{Qian Wang}
\address{Department of Mathematics, Stony Brook University, Stony Brook, NY 11794}
\email{qwang@math.sunysb.edu}
\curraddr{} \email{}

\date{March 15, 2010}






\maketitle

\def\pib{\underline{\pi}}
\def\tt{{t'}}
\def\ee{{\mathcal E}_0}
\def\tM{{t_M}}

\def\cp{\stackrel{\circ}\pi_0}
\def\er{\mbox{err}}
\def\tu{{\tilde{u}}}
\def\sn{{\slashed{\nabla}}}

\def\Q{\mathcal{Q}}
\def\fP{{\mathfrak{P}}}
\def\ff{{\mathfrak{F}}}
\def\tx{{\tilde{x}}}
\def\fg{{\mathfrak{g}}}
\def\spi{{\stackrel{\circ}{\pi}}}
\def\tV{{\tilde{V}}}
\def\ta{{{\tau}}}
\def\nub{{\underline{\nu}}}
\def\ab{{\underline{\a}}}
\def\eh{\hat{\eta}}
\def\tF{{\tilde F}}
\def\zb{{\underline{\zeta}}}
\def\bpi{{\bar{\pi}}}
\def\lt{{\tilde{\vartriangle}}}
\def\nt{{\tilde{\triangledown}}}
\def\X{\mathcal{X}}
\def\ft{{\tilde {f}}}
\def\L{{\mathcal{L}}}
\def\J{{\mathcal{J}}}
\def\dG{{\dot{\Gamma}}}
\def\Rh{{\hat{R}}}
\def\Ab{{\underline{A}}}
\def\bM{\textbf{M}}
\def\M{{\mathcal{M}}}
\def\bT{{\textbf{T}}}
\def\bQ{{\textbf{Q}}}
\def\bR{{\textbf{R}}}
\def\eb{{\bar{\epsilon}}}
\def\bd{{\textbf{D}}}
\def\ti{\tilde}
\def\cs{{C_s}}
\def\ip{{i'}}
\def\jp{{j'}}
\def\dg{{\dot{g}}}
\def\bg{{\textbf{g}}}
\def\phd{{\dot{\phi}}}
\def\hk{{\hat{k}}}
\def\I{{\mathcal I}}
\def\Im{{{\mathcal I}_{lnm}}}
\def\pn{{\Pi^{\la}_n}}
\def\pno{{\Pi^{\la}}}
\def\beaa{\begin{eqnarray*}}
\def\eeaa{\end{eqnarray*}}
\def\stuk{S_{t_k,u}}
\def\ba{\begin{array}}
\def\ea{\end{array}}
\def\d{\delta}
\def\be#1{\begin{equation} \label{#1}}
\def \eeq{\end{equation}}
\newcommand{\eql}{\eqlabel}
\newcommand{\nn}{\nonumber}
\def\medn{\medskip\noindent}
\def\l{\langle}
\def\r{\rangle}
\def\rrrr{{\Bbb R}}
\def\rr{{\bf R}}
\def\PP{{\mathcal P}}
\def\pih{\hat{\pi}}
\def\cir{\overset\circ}
\def\nn{\nonumber}
\def\S{{\mathcal S}}

\def\err{\mbox{Err}}
\def\MM{{\mathcal M}}
\def\Ml{{\mathcal M}_{\ell}}
\def\cga{\overset\circ{\ga}}
\def\AA{{\mathcal A}}
\def\ck{{\check K}}
\def\aa{{\underline{a}}}

\def\dnt{\dot{\nab}}
\def\as{{}^\ast}
\def\mm{{\bf m}}
\def\ns{{\bf n}}
\def\vs{{\bf v}}
\def\S2{{\Bbb S}^2}
\def\Stu{S_{t,u}}

\def\ck{\check}
\def\A{\mathcal {A}}
\def\bopi{{\bold{\pi}}}
\def\V{{\mathcal V}}
\def\bq{{\textbf{Q}}}
\def\ta{{\tilde a}}
\def\tA{{\tilde A}}
\def\k{{\mathcal K}_0}
\def\E{{\mathcal E}}
\def\K{{\underline{K}}}
\def\U{{\mathcal U}}
\def\V{{\mathcal V}}
\def\W{{\mathcal W}}
\def\ze{{\zeta}}
\def\esb{{\underline{\epsilon}}}

\def\ub{\underline{u}}
\def\udb{\underline{\b}}
\def\Lb{\underline{L}}
\def\Lie{{\mathcal L}}
\def\tr{\mbox{tr}}
\def\err{\mbox{Err}}
\def\bA{{\textbf{A}}}
\def\D{{\mathcal D}}
\def\H{{\mathcal H}}
\def\MM{{\mathcal M}}
\def\Ml{{\mathcal M}_{\ell}}
\def\N{{\mathcal N}}
\def\cga{\overset\circ{\ga}}
\def\AA{{\mathcal A}}
\def\La{{\Lambda}}
\def\rrrr{{\Bbb R}}
\def\rr{{\bf R}}
\def\B{{\mathcal B}}
\def\F{{\mathcal F}}
\def\P{{\mathcal P}}
\def\R{{\mathcal R}}
\def\c{\cdot}
\def\hot{\widehat{\otimes}}
\def\sig{\sigma}
\def\s{\sigma}
\def\a{\alpha}
\def\b{\beta}
\def\e{\eta}
\def\g{\gamma}
\def\ep{\epsilon}
\def\l{\langle}
\def\r{\rangle}
\def\ga{\gamma}
\def\Ga{\Gamma}
\def\la{\lambda}
\def\crho{\check\rho}
\def\csig{\check\sigma}
\def\cudb{\check{\underline{\b}}}
\def\OG{\overline{\Gamma}}
\def\opsi{\overline{\psi}}
\def\p{\partial}
\def\P{{\mathcal P}}
\def\th{\theta}
\def\nab{\nabla}
\def\F{{\mathcal{F}}}
\def\varep{\varepsilon}
\def\C#1{{\mathcal C}_{[#1]}}
\def\A#1{{\mathcal A}_{[#1]}}
\def\B#1{{\mathcal B}_{[#1]}}
\def\bb{{\underline{\b}}}
\def\omb{{\underline{\om}}}
\def\Lb{{\underline{L}}}
\def\aaa{{\mathbf a}}
\def\ggg{{\mathbf g}}
\def\div{\mbox{\,div\,}}
\def\curl{\mbox{\,curl\,}}
\def\ddiv{\mbox{\,\bf{div}\,}}
\def\ccurl{\mbox{\,\bf{curl}\,}}
\def\tr{\mbox{tr}}
\def\Tr{\mbox{Tr}}
\def\db{\mbox{\bf{b}}}

\def\tir{{\tilde r}}
\def\tg{{\tilde g}}
\def\sp{{\sigma_p}}
\def\itt{{\mbox{Int}}}
\def\tlN{{\tilde N}}
\def\dual{{\,^\star \mkern-4mu}}
\def\fE{\mathfrak{E}}
\def\sA{\slashed{A}}
\def\wt{\widetilde}
\def\f14{\frac{1}{4}}
\def\f12{{\frac{1}{2}}}
\def\dd{{\bf D}}
\def\t1a{t^{-\frac{1}{a}}}
\def\kp{{\kappa}}
\def\bm{{\bf m}}
\def\sk{{\slashed{k}}}
\def\sl{\slashed}
\def\sD{\slashed{\Delta}}
\def\cff{\stackrel{\circ}\ff}

\def\ee{{\mathcal E}_0}
\def\tM{{t_M}}

\def\cp{\stackrel{\circ}\pi_0}
\def\er{\mbox{err}}
\def\tu{{\tilde{u}}}
\def\sn{{\slashed{\nabla}}}
\def\cc{C(Q_0,\k)}

\def\fP{{\mathfrak{P}}}
\def\ff{{\mathfrak{F}}}
\def\tx{{\tilde{x}}}
\def\fg{{\mathfrak{g}}}
\def\spi{{\stackrel{\circ}{\pi}}}
\def\tV{{\tilde{V}}}
\def\ta{{{\tau}}}
\def\nub{{\underline{\nu}}}
\def\ab{{\underline{\a}}}
\def\eh{\hat{\eta}}
\def\tF{{\tilde F}}
\def\zb{{\underline{\zeta}}}
\def\bpi{{\bar{\pi}}}
\def\lt{{\tilde{\vartriangle}}}
\def\nt{{\tilde{\triangledown}}}
\def\X{\mathcal{X}}
\def\ft{{\tilde {f}}}
\def\L{{\mathcal{L}}}
\def\J{{\mathcal{J}}}
\def\dG{{\dot{\Gamma}}}
\def\Rh{{\hat{R}}}
\def\Ab{{\underline{A}}}
\def\bM{\textbf{M}}
\def\M{{\mathcal{M}}}
\def\bT{{\textbf{T}}}
\def\bQ{{\textbf{Q}}}
\def\bR{{\textbf{R}}}
\def\eb{{\bar{\epsilon}}}
\def\bd{{\textbf{D}}}
\def\ti{\tilde}
\def\cs{{C_s}}
\def\ip{{i'}}
\def\jp{{j'}}
\def\dg{{\dot{g}}}
\def\bg{{\textbf{g}}}
\def\phd{{\dot{\phi}}}
\def\hk{{\hat{k}}}
\def\I{{\mathcal I}}
\def\Im{{{\mathcal I}_{lnm}}}
\def\pn{{\Pi^{\la}_n}}
\def\pno{{\Pi^{\la}}}
\def\beaa{\begin{eqnarray*}}
\def\eeaa{\end{eqnarray*}}
\def\stuk{S_{t_k,u}}
\def\ba{\begin{array}}
\def\ea{\end{array}}
\def\be#1{\begin{equation} \label{#1}}
\def \eeq{\end{equation}}
\def\nn{\nonumber}
\def\medn{\medskip\noindent}
\def\l{\langle}
\def\r{\rangle}
\def\rrrr{{\Bbb R}}
\def\rr{{\bf R}}
\def\PP{{\mathcal P}}
\def\pih{\hat{\pi}}
\def\cir{\overset\circ}
\def\nn{\nonumber}
\def\S{{\mathcal S}}
\def\vr{\varrho}

\def\err{\mbox{Err}}
\def\MM{{\mathcal M}}
\def\Ml{{\mathcal M}_{\ell}}
\def\cga{\overset\circ{\ga}}
\def\AA{{\mathcal A}}
\def\ck{{\check K}}
\def\aa{{\underline{a}}}

\def\dnt{\dot{\nab}}
\def\as{{}^\ast}
\def\mm{{\bf m}}
\def\ns{{\bf n}}
\def\vs{{\bf v}}
\def\S2{{\Bbb S}^2}
\def\Stu{S_{t,u}}

\def\ck{\check}
\def\A{\mathcal {A}}
\def\bopi{{\bold{\pi}}}
\def\V{{\mathcal V}}
\def\bq{{\textbf{Q}}}
\def\ta{{\tilde a}}
\def\tA{{\tilde A}}
\def\k{{\mathcal K}_0}
\def\E{{\mathcal E}}
\def\K{{\underline{K}}}
\def\U{{\mathcal U}}
\def\V{{\mathcal V}}
\def\W{{\mathcal W}}
\def\ze{{\zeta}}
\def\esb{{\underline{\epsilon}}}

\def\ub{\underline{u}}
\def\udb{\underline{\b}}
\def\Lb{\underline{L}}
\def\Lie{{\mathcal L}}
\def\tr{\mbox{tr}}
\def\err{\mbox{Err}}
\def\bA{{\textbf{A}}}
\def\D{{\mathcal D}}
\def\H{{\mathcal H}}
\def\MM{{\mathcal M}}
\def\Ml{{\mathcal M}_{\ell}}
\def\N{{\mathcal N}}
\def\cga{\overset\circ{\ga}}
\def\AA{{\mathcal A}}
\def\La{{\Lambda}}
\def\rrrr{{\Bbb R}}
\def\rr{{\bf R}}
\def\B{{\mathcal B}}
\def\F{{\mathcal F}}
\def\P{{\mathcal P}}
\def\R{{\mathcal R}}
\def\c{\cdot}
\def\hot{\widehat{\otimes}}
\def\sig{\sigma}
\def\s{\sigma}
\def\a{\alpha}
\def\b{\beta}
\def\e{\eta}
\def\g{\gamma}
\def\ep{\epsilon}
\def\l{\langle}
\def\r{\rangle}
\def\ga{\gamma}
\def\Ga{\Gamma}
\def\la{\lambda}
\def\crho{\check\rho}
\def\csig{\check\sigma}
\def\cudb{\check{\underline{\b}}}
\def\OG{\overline{\Gamma}}
\def\opsi{\overline{\psi}}
\def\p{\partial}
\def\P{{\mathcal P}}
\def\th{\theta}
\def\nab{\nabla}
\def\F{{\mathcal{F}}}
\def\varep{\varepsilon}
\def\C#1{{\mathcal C}_{[#1]}}
\def\A#1{{\mathcal A}_{[#1]}}
\def\B#1{{\mathcal B}_{[#1]}}
\def\bb{{\underline{\b}}}
\def\omb{{\underline{\om}}}
\def\Lb{{\underline{L}}}
\def\aaa{{\mathbf a}}
\def\ggg{{\mathbf g}}
\def\div{\mbox{\,div\,}}
\def\curl{\mbox{\,curl\,}}
\def\ddiv{\mbox{\,\bf{div}\,}}
\def\ccurl{\mbox{\,\bf{curl}\,}}
\def\tr{\mbox{tr}}
\def\Tr{\mbox{Tr}}
\def\db{\mbox{\bf{b}}}

\def\tir{{\tilde r}}
\def\tg{{\tilde g}}
\def\sp{{\sigma_p}}
\def\itt{{\mbox{Int}}}
\def\tlN{{\tilde N}}
\def\dual{{\,^\star \mkern-4mu}}
\def\fE{\mathfrak{E}}
\def\sA{\slashed{A}}
\def\wt{\widetilde}
\def\f14{\frac{1}{4}}
\def\f12{{\frac{1}{2}}}
\def\dd{{\bf D}}
\def\t1a{t^{-\frac{1}{a}}}
\def\kp{{\kappa}}
\def\bm{{\bf m}}
\def\sk{{\slashed{k}}}
\def\sl{\slashed}
\def\cc{C(Q_0, \k)}
\def\sD{\slashed{\Delta}}
\def\cff{\stackrel{\circ}\ff}

\def\cc{C(Q_0, \k)}
\def\TC{{TRC}}
\def\crt{{crt}}
\def\cru{{cru}}
\def\crv{{crv}}
\def\crs{{crs}}
\def\nd{{\dot{n}}}
\def\RR{{\mathcal R}_0}
\def\os#1{{\mathcal Osc}(#1)}
\newcommand{\bea}{\begin{eqnarray}}
\newcommand{\eea}{\end{eqnarray}}
\def\eql{\eqlabel}
\def\nn{\nonumber}
\newcommand{\lapp}{\mbox{$\bigtriangleup  \mkern-13mu / \,$}}
\newcommand{\CC}{\mbox{$C  \mkern-13mu / \,$}}
\newcommand{\trchb}{\tr \chib}
\newcommand{\chih}{\hat{\chi}}
\newcommand{\chib}{\underline{\chi}}
\newcommand{\xib}{\underline{\xi}}
\newcommand{\etab}{\underline{\eta}}
\newcommand{\chibh}{\underline{\hat{\chi}}\,}
\newcommand{\les}{\lesssim}
\newcommand{\ges}{\gtrsim}
\newcommand\twonorm[2]{\norm{#1}_{L^2#2}}
\newcommand\ovl{\overline}
\newcommand\und{\underline}
\newcommand{\half}{{\frac 12}undefined }
\newcommand{\piK}{\,^{(K)}\pi}
\newcommand{\piX}{\,^{(X)}\pi}
\newcommand{\Pa}{\,^{(a)}P}
\newcommand{\Pb}{\,^{(b)}P}
\newcommand{\Pm}{\,^{(m)}P}

\def\gac{\stackrel{\circ}\ga}

\newtheorem{assumption}{Assumption}


\section{\bf Introduction} \label{intr}
\setcounter{equation}{0}

Let $(\bM, \bg)$ be a (3+1)-dimensional vacuum globally hyperbolic
space-time, i.e. $\bg$ is a Lorentz metric of signature $(-, +, +,
+)$ satisfying the Einstein vacuum equations
$$
{\bf Ric}(\bg)=0
$$
and every causal curve intersects a Cauchy surface at precisely one
point. If $(\bM, \bg)$ has a compact, constant mean curvature (CMC)
Cauchy surface $\Sigma_0$ with mean curvature $t_0<0$, then there
exists a foliation of a neighborhood of $\Sigma_0$ by compact CMC
surfaces, and the mean curvature varies monotonically from slice to
slice. The CMC conjecture states that there is a foliation in $\bM$
of CMC Cauchy surfaces with mean curvatures taking on all allowable
values, i.e. the mean curvatures take all values in $(-\infty, 0)$
if $\Sigma_0$ is of Yamabe type $-1$ or $0$, while the mean
curvatures take on all values in $(-\infty, \infty)$ if $\Sigma_0$
is of Yamabe type $+1$. Certain progress has been made
(\cite{An02}), the CMC conjecture however remains open. One of the
important step to attack the CMC conjecture is to provide a
reasonable breakdown criterion to detect what may happen when the
 CMC foliation can not be extended.

In order to set up the framework, in this paper we assume that $\M_*$ is a part
of the space-time $(\bM, \bg)$ foliated by CMC hypersurfaces $\Sigma_t$ with mean
curvature $t$ satisfying $t_0\le t<t_*$ for some $t_0<t_*<0$. We
shall refer to $\Sigma_0:=\Sigma_{t_0}$ as the initial slice. Thus,
$\M_*=\bigcup_{t\in [t_0, t_*)} \Sigma_t$ with $t_*<0$ and there is a
time function $t$ defined on $\M_*$, monotonically increasing toward
the future, such that each $\Sigma_t$ is a level hypersurface of $t$
with the lapse function $n$ and the second fundamental form $k$
defined by
$$
n:=\left(-\bg(\bd t, \bd t)\right)^{1/2} \quad \mbox{and}\quad k(X,
Y):=-\bg (\bd_X \bT, Y),
$$
 where $\bT$ denotes the future directed unit normal to
$\Sigma_t$, $\bd$ denotes the space-time covariant differentiation
associated with $\bg$, and $X, Y$ are vector fields tangent to
$\Sigma_t$. Let $g$ be the induced Riemannian metric on $\Sigma_t$
and let $\nabla$ be the corresponding covariant differentiation. For
any coordinate chart ${\mathcal O}\subset \Sigma_0$ with coordinates
$x=(x^1, x^2, x^3)$, let $x^0=t, x^1, x^2, x^3$ be the transported
coordinates on $[t_0, t_*)\times {\mathcal O}$ obtained by following
the integral curves of $\bT$. Under these coordinates the metric
$\bg$ takes the form
\begin{equation}\label{3.1.0}
\bg=-n^2 dt^2 +g_{ij} dx^i dx^j.
\end{equation}
Moreover, relative to these coordinates $t, x^1, x^2, x^3$ there hold the evolution equations
\begin{align}
\p_t g_{ij}&=-2n k_{ij}, \label{bg}\\
\p_t k_{ij}&=-\nab_i \nab_j n+n(R_{ij}+\Tr k \, k_{ij}-2
k_{ia}k^a_j)\label{intro04}
\end{align}
and the constraint equations
\begin{align}
&R-|k|^2+(\Tr k)^2=0,\label{intr.1}\\
&\nab^j k_{ji}-\nab_i \Tr k=0\label{intr.2}
\end{align}
on each $\Sigma_t$, where $R_{ij}$ and $R$ denote the Ricci
curvature and the scalar curvature of the induced metric $g$ on
$\Sigma_t$, and $\Tr k$ denotes the trace of $k$, i.e. $\Tr k=g^{ij}
k_{ij}$. Since $\Tr k=t$ on $\Sigma_t$, it follows from the above
equations that
\begin{equation}\label{3.1.1}
\mbox{div} k=0
\end{equation}
and
\begin{equation}\label{3.1.2}
-\Delta n+|k|^2 n=1
\end{equation}
on each $\Sigma_t$.

The first important breakdown criterion was given by M. Anderson in
\cite{And}, who showed that when a breakdown occurs at $t_*$ there
holds
$$
\limsup_{t\rightarrow t_*^{-}} \|{\bf
R}\|_{L^\infty(\Sigma_t)}=\infty,
$$
where ${\bf R}$ denotes the Riemannian curvature tensor of the
space-time $(\bM, \bg)$. Here the pointwise norm $|{\bf R}|$ is
defined with respect to the Riemannian metric ${\bf g}_{\bT}$ on
$\bM$, where $\bg_{\bT}$ is defined as follows: for any $X, Y\in
T\M_*$ write
$$
X=X^0\bT+ \underline{X} \quad \mbox{and}\quad Y=Y^0 \bT+
\underline{Y}
$$
with $\underline{X}, \underline{Y}\in T \Sigma_t$, then
$$
\bg_{\bT}(X, Y)=X^0Y^0+ g(\underline{X}, \underline{Y}).
$$
The result of Anderson implies that if
\begin{equation}\label{An1}
\sup_{t\in [t_0, t_*)} \|\bR\|_{L^\infty(\Sigma_t)}=
\Lambda_0<\infty
\end{equation}
for all $t_*<0$, then the CMC foliation exists for all values in
$[t_0, 0)$.

Recently, Klainerman and Rodnianski \cite{KR2} provided a new
breakdown criterion which shows that if a breakdown happens at
$t_*<0$ then
$$
\limsup_{t\rightarrow t_*^{-}}
\left(\|k\|_{L^\infty(\Sigma_t)}+\|\nabla \log
n\|_{L^\infty(\Sigma_t)} \right)=\infty,
$$
or, in other words, the CMC foliation can be extended beyond any
value $t_*<0$ for which
\begin{equation}\label{tn1}
\sup_{t\in[t_0, t_*)}\left(\|k\|_{L^\infty(\Sigma_t)}+\|\nab \log
n\|_{L^\infty(\Sigma_t)}\right)=\Lambda_0<\infty.
\end{equation}
In contrast to the breakdown criterion of Anderson, the condition
(\ref{tn1}) of Klainerman and Rodnianski is formally weaker as it
refers only to the second fundamental form $k$ and the lapse
function $n$ which requires one degree less of differentiability.
Moreover, by purely elliptic estimates, one can see that (\ref{An1})
implies immediately (\ref{tn1}), since the boundedness of $\|{\bf
R}\|_{L^\infty}$ exhausts all the dynamical degrees of freedom of
the equations. Therefore, the result in \cite{KR2} is a significant
improvement. We remark that the result of Klainerman and Rodnianski
can not be established by purely elliptic estimates. Instead, the
proof relies heavily on the tools from the theory of hyperbolic
equations. The analogous result has been extended to non-vacuum
space-time in \cite{Shao}.

If we consider the Einstein equation expressed relative to the wave
coordinates, by energy estimates one can see that the breakdown
does not occur unless
\begin{equation}\label{pg}
\int_{t_0}^{t_*} \|\p \bg\|_{L^\infty} dt=\infty.
\end{equation}
This condition however is not geometric since it depends on the
choice of a full coordinate system. Observe that the components of
the second fundamental form $k$ and $\nabla n$ can be viewed as part
of the components of $\p\bg$. It is natural to ask if we have an
integral form of breakdown criterion involving $k$ and $n$ only. The
first main result of the present paper confirms this and provides a
geometric counterpart of (\ref{pg}), which can be viewed as an
improved version of the breakdown criterion of Klainerman and
Rodnianski.

\begin{theorem}[Main theorem I]\begin{footnote}{
Our method applies equally well to the case that $\Sigma_t$
are asymptotically flat and maximal, i.e $\Tr k=0$ and
 can also be extended to Einstein space-time with matter.}
\end{footnote}\label{thm3}
Let $(\M_*, \bg)$ be a globally hyperbolic development of $\Sigma_0$
foliated by the CMC level hypersurfaces of a time function $t<0$.
Then the space-time together with the foliation $\Sigma_t$ can be
extended beyond any value $t_*<0$ for which,
\begin{equation}\label{add}
\int_{t_0}^{t_*} \left(\|k\|_{L^\infty(\Sigma_t)}+\|\nab \log n\|_{
L^\infty(\Sigma_t)}\right) d t=\k<\infty.
\end{equation}
\end{theorem}

 Let us fix the convention for the deformation tensor of $\bT$,  expressed relative
to an orthonormal frame $\{e_0=\bT, e_1, e_2, e_3\}$, as follows,
$$\pi_{\a\b}=-\bg( \bd_{e_{\a}} \bT, e_\b),\mbox{ with } \a, \b=0,1,2,3.$$
It is easy to check
$$\pi_{00}=0, \, \pi_{0i}=-\nab_i \log n,\, \pi_{i0}=0,\,
\pi_{ij}=k_{ij}, \mbox{ with } i,j=1,2,3.$$
 Consequently, the condition (\ref{tn1}) can be formulated as
$$
\sup_{t\in [t_0, t_*)} \|\pi\|_{L^\infty(\Sigma_t)}= \Lambda_0<\infty,
$$
while the weaker condition (\ref{add}) can be formulated as
\begin{equation*}
\|\pi\|_{L_t^1 L_x^\infty(\M_*)}:=\int_{t_0}^{t_*} \|\pi\|_{L^\infty(\Sigma_t)} dt
=\k<\infty. \tag{{\bf A1}}
\end{equation*}

We basically follow the framework in \cite{KR2} to prove Theorem \ref{thm3}; however,
a sequence of difficulties occur due to the weaker condition (\ref{add}).
In order to continue the foliation, according to the local existence
theorem given in \cite[Theorem 10.2.1]{KC}, one must establish a global uniform bound for
the curvature tensor ${\bf R}$ and $L^2$-bounds for its first two covariant derivatives.
Since $(\bM, \bg)$ is a vacuum space-time, by virtue of the Bianchi identity ${\bf R}$
verifies a wave equation of the form
\begin{equation}\label{3.1.3}
\Box_{\bg} {\bf R}={\bf R} \star {\bf R},
\end{equation}
where $\Box$ denotes the covariant wave operator $\Box =\bd^\a
\bd_\a$. Based on higher energy estimates it is standard to show
that the $L^2$ bounds for $\bd {\bf R}$ and $\bd^2 {\bf R}$ can be
bounded in terms of the $L^\infty$ norm of ${\bf R}$. Thus, the
derivation of the $L^\infty$ bound of ${\bf R}$ is a crucial step.
In order to achieve this goal, Klainerman and Rodnianski \cite{Ksob}
succeeded in representing ${\bf R}(p)$, for each $p\in \M_*$, by a
Kirchoff-Sobolev formula of the form
$$
{\bf R}(p)=-\int_{\N^{-}(p, \tau)} {\bf A} \cdot ({\bf R}\star {\bf
R}) +\mbox{other terms}
$$
where ${\bf A}$ is a $4$-covariant tensor defined as a solution of a
transport equation along $\N^{-}(p, \tau)$ with appropriate initial
data at the vertex $p$, $\N^{-}(p, \tau)$ denotes the portion of the
null boundary $\N^{-}(p)$ in the time interval $[t(p)-\tau, t(p)]$.
The past null cone $\N^{-}(p)$ is in general an achronal Lipschitz
hypersurface ruled by the set of past null geodesics from $p$. In
order to derive all necessary estimates, one must show that
$\N^{-}(p)$ remains a smooth hypersurface in the time slab
$[t(p)-\tau, t(p))$ for some universal constant $\tau>0$. Therefore,
it is necessary to provide a uniform lower bound for the past null
radius of injectivity at all $p\in \M_*$.

Let us recall briefly the definition of the past null radius of injectivity at $p$, one may consult
\cite{KRradius} for more details. We
parametrize the set of past null vectors in $T_p \bM$ in terms of
$\omega\in {\Bbb S}^2$, the standard sphere in ${\Bbb R}^3$.
Then, for each $\omega\in {\Bbb S}^2$, let $l_{\omega}$ be the null vector
in $T_p\bM$ normalized with respect to the future, unit, time-like vector $\bT_p$ by
$$
\bg(l_\omega, \bT_p)=1
$$
and let $\Gamma_\omega(s)$ be the past null geodesic with initial data
$\Gamma_\omega(0)=p$ and $\frac{d}{ds}\Gamma_{\omega}(0)=l_\omega$.
We define the null vector field $L$ on $\N^{-}(p)$ by
$$
L(\Ga_\omega(s))=\frac{d}{ds}\Ga_\omega(s)
$$
which may only be smooth almost everywhere on $\N^-(p)$ and can be multi-valued on a set of
exceptional points. We can choose the parameter $s$ with $s(p)=0$ so that
$$
\bd_L L=0\quad  \mbox{and}\quad L(s)=1.
$$
This $s$ is called the affine parameter.

The past null radius of injectivity $i_*(p)$ at $p$ is then defined to be the supremum
over all the values $s_0>0$ for which the exponential map
$$
\fg_p: (s, \omega)\rightarrow
\Gamma_\omega (s)
$$
is a global diffeomorphism from $(0, s_0)\times {\Bbb S}^2$ to its
image in $\N^-(p)$. It is known that $i_*(p)>0$ for each $p$,
$\N^{-}(p)$ is smooth within the null radius of injectivity, and
$$
i_*(p)=\min\{s_*(p), l_*(p)\},
$$
where $s_*(p)$, the past null radius of conjugacy at $p$, is defined to be
the supremum over all values $s_0>0$ such that the exponential
map $\fg_p$ is a local diffeomorphism from $(0, s_0)\times {\Bbb S}^2$ to its image in
$\N^{-}(p)$, and $l_*(p)$, the past cut locus radius at $p$, is defined to be
the smallest value of $s_0$ for which there exist two distinct null geodesics
$\Gamma_1$ and $\Gamma_2$ from $p$ with $\Gamma_1(s_0)=
\Gamma_2(s_0)$. Thus, for a past null geodesic $\Gamma_{\omega}$ from $p$,
a point $q=\Gamma_\omega(s_*)$ is called  a conjugate
point of $p$ if $\fg_p$ is singular at $(s_*, \omega)$, while it is called a
null cut point of $p$ if $\fg_p$ is nonsingular at
$(s_*, \omega)$ and through $q$ there exists another null
geodesic emanating from $p$.

Since we are working on the CMC foliation, it is convenient to
introduce the past null radius of injectivity $i_*(p, t)$ at each
$p$ with respect to the global time function $t$. We define
$i_*(p,t)$ to be the supremum over all the values $\tau>0$ for which
the exponential map
\begin{equation}\label{gp}
{\mathcal G}_p: (t, \omega)\rightarrow \Ga_\omega(s(t))
\end{equation}
is a global diffeomorphism from $(t(p)-\tau, t(p))\times {\Bbb S}^2$
to its image in $\N^{-}(p)$. We remark that $s$ is a function not
only depending on $t$ but also on $\omega$, we suppress $\omega$
just for convenience. It is known that
$$
i_*(p,t)=\min\{s_*(p,t), l_*(p,t)\},
$$
where $s_*(p, t)$ is defined to be the supremum over all values
$\tau>0$ such that the map ${\mathcal G}_p$ is a local
diffeomorphism from $(t(p)-\tau, t(p))\times {\Bbb S}^2$ to its
image, and $l_*(p,t)$ is defined to be the smallest value of
$\tau>0$ for which there exist two distinct null geodesics
$\Ga_1(s(t))$ and $\Ga_2(s(t))$ from $p$ which intersect at a point
with $t=t(p)-\tau.$

In \cite{KRradius} Klainerman and Rodnianski provided a uniform
lower bound on the null radius of injectivity under the assumption
(\ref{tn1}). In order to complete the proof of Theorem \ref{thm3}, one must
provide a uniform lower bound on the null radius of injectivity under
the weaker condition (\ref{add}). This is contained in the second main result
of the present paper.

\begin{theorem}[Main theorem II] \label{thm4}
Assume that $\M_*$ is a globally hyperbolic development of $\Sigma_0$
verifying the condition (\ref{add}). Then for all $p\in \M_*$ there holds
\begin{equation}
i_*(p, t)> \min\{\delta_*, t(p)-t_0\},
\end{equation}
where $\delta_*>0$ is a constant depending only on $Q_0$, $\k$, $|\Sigma_0|$ and $t_*$.
\begin{footnote}
{$Q_0$ denotes the Bel-Robinson energy on the initial slice $\Sigma_0$
which will be defined in Section 2.}
\end{footnote}
\end{theorem}

In order to prove this result, it is useful to review the essential steps in the work
of Klainerman and Rodnianski in \cite{KRradius}. The first step is
to show that
\begin{equation}\label{3.1.4}
s_*(p, t)>\min\{l_*(p, t), \delta_*\}
\end{equation}
for some universal constant\begin{footnote} {A universal constant
always means a constant depending only on $Q_0$, $\k$, $|\Sigma_0|$,
$t_*$ and the number $I_0>0$ such that $I_0^{-1}\le (g_{ij})\le I_0$
on the initial slice $\Sigma_0$. Throughout this paper $C$ always denotes
a universal constant.}
\end{footnote}$\delta_*>0$.
This can be achieved by showing that
\begin{equation}\label{3.1.5}
\sup_{\N^{-}(p, \tau)}\left|\tr\chi -\frac{2}{s(t)}\right|\le C
\end{equation}
with $\tau:=\min\{l_*(p,t), \delta_*\}$, where $\chi$ is the null
second fundamental form $\chi_{AB}=\bg(\bd_A L, e_B)$ of the
2-dimensional space-like surface $S_t:=\N^{-}(p)\cap \Sigma_t$ with
$(e_A)_{A=1,2}$ being a frame field tangent to $S_t$. The analog has
been carried out in \cite{KR1,KR4,KR3,Qwang} for geodesic foliations
under the boundedness assumption of the curvature flux. In order to
adapt those arguments to prove (\ref{3.1.5}) for the time
foliations, one needs to show  that $t(p)-t$ and $s$ are comparable
and the geodesic curvature flux (see  \cite{KRradius}) is bounded,
both of which rely on the relation
\begin{equation}\label{a1}
 |a-1|\le \f12 \quad \mbox{ on } \N^-(p,\tau),
\end{equation}
where $a$, the null lapse function, is defined by $a^{-1}:=\bg(\bT,
L)$  with $a(p)=1$. Note that along a null geodesic
$$
 \frac{dt}{ds}=-(an)^{-1},\quad \frac{d a}{d s}=\nu,\quad
\nu:=k_{NN}-\nab_N\log n,
$$
where $N$ is the unit inward normal of $S_t$ in $\Sigma_t$.  If
(\ref{tn1}) is satisfied, one can see that (\ref{a1}) holds for
$t(p)-\delta_*\le t\le t(p)$ for some universal $\delta_*>0$, and
consequently  $s$ and $t(p)-t$ are comparable. However, under the
weaker condition (\ref{add}) only, it is highly nontrivial to obtain
(\ref{a1}). We observe that (\ref{a1}) can be achieved by
establishing
\begin{equation}\label{nu1}
\|\nu\|_{L_\omega^\infty L_t^2(\N^{-}(p, \tau))}^2 =\sup_{\omega\in
{\Bbb S}^2} \int_{\Gamma_\omega} an |\nu|^2 dt\le C
\end{equation}
where $\Gamma_\omega$ is the portion of a past null geodesic
initiating from $p$ contained in $\N^-(p,\tau)$, for some universal
constant $\delta_*>0$. How to obtain such an estimate on $\nu$ is
the first difficulty we encounter. The idea to derive the trace
estimate (\ref{nu1}) is to employ the techniques in the proof of the
sharp trace inequality in \cite{KR1,KR4,KR3}. Under the assumption
(\ref{add}) only, suppose the sharp trace inequality holds true on
null cone in time foliation, in order to prove (\ref{nu1}),
schematically, we need to prove

i)  there holds for $\sn \nu$ the decomposition
\begin{equation}\label{eq}
\sn\nu=\nab_L P+Q
\end{equation}
 with $P$ and $Q$ appropriate $S_t$ tangent tensors.
 \begin{footnote}{$\sn$ denotes the connection with respect to the
induced metric $\ga$ on $S_t$. }\end{footnote}

ii) there holds
\begin{equation}\label{eq3}
 \|\sn (\nu, P)\|_{L^2(\N^-(p, \tau))}+\|\nab_L (\nu,
P)\|_{L^2(\N^-(p,\tau))}\le C.
\end{equation}

 The decomposition of the form
(\ref{eq}) will be derived in \cite{tqwang}.
 To prove the sharp trace inequality in time foliation and to
control $P$ and $Q$ must be coupled with the proof of a series of
estimates for the Ricci coefficients on null hypersurface
$\N^-(p,\tau)$ including (\ref{3.1.5}) by a delicate bootstrap
argument. Hence, under the condition (\ref{add}) only,
(\ref{3.1.5}), (\ref{a1}) and (\ref{nu1}) should be proved
simultaneously. The proof is rather involved and close to the spirit
of the works \cite{KR1,KR4,KR3,Qwang}. We will present it in
\cite{tqwang} with full details.

 Now we simply consider how to obtain the estimate for $\nu$ in
 (\ref{eq3}). The  estimate for $\nab_N \log n $ of the form  (\ref{eq3}) can be
obtained by elliptic estimates and trace inequality. By elliptic
estimate, in view of
\begin{equation}\label{divcurl}
\div k=0, \curl k=H,
\end{equation}
where $H$ denotes the magnetic part of $\bR$,
 we can only derive $\|k\|_{H_x^1(\Sigma)}\le C$, which, by classic
 trace theorem, loses $1/2$ derivative if
restricted to null cone. However, (\ref{eq3}) requires  the $L^2$
control of one derivative of $k_{NN}$ on null cones. Hence, we must
adopt a different approach, which significantly surpasses the one
via elliptic estimate and trace inequality. This inspires us to use
the tensorial wave equation for $k$, which symbolically is given by
\begin{equation}\label{hkk}
\Box k= k\cdot Ric+ n^{-2} \nabla^2 \dot n+\pi \cdot \nabla k-n^{-3}
\dot n \nabla^2 n + \pi \cdot \pi \cdot \pi +k\cdot \nabla^2 n
-n^{-1} k.
\end{equation}
 We then prove by energy method, the $k$-flux satisfies
\begin{equation}\label{kk}
 \|\sn
k\|_{L^2(\N^-(p,\tau))}+\|\nab_L k\|_{L^2(\N^-(p,\tau))}\le C,
\end{equation}
 which schematically gives the desired control on $k_{NN}$.

The next step is to find  a system of good local space-time
coordinates under which $\bg$ is comparable with the Minkowski
metric. More precisely, for a sufficiently small constant $\ep>0$,
one needs to show that there exists a constant $\delta_*>0$,
depending only on $\ep$ and some universal constants, for which each
geodesic ball $B_{\delta_*}(p)$ with $p\in \Sigma_t$ admits local
coordinates $x=(x^1, x^2, x^3)$ such that under the corresponding
transport coordinates $x^0=t, x^1, x^2, x^3$ the metric $\bg$ has
the expression (\ref{3.1.0}) with
\begin{equation}\label{3.1.6}
|n-n(p)|\le \ep\qquad \mbox{and} \qquad |g_{ij}-\delta_{ij}|\le \ep
\end{equation}
on $B_{\delta_*}(p)\times [t(p)-\delta_*, t(p)]$. The existence of
such local coordinates together with (\ref{3.1.5}) will enable us to
show that $\N^{-}(p, \delta_*)$ is close to the flat cone and
consequently $l_*(p, t)\ge \delta_*$.

The part on $n$ in (\ref{3.1.6}) can be established by elliptic
estimates on $n$ and $\p_t n$. The derivation of the  result for $g$
 under the weaker condition (\ref{add}), however, presents one of
the core difficulties, which invokes new methods and a second
application of (\ref{hkk}).

By the Bel-Robinson energy bound $\Q(t)\le C$ and a result of
Anderson \cite{An97}, one can control the lower bound of harmonic
radius  on $\Sigma_t$, such that with the coordinates $x=(x^1, x^2,
x^3)$ on $B_{\delta_*}(p)\subset \Sigma_t$,
$$
|g_{ij}(x, t(p))-\delta_{ij}|\le \f12\ep.
$$
The challenge is to control time evolution of $g$. Using (\ref{bg}),
one has \begin{footnote}
{We use $\Phi_1\les \Phi_2$
to mean that $\Phi_1\le C\Phi_2$ for some universal constant $C$.}
\end{footnote}
\begin{equation}\label{gev}
|g_{ij}(x, t(p))-g_{ij}(x, t)| \les \int_t^{t(p)} |k(x, \tt)| d\tt.
\end{equation}
If (\ref{tn1}) holds, or more generally, if
$$
\int_{t_0}^{t_*} \|k(\tt)\|_{L^\infty(\Sigma_{\tt})}^q d\tt\le
\Lambda_0<\infty
$$
for some $q>1$, then with $\delta_*$ sufficiently small
\begin{equation}\label{gev2}
|g_{ij}(x,t(p))-g_{ij}(x,t)|\le \Lambda_0^{1/q}
(t(p)-t)^{1-1/q}<\f12 \ep.
\end{equation}

The above argument fails if $k$ verifies (\ref{add}) only. Under the
assumption (\ref{add}), our strategy is to prove directly the
integral on the right of (\ref{gev}) can be small, i.e.
$$\int_t^{t(p)} |k(x,t')| dt'<\frac{1}{2}\ep,\qquad \forall x\in \Sigma $$
by establishing
\begin{equation}\label{2.7.1}
 \sup_{x\in \Sigma} \int_t^{t(p)} |k(x, \tt)|^2 d\tt \le C,
\end{equation}
since
$$
|g_{ij}(x, t(p))-g_{ij}(x, t)|\les \left(\int_t^{t(p)} |k(x, \tt)|^2
d\tt\right)^{1/2} (t(p)-t)^{1/2}\les (t(p)-t)^{1/2}
$$
which implies $|g_{ij}(x, t(p))-g_{ij}(x, t)|<\f12 \ep$  as long as
$\delta_*$ is appropriately chosen.

The major part of the present paper is therefore to establish
(\ref{2.7.1}) under the weaker condition (\ref{add}). To this end,
we will use the Kirchoff parametrix to represent $k$ as
$$
-4\pi n(p) k(p)\cdot J=\int_{\N^-(p, \tau)} \Box k\cdot {\bf A} +
\mbox{other terms},
$$
for any $\delta<i_*(p,t)$, where $J$ is any 2-covariant tensor at
$p$ tangent to $\Sigma_{t(p)}$ and ${\bf A}$ is the $\Sigma$-tangent
tensor defined by
$$
{\bf D}_L {\bf A}_{ij} +\frac{1}{2} \tr \chi {\bf A}_{ij} =0 \,\, \,
\mbox{ on } \N^-(p, \tau), \qquad \lim_{t\rightarrow t(p)} (t(p)-t)
{\bf A}_{ij} =J.
$$
It can be shown that $\|r{\bf A}\|_{L^\infty(\N^-(p, \tau))}\les 1$
together with other estimates on ${\bf A}$, where
$r=\sqrt{(4\pi)^{-1} |S_t|}$ and $|S_t|$ denotes the area of $S_t$.
Thus
$$
n(p)|k(p)|\les \int_{\N^-(p, \tau)} r^{-1} |\Box k| +\mbox{other
terms}.
$$
Next we let $p$ move along an integral curve $\Phi(t)$ of ${\bf
T}$ to get the representations of $k$ at all points on this curve.
Then we can reduce the proof of (\ref{2.7.1}) to showing that
$$
\int_{t(p)-\tau}^{t(p)} \left|\int_{\N^-(\Gamma(t), t-t(p)+\tau)}
r^{-1}|\Box k|+\cdots\right|^2 d t\les 1.
$$
 In view of (\ref{hkk}), we have to employ various estimates of $k$
and $n$ on the null cones, which will be established by delicate
analysis.

This paper is organized as follows. In  Section \ref{3d}, we collect
some preliminary results related  to the CMC foliation, which will
be used frequently in the later sections. In Section \ref{3D}, we
establish various elliptic estimates on the lapse function $n$, in
particular, we show that $n$ can be bounded from below and above by
positive universal constants. In Section \ref{sect4}, we provide the
sketch of the proof of Theorem \ref{thm4}. We will explain how to
use the bootstrap argument to establish (\ref{3.1.5}) and other
related estimates on the null cones. We then show how to use the
estimate (\ref{2.7.1}) to obtain  a system of good local space-time
coordinates which is crucial for completing the proof of Theorem
\ref{thm4}. In order to establish (\ref{2.7.1}), we derive a
tensorial wave equation for $k$ in Section \ref{wave} and provide
the estimate for the so called $k$-flux in Section \ref{flux1} which
will be defined later. In Section 7 we provide some trace estimates
on the surfaces $S_t$. We then use these results in Section 8 to
establish various estimates for $k$, $n$ and $\chi$ on the null
cones. In section 9 we adapt the Kirchoff-Sobolev formula in
\cite{Ksob} to represent the second fundamental form $k$ along the
null cones, through which we give the proof of (\ref{3.1.5}) under
the condition (\ref{add}) and thus complete the proof of
Theorem \ref{thm4}. Finally in Section 9 we complete the proof of Theorem \ref{thm3}.\\

\noindent{\bf Acknowledgement.}  The author would like to thank
Professors Sergiu Klainerman, Michael Anderson and Richard Schoen
for their constant encouragement and support. The author would like
to thank Qinian Jin and Arick Shao for interesting discussions. The
author in particular would like to thank Qinian Jin for improving
the exposition.

\section{\bf Preliminaries}\label{3d}
\setcounter{equation}{0}

For the lapse function $n$, by using the elliptic equation
$-\Delta n+|k|^2 n=1$, it follows easily from the maximum
principle that
\begin{equation}\label{2.5.1}
\frac{1}{\|k\|_{L^\infty(\Sigma_t)}}\le n\le \frac{3}{t^2} \quad
\mbox{on } \Sigma_t.
\end{equation}
Thus, if we knew  that $\|k\|_{L^\infty(\Sigma_t)}$ is uniformly
bounded with respect to $t\in [t_0, t_*)$, then we could get a
positive uniform lower bound on $n$. Unfortunately, we only have the
weaker assumption ({\bf A1}) on $k$, which does not allow
(\ref{2.5.1}) to give a positive uniform lower bound on $n$
directly. In the next section, we will show under the assumption
({\bf A1}) that $C^{-1}\le n\le C$ on $\M_I$ for some universal
constant $C>0$.

For each slice $\Sigma_t$, we use $|\Sigma_t|$ to denote its volume.
Then, by using $\p_t g_{ij}=-2n k_{ij}$ and $\Tr k=t$ on $\Sigma_t$
we have
$$
\frac{d}{d t} \left(|t|^3 |\Sigma_t|\right) =\int_{\Sigma_t}  t^2
\left(t^2 n-3\right) d\mu_g\le 0.
$$
This implies that $|t|^3 |\Sigma_t|$ is decreasing with respect to
$t$. Consequently
\begin{equation}\label{2.5.2}
|\Sigma_t|\le \frac{|t_0|^3}{|t|^3} |\Sigma_{t_0}|\le
\frac{|t_0|^3}{|t_*|^3} |\Sigma_{t_0}|, \quad \forall t_0\le  t\le
t_*.
\end{equation}

\subsection{\bf Bel-Robinson Energy}

We start with a brief review of Bel-Robinson energy, one may consult
\cite{KC} for more details.  Associated to the Weyl tensor $\bR$,
the Bel-Robinson tensor is the full symmetric, traceless tensor
defined by
\begin{equation}\label{emt}
\bold Q[\bR]_{\a\b\ga\delta}=\bR_{\a \lambda
\ga\mu}\bR_{\b}{}^{\lambda}{}_{\delta}{}^\mu+{}^\star
\bR_{\a\lambda\ga\mu}{}^\star \bR_{\b}{}^\lambda{}_\delta{}^\mu.
\end{equation}
Then $\bold Q[\bR](X,Y,X, Y)\ge 0$ whenever $X, Y$ are timelike vectors,
with equality only if $\bR=0$.
Let ${\bf P}_\alpha=\bold Q[\bR]_{\a\b\gamma\delta}\bT^\b \bT^\gamma \bT^\delta$. Since
${\bR}_{\a\b}=0$, a straightforward calculation shows that
\begin{equation}\label{3.2.1}
\bd^\a {\bf P}_\a=-3\pi^{\a\b} \bold Q[\bR]_{\a\b\gamma\delta}
\bT^\gamma \bT^\delta.
\end{equation}
If we introduce the Bel-Robinson energy ${\mathcal Q}(t)$ by
$$
{\mathcal Q}(t):=\int_{\Sigma_t}{\bold Q}[\bR](\bT, \bT, \bT,
\bT) d\mu_{\Sigma_t},
$$
then, by integrating (\ref{3.2.1}) in a slab $\M_J=\cup_{t\in J} \Sigma_t$
with $J=[t_0, t]\subset [t_0, t_*)$, we obtain
\begin{align*}
{\mathcal Q}(t)={\mathcal Q}(t_0)-3\int_{t_0}^t \int_{\Sigma_{t'}} n
\bQ[\bR]_{\a\b00} \pi^{\a\b} d\mu_{\Sigma_{t'}} dt'.
\end{align*}

Let $E$ and $H$ denote the electric and magnetic parts of the curvature tensor $\bR$ defined by
\begin{equation}\label{em1}
E(X, Y)=\bg(\bR(X,\bT)\bT, Y), \qquad
H(X,Y)=\bg({}^\star\bR(X,\bT)\bT,Y)
\end{equation}
with ${}^\star\bR$ the Hodge dual of $\bR$. It is well known that
$E$ and $H$ are traceless symmetric 2-tensors tangent to $\Sigma_t$
with
$$
|\bR|^2=|E|^2+|H|^2,
$$
$$
|\bQ|\le 4(|E|^2+|H|^2)
$$
and
$$
\bQ(\bT,\bT,\bT,\bT)=|E|^2+|H|^2.
$$
Therefore
$$
{\mathcal Q}(t)\le {\mathcal Q}(t_0)+12 \int_{t_0}^t
\|n\pi\|_{L^\infty(\Sigma_{t'})} {\mathcal Q}(t') d t'.
$$
By the Gronwall inequality it follows that
\begin{equation*}
{\mathcal Q}(t) \le {\mathcal Q}(t_0) \exp \left(12\int_{t_0}^t \|n
\pi\|_{L^\infty(\Sigma_{t'})} dt'\right)
\end{equation*}
for all $t\in [t_0, t_*)$. Therefore, in view of the condition ({\bf A1}) we obtain the uniform boundedness
of the Bel-Robinson energy.

\begin{lemma}\label{es1}
Under the condition $({\bf A1})$, there exists a constant $C$
depending only on $\k$ and $t_*$ such that
\begin{equation*}
{\mathcal Q}(t)\le C Q_0^2
\end{equation*}
for all $t\in [t_0, t_*)$, where $Q_0^2:={\mathcal Q}(t_0)$.
\end{lemma}

Consequently we have

\begin{lemma}\label{first}
Let the condition $({\bf A1})$ hold. Then on any CMC
leaf $\Sigma_t\subset \M_*$ there holds
\begin{equation}\label{recal3}
\int_{\Sigma_t} \left(|\nabla k|^2+\frac{1}{4}| k|^4\right)+
\int_{\Sigma_t} |Ric|^2 \les  Q_0^2.
\end{equation}
\end{lemma}

\begin{proof} The inequality on $k$ follows from
\cite[Proposition 8.4]{KR2} and Lemma \ref{es1}. The inequality on
$Ric$ then follows from the identity $R_{ij}-k_{ia}k^{aj}+ \Tr k\,
k_{ij}=E_{ij}$.
\end{proof}

\subsection{\bf Harmonic coordinates}

For any coordinate chart ${\mathcal O}\subset \Sigma_0$ with local coordinates
$x=(x^1, x^2, x^3)$, we denote by $x^0=t, x^1, x^2, x^3$ the
transported coordinates on $I\times {\mathcal O}$ obtained by
transporting along the integral curves of ${\bf T}$.
The following is an immediate consequence of ({\bf
A1}) and (\ref{bg}).

\begin{proposition}\label{comp.01}
Let the assumption $({\bf A1})$ hold. There exists a positive
constant $C_0$ depending only on $\k$ such that, relative to the
induced transported coordinates $x^0=t$, $x^1$, $x^2$, $x^3$ in
$I\times {\mathcal O}$ we have
\begin{equation}\label{comp.02}
C_0^{-1} |\xi|^2\le g_{ij}(t, x)\xi^i \xi^j\le C_0 |\xi|^2.
\end{equation}
\end{proposition}
\begin{proof}
This is \cite[Proposition 2.4]{KR2} which was stated under the
stronger condition (\ref{tn1}), the proof however requires only the
weaker assumption ({\bf A1}).
\end{proof}

This proposition enables us to derive a uniform lower bound on the
volume radius for all the slices $\Sigma_t$. Here, for a
3-dimensional Riemannian manifold $(M, g)$, the volume radius
$r_{vol}(p, \rho)$ at a point $p\in M$ and scales $\le \rho$ is
defined by
\begin{equation*}
r_{vol}(p,\rho)=\inf_{r\le \rho}\frac{|B_r(p)|}{r^3}
\end{equation*}
with $|B_r(p)|$ the volume of $B_r(p)$ relative to metric $g$. The
volume radius $r_{vol}(M,\rho)$ of  $M$ on scales $\le \rho$ is the
infimum of $r_{vol}(p,\rho)$ over all $p\in M$.  Using Proposition
\ref{comp.01}, it has been show in \cite[Proposition 4.4]{KRradius}
that the volume radius $r_{vol}(\Sigma_t, 1)$ of each $\Sigma_t$ on
scales $\le 1$ verifies
$$
r_{vol}(\Sigma_t, 1)\ge v_0
$$
for some constant $v_0>0$ depending only on $\k$.

From the previous subsection we have already obtained, under ({\bf
A1}), that
$$
\|Ric\|_{L^2(\Sigma_t)}\le C \quad \mbox{and}\quad |\Sigma_t|\le
\frac{|t_0|^3}{|t_*|^3}|\Sigma_{t_0}|.
$$
Therefore, Theorem 3.5 in \cite{An97} applies and provides  the
following results on the existence of harmonic coordinates.

\begin{proposition}\label{harcor}
Let the assumption $({\bf A1})$ hold. For any $\ep>0$,
there exists $r_0>0$ depending on $\ep$, $Q_0$,
$\k$, $|\Sigma_0|$ and $t_*$ such that every geodesic ball $B_r(p)\subset
\Sigma_t$ with $r\le r_0$ admits a system of harmonic coordinates
$x=(x^1, x^2, x^3)$ under which
\begin{align}
(1+\ep)^{-1}\delta_{ij}\le  g_{ij} &\le(1+\ep)\delta_{ij}\label{harcor1}\\
r\int_{B_{r}(p)} |\p^2 g_{ij}|^2 d \mu_g & \le \ep.\label{harcor2}
\end{align}
\end{proposition}

We will not use the full strength of this result. The crucial part
in our applications is the existence of a local coordinates $x=(x^1,
x^2, x^3)$ on each $B_{r_0}(p)\subset \Sigma_t$ satisfying
({\ref{harcor1}) with $r_0>0$ depending only on $\ep$, $Q_0$, $\k$,
$|\Sigma_0|$ and $t_*$.

\subsection{\bf Sobolev-type inequalities}

We will give several Sobolev type inequalities under the assumption
({\bf A1}). These inequalities are useful in establishing various
estimates.

\begin{lemma}\label{sob.02}
Let the assumption $({\bf A1})$
hold on $\M_*$. Then for any smooth tensor field $F$ on
$\Sigma_t\subset \M_*$ and any $2\le p\le 6$ there holds
\begin{equation}\label{sob.03}
\|F\|_{L^p(\Sigma_t)}\le C \left(\|\nab
F\|_{L^2(\Sigma_t)}^{3/2-3/p}\|F\|_{L^2(\Sigma_t)}^{3/p-1/2}
+\|F\|_{L^2(\Sigma_t)} \right),
\end{equation}
where $C$ is a constant depending only on $\k$ and
$p$.
\end{lemma}

\begin{proof} This is \cite[Corollary 2.7]{KR2}.
\end{proof}

The following calculus inequality is useful in deriving $L^\infty$
bounds of certain quantities.

\begin{lemma}\label{sobinf}
Let the assumption $({\bf A1})$ hold on $\M_*$. Then for any smooth
tensor field $F$ on $\Sigma_t\subset \M_*$ and $3<p\le 6$ there
holds
\begin{equation*}
\|F\|_{L^\infty(\Sigma_t)}\le C \left(\|\nab^2
F\|_{L^2(\Sigma_t)}^{3/2-3/p} \|\nab F\|_{L^2(\Sigma_t)}^{3/p-1/2}
+\|\nab F\|_{L^2(\Sigma_t)}+\|F\|_{L^2(\Sigma_t)}\right),
\end{equation*}
where $C$ is a constant depending only on $\k$ and $p$.
\end{lemma}

\begin{proof}
By using a partition of unity, the Sobolev embedding $W^{1,
p}({\mathbb R}^3)\hookrightarrow L^\infty({\mathbb R}^3)$ with
$p>3$, and (\ref{comp.02}) in Proposition \ref{comp.01}, it is easy
to derive for any scalar function $f$ on $\Sigma_t$ that
\begin{equation*}
\|f\|_{L^\infty(\Sigma_t)}\le C \left(\|\nab
f\|_{L^p(\Sigma_t)}+\|f\|_{L^p(\Sigma_t)}\right).
\end{equation*}

Now we take $f=|F|^2$ in the above inequality. It yields
\begin{align*}
\|F\|_{L^\infty(\Sigma_t)}^2 &\le
C \left(\|\nab|F|^2\|_{L^p(\Sigma_t)}+\||F|^2\|_{L^p(\Sigma_t)}\right)\\
&\le C\left( \|\nab
F\|_{L^p(\Sigma_t)}+\|F\|_{L^p(\Sigma_t)}\right)\|F\|_{L^\infty(\Sigma_t)}.
\end{align*}
This implies for $p>3$ that
\begin{align*}
\|F\|_{L^\infty(\Sigma_t)}\le C\left( \|\nab F\|_{L^p(\Sigma_t)}
+\|F\|_{L^p(\Sigma_t)}\right).
\end{align*}
The desired inequality follows by applying Lemma \ref{sob.02}
to the term $\|\nab F\|_{L^p(\Sigma_t)}$.
\end{proof}

\section{\bf Elliptic estimates for the lapse function $n$}\label{3D}
\setcounter{equation}{0}

In this section, we establish a series of elliptic estimates on the
lapse function $n$ together with $n^{-1}$ and $\dot n:=\p_t n$ under
the assumption ({\bf A1}). These results will be repeatedly used in
later sections. Throughout this paper we will use $C$ to denote a
universal constant.

\subsection{\bf Estimates on $n$}

\begin{proposition}\label{lapse1}
Let the assumption $({\bf A1})$ hold. Then on every $\Sigma_t\subset
\M_*$ there holds
\begin{align*}
\|\nab^2 n\|_{L^2(\Sigma_t)}+\|\nabla n\|_{L^2(\Sigma_t)}\le C.
\end{align*}
\end{proposition}

\begin{proof}
We multiply the equation $-\Delta n +|k|^2 n=1$ by $n$ and integrate
over $\Sigma_t$ to obtain
$$
\int_{\Sigma_t}\left(|\nabla n|^2 +|k|^2 n^2 \right)
=\int_{\Sigma_t} n.
$$
Since $0<n\le 3/t^2\le 3/t_*^2$ and $|\Sigma_t|\le
|\Sigma_{t_0}||t_0|^3/|t_*|^3$, this immediately gives the desired
bound on $\|\nabla n\|_{L^2(\Sigma_t)}$.

In order to obtain the bound on $\|\nabla^2 n\|_{L^2(\Sigma_t)}$, we
use the B\"ochner identity
\begin{align*}
\int_{\Sigma_t} |\nab^2 n|^2&=\int_{\Sigma_t} \left(|\Delta n|^2-
R_{ij}\nab^i n \nab^j n\right),
\end{align*}
the equation $\Delta n=|k|^2 n-1$, Lemma \ref{first} and the
H\"{o}lder inequality to infer that
\begin{align*}
\|\nabla^2 n\|_{L^2}\les
\|k\|_{L^4}^2+|\Sigma_t|^\f12+\|Ric\|_{L^2}^\f12 \|\nabla
n\|_{L^4}\les 1+\|\nabla n\|_{L^4}.
\end{align*}
With the help of Lemma \ref{first}, we have
$$
\|\nabla^2 n\|_{L^2}\les 1+ \|\nabla^2 n\|_{L^2}^{3/4} \|\nabla
n\|_{L^2}^{1/4}+\|\nabla n\|_{L^2}.
$$
Therefore
$$
\|\nabla^2 n\|_{L^2}\les 1+\|\nabla n\|_{L^2}\les 1
$$
and the proof is complete.
\end{proof}

\begin{proposition}\label{Ba22}
Let the assumption $({\bf A1})$ hold. Then there hold
\begin{align}
\|\nab^3 n\|_{L_t^1 L_x^2(\M_*)}&\le  C \label{lap.sec} \\
\|\nab n\|_{L_t^b L_x^\infty(\M_*)}&\le C \label{lap.inf}
\end{align}
where $1\le b<2$.
\end{proposition}

We will give the proof with the help of the following lemma.

\begin{lemma}\label{Lnew}
Let the assumption $({\bf A1})$ hold. Then for any $1$-form $F$ on
$\Sigma_t\subset \M_*$ we have
\begin{equation}\label{secd}
\|\nab^2 F\|_{L^2(\Sigma_t)}\le C\left(\|\Delta F\|_{L^2(\Sigma_t)}+\|\nab
F\|_{L^2(\Sigma_t)}+\|F\|_{L^2(\Sigma_t)}\right).
\end{equation}
\end{lemma}

\begin{proof}
It is well known that for any $1$-form $F$ on $\Sigma_t$ there holds
the B\"ochner identity
\begin{align}
\int_{\Sigma_t}|\Delta F|^2&=\int_{\Sigma_t} |\nab^2
F|^2-\frac{1}{2}\int_{\Sigma_t} R_{diac} R_{miac} F_d
F_m\nn\\
&\quad\, +\int_{\Sigma_t} R_{ad} \nab_d F_i \nab_a F_i
-\int_{\Sigma_t} R_{idac}\nab_c F_d \nab_a F_i.\label{b1}
\end{align}
Since $\Sigma_t$ is $3$-dimensional, the Riemannian curvature tensor
is completely determined by its Ricci curvature, i.e.
\begin{equation*}
R_{idac}=g_{ia}R_{dc}+g_{dc}R_{ia}-R_{ic}
g_{da}-R_{da}g_{ic}-\frac{1}{2}(g_{ia}g_{dc}-g_{ic}g_{da})R.
\end{equation*}
Thus, we may use (\ref{b1}), the H\"older inequality, Lemma \ref{first},
Lemma \ref{sob.02} and Lemma \ref{sobinf} to obtain the estimate
\begin{align*}
\|\nab^2 F\|_{L^2} &\les \|\Delta F\|_{L^2} +\|Ric\|_{L^2}^{1/2}
\|\nab F\|_{L^4} +\|F\|_{L^\infty}\|Ric\|_{L^2}\\
&\les \|\Delta F\|_{L^2}+\left( \|\nab^2 F\|_{L^2}^{3/4}\|\nab
F\|_{L^2}^{1/4}+\|\nab F\|_{L^2}\right).
\end{align*}
With the help of Young's inequality, the inequality (\ref{secd})
follows immediately.
\end{proof}

\begin{proposition}\label{P1.2.1}
Let the assumption $({\bf A1})$ hold. Then on every $\Sigma_t\subset
\M_*$ there hold
\begin{align}
&\|\nab^3 n\|_{L^2(\Sigma_t)} \le
C\left(\|\nab n\|_{H^1(\Sigma_t)}+\|k\|_{L^\infty(\Sigma_t)}\right),  \label{sob2}\\
&\|\nab n\|_{L^\infty(\Sigma_t)}\le C\left(\|\nab n\|_{H^1(\Sigma_t)}+
\|k\|_{L^\infty(\Sigma_t)}^{3/2-3/p} \|\nab^2 n\|_{L^2(\Sigma_t)}^{3/p-1/2}\right),
\label{sobinf.1}
\end{align}
where $3<p\le 6$.
\end{proposition}

\begin{proof} A simple application of Lemma \ref{Lnew} to $F=\nabla
n$ gives
\begin{equation}\label{new7}
\|\nabla^3 n\|_{L^2}\les \|\Delta \nabla
n\|_{L^2} +\|\nabla^2 n\|_{L^2}+\|\nabla n\|_{L^2}.
\end{equation}
Recall the commutation formula $ \Delta \nabla_i n=\nabla_i \Delta n
+R_{ij} \nabla_j n$ and the equation $-\Delta n+|k|^2 n=1$, we can
estimate
\begin{align*}
\|\Delta \nabla n\|_{L^2} &\les \|k\|_{L^6}^2\|\nabla
n\|_{L^6}+\|k\|_{L^\infty}\|\nabla k\|_{L^2} +\|Ric\|_{L^2} \|\nabla
n\|_{L^\infty}.
\end{align*}
Plugging this into (\ref{new7}), using Lemma \ref{first} and Lemma
\ref{sob.02} gives
$$
\|\nabla^3 n\|_{L^2}\les \|\nabla n\|_{L^\infty} +\|\nabla n\|_{H^1}
+\|k\|_{L^\infty}.
$$
Using Lemma \ref{sobinf} for the term $\|\nab n\|_{L^\infty}$ with
$p=4$, we then obtain
\begin{align*}
\|\nabla^3 n\|_{L^2} &\les  \|\nab^3 n\|_{L^2}^{3/4}\|\nab^2
n\|_{L^2}^{1/4}+ \|\nabla n\|_{H^1}+ \|k\|_{L^\infty}.
\end{align*}
This clearly implies (\ref{sob2}). The inequality (\ref{sobinf.1})
is an immediate consequence of (\ref{sob2}) and Lemma \ref{sobinf}.
\end{proof}

Proposition \ref{Ba22} follows by integrating (\ref{sob2}) and
(\ref{sobinf.1}) in time with the help of $({\bf A1})$ and
Proposition \ref{lapse1}.

\subsection{\bf Estimates on $n^{-1}$}

\begin{proposition}\label{N}
Let the assumption $({\bf A1})$ hold. Then on each $\Sigma_t\subset
\M_*$ there hold
\begin{equation*}
\|\nab^2(n^{-1})\|_{L^2(\Sigma_t)}+\|n^{-1}\|_{L^\infty(\Sigma_t)}\le
C.
\end{equation*}
\end{proposition}

\begin{proof}
We first have from the Bochner identity that
\begin{align}\label{new1}
\int_{\Sigma} |\nabla^2(n^{-1})|^2 &=\int_{\Sigma}
|\Delta(n^{-1})|^2 -\int_{\Sigma} R_{ij} \nabla_i(n^{-1}) \nabla_j
(n^{-1})\nn\\
&\le \|\Delta(n^{-1})\|_{L^2}^2 +\|Ric\|_{L^2}
\|\nabla(n^{-1})\|_{L^4}^2.
\end{align}
Since $-\Delta n+|k|^2 n=1$, we have
\begin{equation}\label{new4}
\Delta(n^{-1})=2 n^{-3} |\nabla n|^2 +n^{-2} -|k|^2 n^{-1}.
\end{equation}
Consequently, it follows from the H\"{o}lder inequality that
\begin{align*}
\|\Delta(n^{-1})\|_{L^2} &\les \|n^{-1} \nabla n\|_{L^4}
\|\nabla(n^{-1})\|_{L^4} +\|k\|_{L^6}^2 \|n^{-1}\|_{L^6}
+\|n^{-1}\|_{L^4}^2.
\end{align*}
Combining this inequality with (\ref{new1}) and using the Sobolev
embedding $H^1(\Sigma)\hookrightarrow L^p(\Sigma)$ with $2\le p\le
6$, which is a consequence of Lemma \ref{sob.02}, we obtain
\begin{align}\label{new3}
\|\nabla^2(n^{-1})\|_{L^2} &\les \|n^{-1} \nabla n\|_{L^4}
\|\nabla(n^{-1})\|_{L^4}
+\left(\|n^{-1}\|_{H^1}+\|k\|_{L^6}^2\right)
\|n^{-1}\|_{H^1} \nn \\
&\quad \, + \|Ric\|_{L^2}^\f12 \|\nabla(n^{-1})\|_{L^4}
\end{align}

We need to estimate $\|n^{-1} \nabla n\|_{L^4}$. To this end, we
multiply the equation $-\Delta n+|k|^2 n=1$ by $n^{-l}$ for some
positive integer $l$ and then integrate by parts over $\Sigma_t$ to
obtain
\begin{equation}\label{new2}
\int_{\Sigma_t} \left(l n^{-l-1} |\nabla n|^2 +n^{-l}\right)
=\int_{\Sigma_t} n^{-l+1} |k|^2.
\end{equation}
Taking $l=7$ gives
$$
\int_{\Sigma_t} n^{-8} |\nabla n|^2 \les \int_{\Sigma_t} n^{-6}
|k|^2\les \|k\|_{L^4}^2 \|n^{-2}\|_{L^6}^3.
$$
Therefore
\begin{align*}
\|n^{-1} \nabla n\|_{L^4} &\le \left(\int_{\Sigma_t} n^{-8} |\nabla
n|^2\right)^{1/8} \left(\int_{\Sigma_t} |\nabla n|^6\right)^{1/8} \les
\|k\|_{L^4}^{1/4} \|n^{-2}\|_{L^6}^{3/8} \|\nabla n\|_{L^6}^{3/4}.
\end{align*}
By Lemma \ref{sob.02} and Proposition \ref{lapse1}, we have
$\|\nabla n\|_{L^6}\le C\left( \|\nabla^2 n\|_{L^2} +\|\nabla
n\|_{L^2}\right) \le C$. By using Lemma \ref{sob.02} and
(\ref{new2}) with $l=5$ we also have
\begin{align*}
\|n^{-2}\|_{L^6} &\les \|n^{-2}\|_{H^1} \les \left(\int_{\Sigma_t}
n^{-4} |k|^2\right)^{1/2} +\|n^{-1}\|_{L^4}^2 \\
&\les \|k\|_{L^6}\|n^{-1}\|_{L^6}^2+\|n^{-1}\|_{L^4}^2\\
&\les \left(1+\|k\|_{L^6}\right) \|n^{-1}\|_{H^1}^2.
\end{align*}
Therefore
$$
\|n^{-1}\nabla n\|_{L^4} \les \left(1+\|k\|_{L^6}^{3/8}\right)
\|k\|_{L^4}^{1/4} \|n^{-1} \|_{H^1}^{3/4}.
$$
Combining this inequality with (\ref{new3}) and using Lemma
\ref{first} to bound $\|k\|_{L^4}$, $\|k\|_{L^6}$ and
$\|Ric\|_{L^2}$, it yields
$$
\|\nabla^2(n^{-1})\|_{L^2}\les \|n^{-1}\|_{H^1}^{3/4}
\|\nabla(n^{-1})\|_{L^4} +(\|n^{-1}\|_{H^1}+1)\|n^{-1}\|_{H^1}
+\|\nabla(n^{-1})\|_{L^4}.
$$
Applying Lemma \ref{sob.02} to the term $\|\nabla(n^{-1})\|_{L^4}$ gives
\begin{align*}
\|\nabla^2(n^{-1})\|_{L^2}&\les \|n^{-1}\|_{H^1}^{3/4}
\left(\|\nabla^2(n^{-1})\|_{L^2}^{3/4}
\|\nabla(n^{-1})\|_{L^2}^{1/4}+\|\nabla(n^{-1})\|_{L^2}\right)
+\|n^{-1}\|_{H^1}^2\\
& \quad\, + \|\nabla^2(n^{-1})\|_{L^2}^{3/4}
\|\nabla(n^{-1})\|_{L^2}^{1/4}+\|n^{-1}\|_{H^1}.
\end{align*}
With the help of Young's inequality, we obtain
\begin{equation}\label{new5}
\|\nabla^2(n^{-1})\|_{L^2}\les \|n^{-1}\|_{H^1}^4 +\|n^{-1}\|_{H^1}.
\end{equation}

In order to estimate $\|n^{-1}\|_{H^1}$,  we use (\ref{new2}) with $l=3$ to obtain
$$
\int_{\Sigma_t}\left(3 |\nabla(n^{-1})|^2 +n^{-3}\right)
=\int_{\Sigma_t} |k|^2 n^{-2}.
$$
It then follows from the H\"{o}lder inequality and Lemma
\ref{sob.02} that
$$
\|\nabla(n^{-1})\|_{L^2}\les \|k\|_{L^4} \|n^{-1}\|_{L^4}\les
\|k\|_{L^4} \left(\|\nabla(n^{-1})\|_{L^2}^{3/4}
\|n^{-1}\|_{L^2}^{1/4}+\|n^{-1}\|_{L^2}\right).
$$
This clearly implies
\begin{equation}\label{new6}
\|\nabla(n^{-1})\|_{L^2}\les \left(\|k\|_{L^4}+\|k\|_{L^4}^4\right)
\|n^{-1}\|_{L^2}\les \|n^{-1}\|_{L^2}.
\end{equation}
The combination of (\ref{new5}) and (\ref{new6}) gives
$$
\|\nabla^2(n^{-1})\|_{L^2}+\|\nabla (n^{-1})\|_{L^2}\les
\|n^{-1}\|_{L^2}^{4}+\|n^{-1}\|_{L^2}.
$$
Note that (\ref{new2}) with $l=2$ gives
$$
\|n^{-1}\|_{L^2}^2\le \int_{\Sigma_t} n^{-1} |k|^2\le \|k\|_{L^4}^2
\|n^{-1}\|_{L^2}.
$$
This implies $\|n^{-1}\|_{L^2}\le \|k\|_{L^4}^2\le C$. We therefore
obtain $\|n^{-1}\|_{H^2}\le C$. With the help of Lemma \ref{sobinf}
the estimate $\|n^{-1}\|_{L^\infty}\le C$ follows immediately.
\end{proof}

\subsection{\bf Derivative estimates about $\dot{n}$}

In this subsection we will give various estimates on the
derivative $\dot n:=\p_t n$. We start with deriving an elliptic
equation for $\dot{n}$. By straightforward calculation we
have
\begin{equation}\label{1.28.1}
 \Delta \dot{n}= -\dg^{ij} \nab_i\nab_j n + \p_t (\Delta n)
+ g^{ij} \dot{\Gamma}_{ij}^a \nab_a n
\end{equation}
Recall that
$$
\dot{\Gamma}_{ij}^a=\frac{1}{2} g^{ab}(\nab_i \dg_{jb}+\nab_j
\dg_{ib}-\nab_b \dg_{ij})\nn.
$$
From (\ref{bg}), (\ref{3.1.1}) and the fact $\Tr k=t$ it then
follows
$$
g^{ij} \dot{\Gamma}_{ij}^a \nab_a n=-2 k_i^a \nabla^i n \nabla_a n +
\Tr k |\nabla n|^2.
$$
Plugging this identity into (\ref{1.28.1}) and using
$\dot{g}^{ij}=2n k^{ij}$ and $\Delta n=|k|^2 n-1$ we obtain
$$
\Delta \dot{n}=-2n k^{ij} \nabla_i\nabla_j n+|k|^2 \dot{n}
+\p_t(|k|^2) n -2k_i^a\nabla^i n\nabla_a n +\Tr k|\nabla n|^2
$$
We may use the equations (\ref{bg}) and (\ref{intro04}) to derive
\begin{align*}
\p_t(|k|^2)&=-2k^{ij} \nab_i\nab_j n + 2n R_{ij} k^{ij} +2n |k|^2
\Tr k.
\end{align*}
Consequently, we obtain
\begin{align}\label{1.28.2}
\Delta\dot{n}&=-4n k^{ij} \nabla_i\nabla_j n+|k|^2 \dot{n}-2
k_i^a\nabla^i n \nabla_a n +\Tr k |\nabla n|^2\nn\\
&\quad\,  +2n R_{ij} k^{ij} +2n |k|^2 \Tr k.
\end{align}

Now we multiply the equation (\ref{1.28.2}) by $\dot{n}$ and
integrate over $\Sigma_t$, by using the boundedness of $n$ and the
H\"older inequality we obtain
\begin{align*}
\int_{\Sigma_t}& \left(|\nabla \dot{n}|^2 +|k|^2
|\dot{n}|^2\right)\\
&\les \int_{\Sigma_t} \left(|\dot{n}| |k||\nabla^2 n| +|\dot{n}|
|k||\nabla n|^2 +|\dot n||Ric||k| +|\dot n||k|^3\right)\\
&\le \left(\|\nabla^2 n \|_{L^2} +\|\nabla n\|_{L^4}^2
+\|Ric\|_{L^2}\right)\|k\|_{L^4} \|\dot{n}\|_{L^4}
+\|k\|_{L^6}^3\|\dot{n}\|_{L^2}.
\end{align*}
Using the bounds derived in Lemma \ref{first} and Proposition
\ref{lapse1} together with the Sobolev embedding we have
\begin{align*}
\int_{\Sigma_t} \left(|\nabla \dot n|^2 +|k|^2 |\dot n|^2 \right)
\les \|\dot n\|_{L^4}+\|\dot n\|_{L^2}\les \|\nabla \dot n\|_{L^2}
+\|\dot n\|_{L^2}.
\end{align*}
Recall that $|k|^2=|\hk|^2 + t^2/3$ and $|t|\ge |t_*|>0$. Therefore
$$
\|\nabla \dot n\|_{L^2}^2+\|\dot n\|_{L^2}^2\les \|\nabla \dot
n\|_{L^2}+\|\dot n\|_{L^2}.
$$
We therefore obtain

\begin{lemma}\label{L2.1.1}
Let the assumption $({\bf A1})$ hold. Then for each $\Sigma_t\subset
\M_*$, there holds
\begin{equation}\label{phdin}
\|\nab \dot n\|_{L^2(\Sigma_t)}+\|\dot n\|_{L^2(\Sigma_t)}\le C.
\end{equation}
\end{lemma}

Now we are ready to give some estimates on the mixed norms of
$\dot n$.

\begin{proposition}\label{phit01}
Let the assumption $({\bf A1})$ hold. Let $\dot{n}=\p_t n$. Then
there hold
\begin{align*}
\|\nab^2\dot{n}\|_{L_t^1 L_x^2(\M_*)}\le C \qquad \mbox{and}\qquad
\|\dot{n}\|_{L_t^{b} L_x^\infty(\M_*)}\le C
\end{align*}
for any $1\le b<2$.
\end{proposition}

\begin{proof}
In view of the assumption ({\bf A1}), it suffice to establish on
every $\Sigma_t$ the inequalities
\begin{equation}\label{phd2}
\|\nab^2 \dot n \|_{L^2(\Sigma_t)}\les
C\left(\|k\|_{L^\infty(\Sigma_t)}+1\right),
\end{equation}
and
\begin{equation}\label{phd3}
\|\dot n\|_{L^\infty(\Sigma_t)}\le
C\left(\|k\|_{L^\infty(\Sigma_t)}^{3/2-3/p}+1\right)
\end{equation}
for any $3<p\le 6$.

By the B\"ochner identity, we have
\begin{equation*}
\|\nab^2 \dot n\|_{L^2}^2\le\|\Delta \dot n \|_{L^2}^2 +\|Ric
\|_{L^2}\|\nabla \dot n\|_{L^4}^2.
\end{equation*}
By using $\|Ric\|_{L_x^2}\les 1$ and applying Lemma \ref{sob.02} to
$\|\nabla \dot n\|_{L^4}$ we obtain
$$
\|\nabla^2 \dot n\|_{L^2}\les \|\Delta \dot n\|_{L^2} +\|\nabla \dot
n\|_{L^2}^{3/4} \|\nabla \dot n\|_{L^2}^{1/4}+\|\nabla \dot
n\|_{L^2}.
$$
In view of Young's inequality and (\ref{phdin}), it follows
\begin{equation}\label{phit2d}
\|\nab^2 \dot n\|_{L^2}\les \|\Delta\dot n\|_{L^2}+ 1.
\end{equation}
From the equation (\ref{1.28.2}) it follows that
$$
\|\Delta \dot n\|_{L^2}\les\|k\|_{L^\infty}\|\left(\|\nab^2
n\|_{L^2} +\|Ric\|_{L^2}\right) +\|k\|_{L^6}^2 \|\dot n\|_{L^6}+
\|\nabla n\|_{L^6}^2 \|k\|_{L^6} +\|k\|_{L^6}^3.
$$
With the help of the estimates derived in Lemma \ref{first},
Proposition \ref{lapse1} and (\ref{phdin}) together with the Sobolev
embedding we have $\|\Delta \dot n\|_{L^2}\les \|k\|_{L^\infty}
+1$. Therefore $\|\nabla^2 \dot n\|_{L^2}\les
\|k\|_{L^\infty}+1$ which is exactly (\ref{phd2}). The inequality
(\ref{phd3}) immediately follows from Lemma \ref{sobinf},
(\ref{phd2}) and (\ref{phdin}).
\end{proof}

\section{\bf Null radius of injectivity: proof of main theorem II} \label{sect4}
\setcounter{equation}{0}

In this section we will give the sketch of the proof of Theorem
\ref{thm4}. The complete proof is rather involved and requires a
delicate bootstrap argument. For any $t_0<t_1<t_*$ we consider the slab
$\M_I=\cup_{t\in I}\Sigma_t$ with $I=[t_0, t_1]$. We set, for each $p\in \M_I$,
$$
\tilde{i}_*(p,t)=\left\{\begin{array}{lll}
+\infty, & \mbox{~~~ if } i_*(p,t)> t(p)-t_0,\\
i_*(p,t), &\mbox{~~~ otherwise}
\end{array}\right.
$$
and define
\begin{equation}\label{3.8.4}
i_*:=\min\{\tilde{i}_*(p,t): p\in \M_I\}.
\end{equation}
Due to the compactness of $\M_I$, we have $i_*>0$. In order to complete
the proof of Theorem \ref{thm4}, it suffices to show that $i_* > \delta_*$
for some universal constant $\delta_*>0$.

We will use the following result concerning the lower bound on the null radius of
injectivity of a globally hyperbolic space-time which has
essentially been proved in \cite{KRradius}.

\begin{theorem}\label{KR08}
Let $C^{-1}\le n\le C$ on $\M_I$ for some constant $C>0$. Then there
exists a small constant $\ep>0$ depending only on $C$ such that if,
for some constant $\delta_*>0$, the following three conditions hold
for all $p\in \M_I$:
\begin{enumerate}
 \item[{\bf C1.}] the null radius of conjugacy satisfies
 $$
 s_*(p, t)> \min\{i_*, \delta_*\};
 $$

 \item[{\bf C2.}] for each $t$ satisfying
 $$
 0\le t(p)-t\le
 \min\{i_*,  \delta_*\},
 $$
 the metric $\gamma_t$ on ${\mathbb
 S}^2$, obtained by restricting the metric $g$ on $\Sigma_t$ to
 $S_t:=\N^{-}(p)\cap \Sigma_t$ and then pulling it back to ${\mathbb S}^2$ by the
 exponential map ${\mathcal G}(t, \cdot)$, verifies
$$
|\gamma_t(X, X)-\stackrel{\circ}\ga (X, X)|< \epsilon \stackrel{\circ}\ga(X, X), \quad
\forall X\in T{\mathbb S}^2,
$$
where $\stackrel{\circ}\ga$ is the standard metric on ${\mathbb
S}^2$;

\item[{\bf C3.}] On ${\mathcal U}_p:=I_p\times
B_{\delta_*}(p)$ with $I_p:=[t(p)-\min\{i_*, \delta_*\}, t(p)]$
and $B_{\delta_*}(p)\subset \Sigma_{t(p)}$ a geodesic ball,
there is a system of coordinates $x^\alpha$ with $x^0=t$ relative to
which the metric ${\bf g}$ is close to the Minkowski metric ${\bf
m}_{\alpha\beta}=-n(p) dt^2+\delta_{ij} dx^i dx^j$ in the sense
that
$$
|n-n(p)|+ |g_{ij}-\delta_{ij}|<\epsilon \quad
\mbox{on } {\mathcal U}_p,
$$
\end{enumerate}
then there holds $i_*> \delta_*$, i.e. the null radius of
injectivity verifies
$$
i_*(p,t)> \min\{\delta_*, t(p)-t_0\}
$$
for all $p\in \M_I$.
\end{theorem}

Let us briefly outline the idea of proof. Assume that $i_*\le
\delta_*$. Let $p_0\in \M_I$ be a point such that $i_*(p_0,t)=i_*\le
t(p)-t_0$. By {\bf C1} we have $s_*(p_0,t)>i_*(p_0,t)=l_*(p_0,t)$.
Thus there exist two distinct past null geodesics $\gamma_1$ and
$\gamma_2$ initiating at $p_0$ intersect at a point $q_0$ with
$t(q_0)=t(p_0)-i_*$. According to the definition of $i_*$ and
\cite[Lemma 3.1]{KRradius} $\gamma_1$ and $\gamma_2$ are opposite at
both $p_0$ and $q_0$. On the other hand, under the conditions {\bf
C2} and {\bf C3}, Lemma 3.2 and Lemma 3.3 in \cite{KRradius} imply that
such two null geodesics can not intersect in the time slab
$[t(p_0)-i_*, t(p_0)]$.

Theorem \ref{KR08} provides a general framework to estimate the null
radius of injectivity from below. Under the condition (\ref{tn1}),
in \cite{KRradius} Klainerman and Rodnianski showed that the
conditions {\bf C1}--{\bf C3} hold with a universal constant
$\delta_*>0$; thus they derived a universal lower
bound on the null radius of injectivity.

In the following we will describe how to verify the conditions
{\bf C1}--{\bf C3} under the assumption ({\bf A1}). To this end, for each $p\in \M_I$
consider the past null cone $\N^{-}(p)$, let $s$ be its affine parameter and
let $S_t=\N^{-}(p)\cap \Sigma_t$. Then $S_t$ is diffeomorphic to
${\Bbb S}^2$ for each $t$ satisfying $t(p)-i_*(p, t)<t<t(p)$.
Let $\gamma$ be the restriction of $\bg$ to $S_t$ and
let $|S_t|$ be the corresponding area. The radius of $S_t$ is defined to be
\begin{equation}\label{3.9.1}
r:=\sqrt{(4\pi)^{-1} |S_t|}
\end{equation}
which is a function of $t$ only.

On $\N^{-}(p, \tau)\setminus\{p\}$ with $\tau<i_*(p,
t)$ we can define a conjugate null vector $\underline{L}$ with
$\bg(L, \underline{L})=-2$ and such that $\underline{L}$ is
orthogonal to the leafs $S_t$. In addition we can choose
$(e_A)_{A=1,2}$ tangent to $S_t$ such that $(e_A)_{A=1,2}$,
$e_3=\Lb$, $e_4=L$ form a null frame, i.e.
$$
\bg(L,\Lb)=-2, \quad \bg(L,L)=\bg(\Lb,\Lb)=\bg(L, e_A)=\bg(\Lb, e_A)=0, \quad \bg(e_A, e_B)=\delta_{AB}.
$$
The null second fundamental forms $\chi$, $\underline{\chi}$,
the torsion $\zeta$ and the Ricci coefficient $\underline{\zeta}$ of the foliation $S_t$
are then defined as follows
$$
\chi_{AB}=\bg(\bd_A L, e_B), \qquad \underline{\chi}_{AB}=\bg(\bd_A
\underline{L}, e_B),
$$
$$
\zeta_A=\frac{1}{2} \bg(\bd_A L, \underline{L}), \qquad
\underline{\zeta}_A=\frac{1}{2} \bg(e_A, \bd_L \underline{L}).
$$
In addition we define
$$
\tr \chi=\gamma^{AB} \chi_{AB}, \qquad
\hat{\chi}_{AB}=\chi_{AB}-\frac{1}{2} \tr \chi \, \gamma_{AB}.
$$
We can define $\tr \chib$ and $\hat{\chib}$ similarly.

We introduce the null lapse function
$$
a^{-1}:=\bg(L, \bT).
$$
Then $a>0$ and $a(p)=1$. It is easy to see that
$$
L=-a^{-1} (\bT +N), \qquad \Lb=-a (\bT-N),
$$
where $N$ denotes the unit inward normal to $S_t$ in $\Sigma_t$. We
also introduce the function
$$
\nu:=-n^{-1}\nabla_N n +k_{NN}
$$
which is relevant to the estimate on $a$.

For any $S_t$-tangent tensor field $F$ we define the norm
$\|F\|_{L_\omega^\infty L_t^2(\N^-(p,\tau))}$ by
\begin{equation*}
\|F\|_{L_\omega^\infty L_t^2(\N^-(p,\tau))}^2:=\sup_{\omega\in {\Bbb
S}^2} \int_{t(p)-\tau}^{t(p)} |F|^2 na dt:=\sup_{\omega\in {\Bbb
S}^2} \int_{\Gamma_\omega} |F|^2 na dt,
\end{equation*}
where $\Gamma_\omega$ denotes the portion of a past null geodesic
initiating from $p$ contained in $\N^-(p,\tau)$.

 The following result is sufficient to prove the conditions
{\bf C1}--{\bf C3} in Theorem \ref{KR08}.

\begin{theorem}\label{bts.1}
Let the assumption $({\bf A1})$ hold. Then there exist universal
constants $\delta_*>0$ and $C_*>0$ such that for any $p\in \M_I$
there hold
\begin{equation}\label{MB}
\int_{t(p)-\tau}^{t(p)} |k(\Phi(t))|^2 dt \le C_*
\end{equation}
with $\Phi$ the integral curve of ${\bf T}$ through $p$, and
\begin{align}\label{MB1}
|a-1|\le \frac{1}{2}, \quad \left|\mbox{tr} \chi
-\frac{2}{s}\right|\le C_*, \quad \|\hat{\chi}\|_{L_\omega^\infty
L_t^2(\N^{-}(p, \tau))}^2 \le C_*
\end{align}
on any null cones $\N^{-}(p, \tau)$, where $\tau:=\min\{i_*,
\delta_*\}$.
\end{theorem}

In fact, the estimate on $\tr \chi$ in (\ref{MB1}) implies the
condition {\bf C1}, see \cite{HE,CH09}.
Next we will show that the
estimates in (\ref{MB1}) imply the condition {\bf C2}. To see this,
we recall that $\frac{d s}{d t}=-na$ and $\frac{d}{d s} \ga_{AB}=2
\chi_{AB}$. Then
\begin{align*}
\frac{d}{dt}(s^{-2} \ga_{AB})&=-na \left(-2s^{-3} \ga_{AB}+2 s^{-2}
\chi_{AB}\right)
\end{align*}
Let $X\in T {\mathbb S}^2$ be any vector field.
We integrate the above equation along any null geodesic and note that
$\lim_{t\rightarrow t(p)^{-}} s(t)^{-2}
\ga(t)=\stackrel{\circ}\ga$,( see \cite{Qwang}), it follows that
\begin{align*}
\left|s(t)^{-2} \ga(X,X)-\stackrel{\circ}\ga(X,X)\right|
&\le \int_t^{t(p)} \left(2|\chih| +\left|\tr \chi -\frac{2}{s(t')}\right|\right)
s(t')^{-2} \gamma(X,X) n a d t'
\end{align*}
Let
$$
\Theta:=2|\chih| +\left|\tr \chi -\frac{2}{s}\right|.
$$
We then have
\begin{align*}
\left|s(t)^{-2} \ga(X,X)-\stackrel{\circ}\ga(X,X)\right| &\le \int_t^{t(p)} \Theta
\left|s(t')^{-2} \ga(X,X)-\stackrel{\circ}\ga(X,X)\right| na dt'\\
&\quad\, + \stackrel{\circ}\ga(X,X) \int_t^{t(p)} \Theta(t') na d t'.
\end{align*}
Therefore, it follows from the Gronwall inequality that
\begin{align*}
\left|s(t)^{-2} \ga(X,X)-\stackrel{\circ}\ga(X,X)\right|
&\le  \stackrel{\circ}\ga(X,X)  \int_t^{t(p)} \Theta na dt' \exp\left(\int_t^{t(p)}
\Theta(t') n a dt'\right).
\end{align*}
Since $0<n\le 3/t_*^2$, the estimate (\ref{MB1}) in Theorem \ref{bts.1} implies
$$
\int_t^{t(p)} \Theta n a dt'\le C\left((t(p)-t)^{1/2}+(t(p)-t)\right)
\le C(t(p)-t)^{1/2}
$$
and consequently
\begin{equation}\label{3.10.1}
\left|s^{-2} \ga(X,X)-\stackrel{\circ}\ga(X,X)\right| \le
C(t(p)-t)^{1/2}  \stackrel{\circ}\ga(X,X)
\end{equation}
for all $t(p)-\min\{i_*, \delta_*\}\le t< t(p)$, where $C$ is a universal constant.
The condition {\bf C2} is thus verified.

The verification of the condition {\bf C3},
using the estimate (\ref{MB}), is given in the following result.

\begin{lemma}\label{cltmin1}
Let the assumption $({\bf A1})$ hold. For any $\ep>0$, there exists
a constant $\delta_*>0$ depending only on $Q_0$, $\k$, $t_*$ and
$\ep$ such that for every point $p\in \M_I$ there exists on
${\mathcal U}_p:=I_p \times B_{\delta_*}(p)$ with
$I_p=[t(p)-\min\{i_*,\delta_*\}, t(p)]$ a system of
transported coordinates $t, x=(x^1, x^2, x^3)$ relative to which
${\bf g}$ is close to the Minkowski metric $\bm(p)=-n(p)^2 d
t^2+\delta_{ij} d x^i d x^j$, in the sense that
\begin{align}\label{3.8.1}
|g_{ij}-\delta_{ij}|<\ep\qquad \mbox{and}\qquad |n-n(p)|<\ep.
\end{align}
\end{lemma}

\begin{proof}
It follows from Proposition \ref{harcor} that there exists a
constant $\delta_0>0$ depending only $\k$, $Q_0$, $t_*$ and $\ep$
such that every geodesic ball $B_{\delta_0}(p)\subset \Sigma_{t(p)}$
admits a system of harmonic coordinates $x=(x^1, x^2, x^3)$ under
which
\begin{equation}\label{ep0}
(1+\ep/2)^{-1}\delta_{ij}\le g_{ij} \le (1+\ep/2)\delta_{ij}.
\end{equation}
Under the transported coordinates $t, x=(x^1, x^2, x^3)$, let
$p=(t(p), 0)$ and let $q=(t, x)$ be an arbitrary point in
$I_p\times B_{\delta_*}(p)$ with $I_p=[t(p)-\min\{i_*,
\delta_*\}, t(p)]$, where $0<\delta_*\le \delta_0$ is a constant to
be determined. By using the equation $\p_t g_{ij}=-2n k_{ij}$ we
have
\begin{align*}
|g_{ij}(t,x)-g_{ij}(t(p), x)|=\left|\int_t^{t(p)} \p_t g_{ij}(t', x)
d t' \right|=2\int_t^{t(p)} n |k| d\tt.
\end{align*}
Using the bound $0<n\le 3/t_*$, the H\"older inequality and
the estimate (\ref{MB}) in Theorem \ref{bts.1}, it follows for some
universal constant $C_1>0$ that
\begin{align*}
|g_{ij}(t,x)-g_{ij}(t(p),x)|\le C_1 (t(p)-t)^{1/2} \le C_1\delta_*^{1/2}.
\end{align*}
In view of (\ref{ep0}), we thus obtain
\begin{equation}\label{comp3}
|g_{ij}(t,x)-\delta_{ij}|\le |g_{ij}(t,x)-g_{ij}(t(p),x)|+
|g_{ij}(t(p),x)-\delta_{ij}| \le C_1 \delta_*^{1/2}
+\frac{\epsilon}{2},
\end{equation}
which gives the first inequality in (\ref{3.8.1}) by letting
$C_1\delta_*^{1/2}<\epsilon/2$.

Next we prove the second inequality in (\ref{3.8.1}). From
Proposition \ref{phit01} we have
$$
|n(t,x)-n(t(p),x)|\le \int_t^{t(p)} |\dot{n}(t', x)| d t'\le
(t(p)-t)^{1/4} \|\dot{n}\|_{L_t^{4/3} L_x^\infty}\le C_2
\delta_*^{1/4},
$$
while by employing  Morrey's estimate, Lemma \ref{sob.02}
and Proposition \ref{lapse1} we have
\begin{align*}
|n(t(p),x)-n(t(p), 0)|&\le C_2 \delta_*^{1/4} \|\nab
n\|_{L^4(\Sigma_{t(p)})} \\
&\le C_2 \delta_*^{1/4} \left(\|\nab^2 n\|_{L^2}^{3/4}\|\nab
n\|_{L^2}^{1/4}+\|\nab n\|_{L_x^2}\right)\\
& \le   C_2 \delta_*^{1/4},
\end{align*}
where $C_2>0$ is a universal constant. Therefore
$$
|n(t,x)-n(p)|\le 2 C_2 \delta_*^{1/4}
$$
which implies the second inequality in (\ref{MB}) by further
letting $2 C_2 \delta_*^{1/4}<\epsilon$.
\end{proof}

The proof of Theorem \ref{bts.1} is based on a delicate bootstrap
argument. We first fix some notations and terminology. Related to
the deformation tensor $\pi_{\a\b}$ of $\bT$, we introduce the
$\Sigma_t$-tangent tensor $h_\a^\mu h_\b^\nu \pi_{\mu\nu}$, where
$$
h_\a^\b=\delta_\a^\b +{\bf T}_\a {\bf T}^\b
$$
denotes the projection tensor. It is easy to see that
$k_{ij}=h_i^\mu h_j^\nu \pi_{\mu\nu}$ and thus this
tensor is an extension of $k$. We will denote it by the same
notation $k$, i.e.
\begin{equation}\label{defk}
k_{\a\b}=h_\a^\mu h_\b^\nu \pi_{\mu\nu}
\end{equation}
Note that $k_{0\a}=k_{\a0}=0$.

Corresponding to the null vector $L$, let $\nabla_L k$  be the
$\Sigma_t$-tangent tensor defined by
$$
\nabla_L k_{ij}:=h_i^\a h_j^\b \bd_L k_{\a\b}
$$
and let $$
|\nabla_L k|^2=g^{ii'}g^{jj'} \nabla_L k_{ij} \nabla_L k_{i'j'}.
$$
We also introduce $\sn k$ by $\sn_A k_{ij}:=\nabla_A k_{ij}$ and set
$$
|\sn k|^2= \gamma^{AB} g^{ii'}g^{jj'} \nabla_A k_{ij} \nabla_B k_{i'j'}.
$$
Corresponding to the second fundamental form $k$,  then, for each
$p\in \M_I$, we introduce on the null cone $\N^{-}(p, \tau)$ the
$k$-flux
\begin{equation}\label{dfk}
\F[k](p,\tau)=\int_{\N^-(p,\tau)} \left(|\sn k|^2 +|\nab_L
k|^2\right),
\end{equation}
where, for each function $f$ and $\tau <i_*(p,t)$,
\begin{equation*}
\int_{\N^-(p,\tau)} f:=\int_{t(p)-\tau}^{t(p)} \int_{S_t} f na
d\mu_\gamma dt.
\end{equation*}

Corresponding to the time foliation, we recall the null components of the
Riemannian curvature tensor ${\bR}$ as follows
\begin{eqnarray}
\alpha_{AB}=\bR(L, e_A, L, e_B),&&\qquad  \beta_A=\frac{1}{2} \bR(e_A, L,
\underline{L}, L),\nonumber\\
\rho=\frac{1}{4} \bR(\underline{L},L, \underline{L},L),&&\qquad
\sigma=\frac{1}{4}{{}^\star \bR}(\underline{L}, L, \underline{L}, L),\label{f14}\\
\underline{\beta}_A=\frac{1}{2} \bR(e_A, \underline{L},
\underline{L},L),&& \qquad \underline{\a}_{A B}=\bR(\underline{L}, e_A,
\underline{L}, e_B).\nonumber
\end{eqnarray}
The corresponding curvature flux $\R(p,\tau)$ on the null cone
$\N^{-}(p, \tau)$ is given by
\begin{equation*}
\R(p,\tau)=\int_{t(p)-\tau}^{t(p)} \int_{S_t}
\left(|\a|^2+|\b|^2+|\rho|^2+|\sigma|^2+|\udb|^2\right) n a d
\mu_\gamma dt.
\end{equation*}

The following result says that once the null lapse $a$ is well controlled, then
the $k$-flux and the curvature flux can be bounded by a universal constant.

\begin{theorem}\label{flux.01}
Let the condition $({\bf A1})$ hold. Then there exists a universal
constant $C_*\ge 1$ such that for all $p\in \M_I$ if $|a-1|\le 1/2$
on $\N^{-}(p, \tau)$ for some $0<\tau\le i_*$ then there holds
$$
\R(p, \tau)+\F[k](p, \tau)\le C_*.
$$
\end{theorem}

We will prove Theorem \ref{flux.01} in Section \ref{flux1}. This
result requires $1/2\le a\le 3/2$ on $\N^{-}(p, \tau)$ which is
obvious for small $\tau>0$ since $a(p)=1$. In order for the above
result to be applicable, we must show that there is a universal
constant $\delta_*>0$ such that the same bound on $a$ holds with
$\tau:=\min\{i_*, \delta_*\}$, and so does the same bound on
$\R(p,\tau)+\F[k](p,\tau)$. We will use a bootstrap argument to
achieve this together with various estimates on $\tr \chi$, $\chih$
and $\nu$. That is, we will make the following bootstrap assumptions
\begin{align*}
|a-1|&\le \frac{1}{2}, \tag {\bf BA1}\\
\left|\tr\chi-\frac{2}{s}\right| &\le \ee, \tag {\bf BA2}\\
\|\chih\|_{L_\omega^\infty L_t^2(\N^{-}(p,\tau))}^2 &\le \ee, \tag {\bf BA3}\\
\|\nu\|_{L_\omega^\infty L_t^2(\N^{-}(p, \tau))}^2  &\le \ee, \tag {\bf BA4}
\end{align*}
on the null cone $\N^{-}(p, \tau)$ for all $p\in \M_I$, where
$0<\tau\le i_*$ and $\ee\ge 1$ are two numbers satisfying $\ee \tau\le 1$.
Due to the continuity of the quantities involved
and the compactness of $\M_I$, the bootstrap assumptions ({\bf
BA1})--({\bf BA4}) hold automatically for sufficiently small
$\tau>0$. Our goal is to show that we can choose universal constants
$\ee\ge 1$ and $\delta_*>0$ such that ({\bf BA1})--({\bf BA4}) hold with
$\tau=\min\{i_*, \delta_*\}$. We will achieve this by
showing that the estimates in ({\bf BA1})--({\bf BA4}) can be
improved.

We will first derive various intermediate consequences of the
bootstrap assumptions. In particular, we will derive the estimate on
the important quantity $\N_1[\slashed{\pi}]$ which is defined as
follows. For any $S_t$ tangent tensor field $F$ defined on the null
cone $\N^{-}(p, \tau)$, the Sobolev norm $\N_1[F](p,\tau)$ is
defined by
\begin{equation}\label{n1sob}
\N_1[F](p, \tau):=\|r^{-1}F\|_{L^2(\N^{-}(p,\tau))}+\|\nab_L
F\|_{L^2(\N^{-}(p,\tau))} +\|\sn F\|_{L^2(\N^{-}(p,\tau))}.
\end{equation}
Recall that the components of the deformation tensor $\pi$ of $\bT$
under transported coordinates are given by $\pi_{00}=0$,
$\pi_{0i}=-n^{-1} \nabla_i n$ and $\pi_{ij}=k_{ij}$. Let us denote
by $\lambda=-\frac{1}{3} \Tr k=-\frac{1}{3} t$ and $\hk$ the
traceless part of $k$. We decompose $\hk$ on each $S_t$ by
introducing components
\begin{equation}\label{compo}
\e_{AB}=\hk_{AB},\quad \quad \ep_A=\hk_{AN},\quad\quad
{\mathfrak{\delta}}= \hk_{NN}
\end{equation}
where $(e_A)_{A=1,2}$ is an  orthonormal frame on $S_t$ and $N$ is the inward unit normal of
$S_t$ in $\Sigma_t$. Let $\eh_{AB}$ be the traceless part of $\e$. Since
$\delta^{AB} \e_{AB}=-\delta$, we have
\begin{equation*}
\eh_{AB}=\e_{AB}+\frac{1}{2}\delta_{AB}\delta.
\end{equation*}
We will denote by  $\slashed{\hk}$, $\sn\slashed{\hk}$ and
$\sl{\pi_0}$ the collections
$$
\slashed{\hk}=(\delta, \ep, \eh), \qquad \sn\slashed{\hk}=(\sn
\delta, \sn \ep, \sn \eh),\qquad \sl{\pi_0}=(\sn\log n, \nab_N \log
n)
$$
respectively. We then define $\sl{\pi}$ to be the collection
\begin{equation}\label{psl}
\sl{\pi}=(\sl{\hk}, \sl{\pi_0}, \lambda).
\end{equation}
We define $\N_1[\sl{\pi}](p, \tau)$ according to (\ref{n1sob}) with
$F$ replaced by $\sl{\pi}$.

With the help of the bound on $k$-flux given in Theorem
\ref{flux.01} and various estimates on the lapse $n$ given in
Section 3, we will show that $\N_1[\slashed \pi](p,\tau)$ can be
bounded in a suitable way under ({\bf A1}) and the bootstrap
assumptions.

\begin{theorem}\label{d2laps}
Let $({\bf A1})$ hold. Then there exists a universal constant
$C$ such that under the bootstrap assumptions $({\bf BA1})$--$({\bf
BA3})$ with $\ee \tau\le 1$ there holds
\begin{equation}\label{piflux}
\N_1[\slashed{\pi}](p,\tau)\le C
\end{equation}
for all $p\in \M_I$.
\end{theorem}

We will prove Theorem \ref{d2laps} in Section 8. From Theorem \ref{flux.01} and
Theorem \ref{d2laps} it follows that
\begin{equation}\label{3.10.5}
\R(p,\tau) +\N_1[\sl{\pi}](p,\tau)\le C_0,
\end{equation}
where $C_0\ge 1$ is a universal constant.

With the help of (\ref{3.10.5}), we can establish the following result which
enables us to improve the estimates in the bootstrap assumptions.

\begin{theorem}\label{cltmin6}
There exist two universal constants $\delta_0>0$ and $C_1\ge 1$ such that,
under the bootstrap assumptions $({\bf BA1})$--$({\bf BA4})$ with $\ee \tau\le 1$,
if $\tau< \min\{i_*, \delta_0\}$ then there hold
\begin{align}
|a-1|&\le C_1  \tau^{1/2},\label{3.7.1}\\
\left|\tr \chi -\frac{2}{s}\right|&\le C_1,\label{3.7.2}\\
\|\hat{\chi}\|_{L_\omega^\infty L_t^2(\N^{-}(p, \tau))}^2 &\le C_1, \label{3.7.3}\\
\|\nu\|_{L_\omega^\infty L_t^2(\N^{-}(p,\tau))}^2 &\le C_1
\label{3.7.4}
\end{align}
on the null cones $\N^{-}(p, \tau)$ for all $p\in \M_I$.
\end{theorem}

The significance of Theorem \ref{cltmin6} lies in that it allows us
to choose $\ee\ge 1$ and $\delta_*>0$ universal such that ({\bf
BA1})--({\bf BA4}) hold on $\N^{-}(p,\tau)$ with $\tau=\min\{i_*,
\delta_*\}$. To see this, we choose $\ee$ and $\delta_*$ in the way
that
\begin{equation}\label{3.10.10}
\ee:=2 C_1  \quad \mbox{and} \quad
\delta_*=\min\{(4 C_1)^{-2}, \delta_0\}.
\end{equation}
With such $\ee$ and $\delta_*$, the estimates (\ref{3.7.1})--(\ref{3.7.4}) imply that the estimates
({\bf BA1})--({\bf BA4}) can be improved as
$$
|a-1|\le \frac{1}{4}, \,\, \left|\tr \chi -\frac{2}{s}\right|\le
\frac{1}{2} \ee,\,\, \|\hat{\chi}\|_{L_\omega^\infty L_t^2(\N^{-}(p, \tau))} ^2
\le \frac{1}{2}\ee, \,\, \|\nu\|_{L_\omega^\infty L_t^2(\N^{-}(p, \tau))}^2 \le
\frac{1}{2}\ee
$$
on $\N^{-}(p, \tau)$ if $\tau\le \min\{i_*,\delta_*\}$. By repeated
use of Theorem \ref{flux.01}, Theorem \ref{d2laps} and Theorem
\ref{cltmin6}, the bootstrap principle implies that the estimates in
the bootstrap assumptions ({\bf BA1})--({\bf BA4}) hold with
$\tau=\min\{i_*, \delta_*\}$, where $\ee$ and $\delta_*$ are
determined by (\ref{3.10.10}) which are positive universal
constants. Consequently, we obtain (\ref{MB1}) in Theorem
\ref{bts.1}.

We remark that the analogous results to Theorem \ref{cltmin6} have been
proved in \cite{KR1,Qwang} for the geodesic foliations where only the bound
of the curvature flux is used. In time foliations, however,
the proof of Theorem \ref{cltmin6} relies not only on the curvature flux but also on
$\N_1[\sl{\pi}]$.

Assuming (\ref{3.7.4}), the following simple argument shows how to
derive (\ref{3.7.1}) with the help of ({\bf BA1}).
 Recall that $a^{-1}=\bg (L, \bT)$ and
$L=-a^{-1}(N+\bT)$. We have
$$
\frac{d}{d s} a^{-1}=\bg (L, {\bf D}_L \bT)
= a^{-2} \bg ( N, {\bf D}_{\bT} \bT ) +a^{-2} \bg ( N,
{\bf D}_N \bT).
$$
Since ${\bf D}_{\bT} \bT=n^{-1} \nabla n $ and $k_{NN}=-\l N, {\bf
D}_N \bT\r$ we obtain $\frac{d}{ds} a^{-1} =- a^{-2}\left(\pi_{0N}+k_{NN}\right)$.
Consequently
\begin{equation}\label{la}
L(a)=\frac{d}{ds} a=\pi_{0N}+k_{NN}.
\end{equation}
Since $\frac{ds}{dt}=-na$, we have
$$
\frac{d}{d t} a= -na \left(\pi_{0N}+k_{NN}\right).
$$
Integrating the above equation along null geodesics initiating from
$p$ and using $a(p)=1$ yields
$$
a-1=\int_t^{t(p)} \left(\pi_{0N} +k_{NN}\right) na dt'=\int_t^{t(p)}
\nu na dt'.
$$
Since $0< n \le 3/t_*^2$,  ({\bf BA1}) and (\ref{3.7.4}}) imply
$$
|a-1|\le C_1 (t(p)-t)^{1/2}\le C_1 \tau^{1/2}
$$
for all $t(p)-\tau\le t \le t(p)$, where $C_1$ could be a different
but universal constant.

 The derivation of (\ref{3.7.2})--(\ref{3.7.4})
however is highly nontrivial and requires lengthy calculation.
The complete proof is contained in \cite{tqwang} where
other related estimates for Ricci coefficients are proved simultaneously.

In order to complete the proof of Theorem \ref{bts.1}, it remains to prove (\ref{MB}) which is
restated in the following result.

\begin{theorem}\label{fina3}
Assume that the condition $({\bf A1})$ holds. Then there exist universal constants
$\delta_*>0$ and $C>0$ such that
$$
\int^{t(p)}_{t(p)-\min\{i_*, \delta_*\} } |k(\Phi(t))|^2 n d t\le C
$$
for all $p\in \M_I$, where $\Phi$ denotes the integral curve of ${\bf T}$ through $p$.
\end{theorem}

The proof of Theorem \ref{fina3} forms the core part of the present paper. It is based
on the formula of $\Box k$ given in Section \ref{wave} and
a Kirchoff-Sobolev representation for $k$ given in Section 9 together with
various estimates on null cones derived in Section 8.

\section{\bf Tensorial wave equation for the second fundamental form}\label{wave}
\setcounter{equation}{0}

In this section we will derive the formula for $\Box k$, where $k$ is
defined in (\ref{defk}) whose projection to $\Sigma_t$ is exactly
the second fundamental form.

\begin{proposition}\label{wavep}
The tensor $k$ defined by (\ref{defk}) verifies the tensorial
wave equation
\begin{align}\label{boxkh}
\Box k_{ij}&=  -n^{-3} \dot{n} \nabla_i\nabla_j n+ n^{-2} \nabla_i
\nabla_j \dot n +2\pi_{0a} \left(\nabla^a k_{ij} -\nabla_i
k_j^a-\nabla_j k_i^a\right) \nn\\
&\quad \, -2 \Tr k \, R_{ij} -R k_{ij} +R \Tr k \, g_{ij}
+2(k_i^a R_{aj}+ k_j^a R_{ai})-2 R_{ab} k^{ab} g_{ij} \nn\\
&\quad\, +n^{-1}( 2k_i^a\nab_a \nab_j n + 2 k_j^a \nab_a \nab_i n
-\Delta n k_{ij} -\Tr k\, \nabla_i\nabla_j n)\nn\\
&\quad\,  +2k_{ia} k^{ab} k_{bj}-\pi_{0a}\pi_0^a k_{ij} -n^{-1}
k_{ij}.
\end{align}
\end{proposition}

\begin{proof}
We first recall that
$$
\Box k_{ij}=-{\bf D}_0{\bf D}_0 k_{ij} +g^{pq} {\bf
D}_p {\bf D}_q k_{ij}.
$$
By using $k_{0\a}=k_{\a0}=0$ and ${\bf
D}_i e_j =\nabla_i e_j-k_{ij} {\bf T}$, we can obtain through a
straightforward calculation that
$$
g^{pq} {\bf D}_p{\bf D}_q k_{ij} =\triangle k_{ij} +\Tr k \, {\bf D}_0
k_{ij}+2k_{ia}k^{ab} k_{bj}.
$$
By using ${\bf D}_{\bf T} {\bf T}=n^{-1} \nabla^i n e_i=-\pi_0^i
e_i$ and $k_{0\a}=k_{\a0}=0$, we can obtain
\begin{align*}
{\bf D}_0 {\bf D}_0 k_{ij} &=e_0({\bf D}_0 k_{ij})+k_i^a {\bf D}_0
k_{aj}+k_j^a {\bf D}_0 k_{ia}+\pi_{0a} \nabla^a
k_{ij}\nn\\
&\quad \, +\pi_{0i} {\bf D}_0 k_{0j}+\pi_{0j} {\bf D}_0 k_{i0}.
\end{align*}
It is easy to see ${\bf D}_0 k_{0j}=\pi_{0a} k^a_j$. From the
equation (\ref{intro04}) it also follows that
\begin{equation}\label{1.29.2}
{\bf D}_0 k_{ij} =e_0(k_{ij}) +2k_{ia} k^a_j=-n^{-1}
\nabla_i\nabla_j n +R_{ij} +\Tr k\, k_{ij}.
\end{equation}
Consequently
\begin{align*}
{\bf D}_0{\bf D}_0 k_{ij} &=e_0({\bf D}_0 k_{ij}) +\pi_{0a}
\nabla^a k_{ij}-n^{-1}\left( k_i^a  \nabla_a\nabla_j n +k_j^a
\nabla_a\nabla_i n\right) \nn \\
&\quad\, +\left( k_i^a R_{aj}+k_j^a R_{ai}\right) +2 \Tr k \, k_{ia}
k^a_j +\pi_{0i} \pi_{0a} k^a_j + \pi_{0j} \pi_{0a} k^a_i.
\end{align*}
Therefore
\begin{align}\label{1.29.4}
\Box k_{ij} &=-e_0({\bf D}_0 k_{ij}) -\pi_{0a} \nabla^a k_{ij}
-\pi_{0i} \pi_{0a} k^a_j
- \pi_{0j} \pi_{0a} k^a_i\nn\\
&\quad\, +n^{-1}\left( k_i^a \nabla_a\nabla_j n +k_j^a
\nabla_a\nabla_i
n\right) -\left(k_i^a R_{aj} +k_j^a  R_{ai}\right)\nn\\
& \quad \, -2 \Tr k \, k_{ia} k^a_j + \triangle k_{ij} +\Tr k \, {\bf
D}_0 k_{ij}+2k_{ia}k^{ab} k_{bj}.
\end{align}

We need to compute $e_0({\bf D}_0 k_{ij})$. It follows from
(\ref{1.29.2}) and $\Tr k=t$, we have
\begin{align}\label{1.29.5}
e_0({\bf D}_0 k_{ij})&=n^{-3} \dot{n} \nabla_i\nabla_j n -n^{-2}
\p_t(\nabla_i\nabla_j n) +n^{-1} \p_t R_{ij} \nn\\
&\quad \, +n^{-1} k_{ij}+\Tr k \,{\bf D}_0 k_{ij}- 2 \Tr k \, k_{ia}
k^a_j.
\end{align}
In order to compute $\p_t(\nabla_i\nabla_j n)$ and $\p_t R_{ij}$,
let $\Gamma_{ij}^a$ denote the Christoffel symbol of $\Sigma_t$.
Then it follows from the equation $\p_t g_{ij}=-2n k_{ij}$ that
$$
\dot{\Gamma}_{ij}^a = -n \left(\nabla_i k_j^a+\nabla_j k_i^a
-\nabla^a k_{ij}\right)-\nabla_i n k_j^a -\nabla_j n k_i^a
+\nabla^a n k_{ij}.
$$
Using $\mbox{div} k=0$ and $\Tr k=t$, this in particular implies
$\dot{\Gamma}_{aj}^a=-\Tr k \nabla_j n$. Therefore, noting that
$\p_t(\nabla_i\nabla_j n)=\nabla_i \nabla_j \dot
n-\dot{\Gamma}_{ij}^a \nabla_a n$, we can obtain
\begin{align}\label{1.29.6}
\p_t(\nabla_i\nabla_j n) &=\nabla_i\nabla_j \dot n +n \nabla_a n
\left(\nabla_i k_j^a +\nabla_j k_i^a -\nabla^a k_{ij}\right)\nn \\
&\quad\,  +\left(\nabla_i n k_j^a+\nabla_j n k_i^a\right) \nabla_a
n -|\nabla n|^2 k_{ij}.
\end{align}
Noting also that $\p_t R_{ij} =\nab_a \dot \Gamma_{ij}^a-\nab_i
\dot \Gamma_{aj}^a\nn$ and $\mbox{div} k=0$, we have
\begin{align*}
\p_t R_{ij}&=\nab_a n \left(2 \nab^a k_{ij}-\nab_i k_j^a -\nab_j
k_i^a\right)-n \left(\nab_a \nab_i k_j^a +\nab_a \nab_j
k_i^a-\triangle k_{ij}\right)\\
& \quad \, +\Delta n k_{ij} -\left(\nab_a \nab_i n \c k_j^a+\nab_a
\nab_j n\c k_i^a\right) +\Tr k \, \nabla_i\nabla_j n.
\end{align*}
With the help of the commutation formula
$$
\nabla_a \nabla_i k_j^a=[\nab_a, \nab_i] k_j^a=R_j{}^a {}_{bi}
k^b_a + R_{ai} k^{a}_j
$$
and the curvature decomposition formula
\begin{equation*}
R_j{}^a {}_{b i}=g_{jb}R^a_i+R_{jb} \delta^{a
}_i-R_{ij}\delta^{a}_ b-R^a_b g_{ji}-\frac{1}{2}(g_{jb}
\delta^{a}_i- g_{ij}\delta_b^a)R,
\end{equation*}
we obtain
\begin{align*}
\nab_a \nabla_i k_j^a&= 2 R_{ia} k^a_j +R_{ja} k^a_i - \Tr k \,
R_{ij} -R_{ab} k^{ab} g_{ij} -\frac{1}{2} R k_{ij} +\frac{1}{2} R
\,\Tr k\, g_{ij}.
\end{align*}
Consequently
\begin{align}\label{1.29.7}
\p_t R_{ij}&=\nab_a n \left(2 \nab^a k_{ij}-\nab_i k_j^a -\nab_j
k_i^a\right)-\left(\nab_a \nab_i n  k_j^a+\nab_a \nab_j
n k_i^a\right)\nn \\
& \quad\, +n \triangle k_{ij} +\triangle n k_{ij} -3n \left(R_{ia} k^a_j
+R_{ja} k^a_i\right)+2n \Tr k \, R_{ij} \nn \\
&\quad \, +2n R_{ab} k^{ab} g_{ij} + n R k_{ij} -n R \Tr k \, g_{ij}
+\Tr k \,\nabla_i\nabla_j n.
\end{align}
Plugging (\ref{1.29.6}) and (\ref{1.29.7}) into (\ref{1.29.5}), and
using $\pi_{0i}=-n^{-1}\nabla_i n$,  it yields
\begin{align*}
e_0({\bf D}_0 k_{ij}) &=n^{-3} \dot n \nabla_i\nabla_j n -n^{-2}
\nabla_i \nabla_j \dot n -\pi_{0a} \left(3\nabla^a k_{ij} -2
\nabla_i k_j^a-2 \nabla_j k_i^a\right) -\pi_{0i} \pi_{0a} k_j^a\\
&\quad\, -\pi_{0j} \pi_{0a} k_i^a +
\pi_{0a} \pi_0^a k_{ij}-n^{-1} \left(\nabla_a\nabla_i n
k_j^a+\nabla_a\nabla_j n k_i^a -\Tr k\, \nabla_i\nabla_j n\right)\\
&\quad \, +\triangle k_{ij} +n^{-1} \triangle n \, k_{ij} -3 \left(R_{ia}
k_j^a+R_{ja} k_i^a\right)+2 \Tr k\, R_{ij} +2R_{ab} k^{ab}
g_{ij}\\
&\quad\, +R k_{ij} -R \Tr k \, g_{ij} +n^{-1} k_{ij} +\Tr k \,{\bf
D}_0 k_{ij} -2\Tr k\, k_i^a k_{aj}.
\end{align*}
Plugging the above equation into (\ref{1.29.4}) gives the desired
equation.
\end{proof}

\section{\bf Proof of Theorem \ref{flux.01}}\label{flux1}
\setcounter{equation}{0}

In this section we will complete the proof of Theorem \ref{flux.01},
i.e. we will show that if $|a-1|\le 1/2$ on $\N^{-}(p, \tau)$ for
some $0<\tau\le i_*$ then
$$
\R(p,\tau)+\F[k](p,\tau)\le C_*,
$$
where $C_*$ is a universal constant.

We will use the following result (see \cite[Lemma 8.1.1]{KC}).

\begin{lemma}\label{recal.1}
Let $P$ be a vector field defined on the domain $\J^-(p,\tau)$. Then
\begin{equation*}
\int_{\N^-(p,\tau)} {\bf g}(P, L)=\int_{\J^-(p, \tau)} {\bf D}^{\mu}
P_\mu-\int_{\Sigma_{t(p)-\tau} \cap \J^-(p)} {\bf g}(P, {\bf T})
d\mu_g,
\end{equation*}
where $\J^{-}(p)$ denotes the causal past of $p$, $\J^{-}(p, \tau)$
denotes the portion of $\J^{-}(p)$ in the slab $[t(p)-\tau, t(p)]$,
and
$$
\int_{\J^-(p,\tau)}  f=\int_{t(p)-\tau}^{t(p)}
dt\left(\int_{\Sigma_{t}\cap \J^-(p)} f n d\mu_g\right).
$$
\end{lemma}

We first show the boundedness of the curvature flux $\R(p,\tau)$.
With the Bel-Robinson tensor $\bQ[\bR]$ defined in Section 2, we
introduce $P_\mu=Q[\bR]_{\mu \b\ga\delta}\bT^\b \bT^\ga \bT^\delta$.
We may apply Lemma \ref{recal.1} to obtain
\begin{align*}
\int_{\N^-(p,\tau)} \bg(P, L)&=\int_{\J^-(p,\tau)} \bd^\mu
P_\mu-\int_{\Sigma_{t(p)-\tau} \cap \J^-(p)} \bQ[\bR](\bT, \bT, \bT.
\bT)d\mu_g,
\end{align*}
With the help of the calculations in subsection 2.1, ({\bf A1}) and
Lemma \ref{es1}, it then yields
\begin{align}\label{cvf}
\left|\int_{\N^-(p,\tau)} \bg(P, L)\right|&\les C.
\end{align}
Note that $\bg(P, L)=\bQ[\bR](\bT,\bT,\bT,L)$ and $\bT=-\f12(a
L+a^{-1}\Lb)$. Since $|a-1|\le 1/2$ on $\N^{-}(p,\tau)$, it follows
from \cite[Lemma 7.3.1]{KC} that $-\bg(P,L)$ is equivalent to
\begin{equation*}
|\a|^2+|\b|^2+|\udb|^2+|\rho|^2+|\sigma|^2.
\end{equation*}
Thus, there holds, for some universal constant $C>0$,
\begin{equation*}
C^{-1} \R(p,\tau)\le\left|\int_{\N^-(p,\tau)} \bg(P,  L)\right| \le
C \R(p,\tau).
\end{equation*}
By (\ref{cvf}), we conclude that $\R(p,\tau)\le C_*$ for some
universal constant $C_*$.

Next we will show the boundedness of the $k$-flux
$\F[k](p,\tau)$. With the help of the projection tensor
\begin{equation*}
h^{\a\b}=\bg^{\a\b}+ \bT^{\a}\bT^{\b},
\end{equation*}
for any tensor field $U_{\a_1
\a_2\cdots\a_m}$ in $T\M$, we define $|U|$ as follows
\begin{align*}
|U|^2&=h^{IJ}U_I U_J=h^{\a_1 \b_1}\cdots h^{\a_m \b_m}
U_{\a_1 \a_2\cdots \a_m} U_{\b_1 \b_2\cdots \b_m}\\
h^{IJ}&=h^{\a_1 \b_1}\cdots h^{\a_m \b_m}, \quad\, U_I=U_{\a_1
\a_2\cdots \a_m},\quad\,\, U_J=U_{\b_1\b_2\cdots \b_m}.
\end{align*}
For any $\Sigma_t$-tangent tensor field $U$ in $\M_I$, we define
the energy momentum tensor $Q[U]_{\a\b}$ associated with the
covariant wave operator acting on tensors:
\begin{equation*}
Q[U]_{\a\b}:=h^{IJ}\bd_{\a}U_I \bd_\b U_J-\frac{1}{2}
\bg_{\a\b}h^{IJ}{\bg}^{\mu\nu}\bd_{\mu} U_I\bd_{\nu}U_J.
\end{equation*}
We have
\begin{align*}
\bd^{\b}Q[U]_{\a\b}&=h^{IJ}\bd_\a U_I \Box U_J
+h^{IJ}(\bd_\b\bd_\a
U_I-\bd_\a \bd_\b U_I) \bd^{\b} U_J  \\ 
&\quad \, +\bd^{\b}h^{IJ} (\bd_{\a}U_I \bd_{\b}U_J
-\frac{1}{2}\bg_{\a\b}\bg^{\mu\nu}\bd_{\mu}U_I
\bd_{\nu}U_J) 
\end{align*}
It is easy to see that the last term in the above equation can be
written symbolically as $\pi\cdot {\bf D} U\cdot {\bf D} U$.

Now we apply the above equation to $U=k$. Noting that $h^{0\a}=0$
and $h^{ij}=g^{ij}$, we have
\begin{align}\label{energy.02}
\bd^{\b}(Q[k]_{\a\b}\bT^{\a})& =\bd^{\b} \bT^{\a}Q[k]_{\a\b}+
\bd^{\b}Q[k]_{0\b}\nn\\
&= -k^{ij} Q[k]_{ij} -\pi^{0j} Q[k]_{0j} +\bd_0 k^{ij} \Box k_{ij}
\nn\\
&\quad\,  +[\bd_a, \bd_0] k_{ij} \nabla^a k^{ij} +\pi\cdot \bd
k\cdot \bd k.
\end{align}
In view of the commutation formula
\begin{equation*}
[\bd_m, \bd_0] k_{ij}=\bR_{i}{}^b{}_{m0} k_{bj}+\bR_{j}{}^b{}_{m0}
k_{ib}=-\ep_{ib}^s H_{sm} k^b_j-\ep_{jb}^s H_{sm} k_i^b,
\end{equation*}
we derive symbolically
\begin{align}\label{flen02}
\bd^\b ( Q[k]_{\a\b} \bT^\a)&=-k^{ij} Q[k]_{ij} -\pi^{0j}
Q[k]_{0j} +\bd_0 k^{ij} \Box k_{ij}\nn\\
&\quad\,  + H\cdot k\cdot \nabla k + \pi \cdot \bd k\cdot \bd k.
\end{align}
From the definition of $Q[k]$, it is easy to see that
\begin{align}
Q[k]_{00}&=\frac{1}{2}(|\bd_0 k|^2+|\nabla k|^2),\label{fluen01}\\
Q[k]_{0j}&=\bd_0 k_{pq} \nabla_j k^{pq},\\
Q[k]_{ij}&=\nabla_i k_{pq} \nabla_j k^{pq} -\frac{1}{2} g_{ij}
\left(-|\bd_0 k|^2+ |\nabla k|^2\right).
\end{align}
Therefore
\begin{align}\label{div.01}
\bd^{\b}(Q[k]_{\a\b}\bT^{\a})&=\frac{1}{2} \Tr k \left(-|\bd_0
k|^2+|\nabla k|^2\right) +k\cdot \nabla k \cdot
\nabla k\nn\\
&\quad \, + \bd_0 k\cdot \Box k+ H\cdot k\cdot \nabla k + \pi \cdot
\bd k \cdot \bd k.
\end{align}

We now apply Lemma \ref{recal.1} to $P^\b:=\bT^{\a}Q[k]_\a^\b$ and obtain
\begin{align}
\int_{\N^-(p,\tau)} Q[k](\bT, L) +\int_{\Sigma_{t(p)-\tau}\cap
\J^-(p)} Q[k]_{00}&=\int_{\J^-(p,\tau)}
\bd^{\b}(Q[k]_{\a\b}\bT^{\a}). \label{flux.00}
\end{align}
For the null pair $L$ and $\Lb$, it is easy to see that
$$
Q[k](L,L)=|\nab_L k|^2, \quad \quad Q[k](\Lb,L)=|\sn k|^2.
$$
Since $\bT=-\frac{1}{2}(a L+a^{-1} \Lb)$, we have
$$
Q[k](\bT, L)=-\frac{1}{2}\left(a Q[k](L, L) +a^{-1}
Q[k](\Lb,L)\right)=-\frac{1}{2}\left(a |\nabla_L k|^2+ a^{-1} |\sn
k|^2\right).
$$
Since $|a-1|\le 1/2$, the $k$-flux
defined in (\ref{dfk}) verifies the inequality
\begin{equation*}
-\int_{\N^-(p,\tau)} Q[k](\bT, L)\le \F[k](p,\tau)\le -4
\int_{\N^-(p,\tau)} Q[k](\bT, L).
\end{equation*}
Thus we derive from (\ref{flux.00}) and (\ref{fluen01}) that
\begin{equation}\label{flux.02}
\F[k](p,\tau)\le 4\left|\int_{\J^-(p,\tau)} \bd^\b(Q[k]_{\a\b}
\bT^\a) \right|+ 2 \int_{\Sigma_{t(p)-\tau}\cap \J^-(p)}
\left(|\bd_0 k|^2+|\nabla k|^2\right).
\end{equation}
In view of (\ref{1.29.2}), Lemma \ref{first}, Proposition
\ref{lapse1} and Proposition \ref{N}, we have
\begin{align}\label{1.30.1}
\int_{\Sigma_t} &\left(|\bd_0 k|^2+|\nabla
k|^2\right)\nn\\
&\les \|\nabla^2 n\|_{L^2(\Sigma_t)}^2 +\|Ric\|_{L^2(\Sigma_t)}^2
+\|k\|_{L^4(\Sigma_t)}^4 +\|\nabla k\|_{L^2(\Sigma_t)}^2 \le C.
\end{align}
Moreover, in view of (\ref{div.01}), ({\bf A1}),
Lemma \ref{first}, and the above inequality we have
\begin{align}
\left|\int_{\J^-(p,\tau)} \bd^\b(Q[k]_{\a\b} \bT^\a) \right| &\les
\int_{t(p)-\tau}^{t(p)} \|\bd_0 k\|_{L^2(\Sigma_\tt)}\|\Box
k\|_{L^2(\Sigma_{\tt})} d\tt\nn\\
&\quad\, +\int_{t(p)-\tau}^{t(p)}
\|\pi\|_{L^\infty(\Sigma_\tt)} \left(\|\bd_0
k\|_{L^2(\Sigma_\tt)}^2 +\|\nabla
k\|_{L^2(\Sigma_\tt)}^2\right) d\tt\nn\\
&\quad \, +\int_{t(p)-\tau}^{t(p)} \|k\|_{L^\infty(\Sigma_\tt)}
\|H\|_{L^2(\Sigma_\tt)} \|\nabla k\|_{L^2(\Sigma_\tt)} d\tt
\nn\\
&\les \int_{t(p)-\tau}^{t(p)} \|\Box k\|_{L^2(\Sigma_\tt)} d\tt
+\int_{t(p)-\tau}^{t(p)}
\|\pi\|_{L^\infty(\Sigma_\tt)} d\tt\nn\\
&\le C + C\int_{t(p)-\tau}^{t(p)} \|\Box k\|_{L^2(\Sigma_\tt)} d\tt.
\end{align}
Therefore
\begin{equation}\label{fluxk.2}
\F[k](p,\tau)\le C+ C\int_{t(p)-\tau}^{t(p)} \|\Box
k\|_{L^2(\Sigma_\tt)} d\tt.
\end{equation}

We now recall the formula for $\Box k$ given in Proposition
\ref{wavep} which symbolically can be written as
\begin{align*}
\Box k= -n^{-3} \dot n \nabla^2 n+ n^{-2} \nabla^2 \dot n + \pi \cdot
\pi \cdot \pi +k\cdot \nabla^2 n  +k\cdot Ric
+\pi \cdot \nabla k -n^{-1} k.
\end{align*}
Since $C^{-1}\le n\le C$, we obtain
\begin{align*}
\|\Box k\|_{L_t^1 L_x^2} &\les \|\dot n\|_{L_t^1 L_x^\infty}
\|\nabla^2 n \|_{L_t^\infty L_x^2} +\|\nabla^2 \dot n\|_{L_t^1
L_x^2} +\|\pi\|_{L_t^1L_x^\infty}
\|\pi\|_{L_t^\infty L_x^4}^2\\
&\quad\,  +\|k\|_{L_t^1 L_x^\infty} \|\nabla^2 n\|_{L_t^\infty
L_x^2} +\|k\|_{L_t^1 L_x^\infty}
\|Ric\|_{L_t^\infty L_x^2}\\
&\quad\,  +\|k\|_{L_t^1 L_x^2} +\|\pi \|_{L_t^1
L_x^\infty} \|\nabla k\|_{L_t^\infty L_x^2}.
\end{align*}
In view of the assumption ({\bf A1}), Lemma \ref{first}, Proposition
\ref{lapse1}, Proposition \ref{phit01} and (\ref{1.30.1}), it
follows
$$
\|\Box k\|_{L_t^1 L_x^2}\le C\left(1+\|\pi\|_{L_t^1 L_x^\infty}
+\tau\right) \le C.
$$
Combining the above inequality with (\ref{fluxk.2}) completes the
proof of Theorem \ref{flux.01}.

\section{\bf Trace estimates}\label{null}
\setcounter{equation}{0}

For a point $p\in \M_I$, let $s$ be the affine parameter on the null
cone $\N^{-}(p)$ and let $r$ be the radius of $S_t:=\N^{-}(p)\cap \Sigma_t$
which is defined by (\ref{3.9.1}). On each $S_t$ we introduce the ratio of
area elements
\begin{equation}\label{vt1}
v_t(\omega)=\frac{\sqrt{|\ga|}}{\sqrt{|\cga|}}, \qquad  \omega\in
{\Bbb S}^2.
\end{equation}
We will first show that all the quantities $s$, $r$, $v_t^{1/2}$ and $t(p)-t$
are comparable under the bootstrap assumptions ({\bf BA1})--({\bf BA3}).
Here we say two quantities $\varphi$ and $\psi$ are comparable in the sense that
$C^{-1} \psi\le \varphi\le C\psi$ for some universal constant $C>0$.

\begin{lemma}\label{comp1}
Under the bootstrap assumptions $({\bf BA1})$--$({\bf BA3})$, the
four quantities $s(t)$, $r(t)$, $v_t^{1/2}$ and $t(p)-t$
are comparable on the null cone $\N^{-}(p,\tau)$ with
$\tau\le\min\{i_*, \delta_*\}$, where $\delta_*>0$ is a
universal constant.
\end{lemma}

\begin{proof} The comparability of $s$ and $t(p)-t$ follows from
the relation $\frac{d s}{d t}=-na$ and the bootstrap assumption ({\bf BA1}).
Similar to the derivation of (\ref{3.10.1}), we have under the bootstrap assumptions
({\bf BA1})--({\bf BA3}) that
\begin{equation}\label{e4}
|\gamma-s(t)^2 \stackrel{\circ}\ga|\le \frac{1}{2} s(t)^2\stackrel{\circ} \ga
\end{equation}
for all $t(p)-\min\{i_*, \tau, \delta_*\}\le t<t(p)$, where
$\delta_*$ is a universal constant. This implies immediately that
$\frac{1}{2}s(t)^2\le v_t\le \frac{3}{2} s(t)^2$. Consequently $v_t$
and $t(p)-t$ are comparable. Thus for the area $|S_t|$ of
$S_t$ there holds
$$
C^{-1} (t(p)-t)^2\le |S_t|\le C (t(p)-t)^2
$$
for some universal constant $C$. This together with the definition of $r$ gives the comparability
of $r$ and $t(p)-t$.
\end{proof}

\subsection{\bf Optical function}

In this section we give a brief review of the
construction of optical functions, one may see \cite{KC}
for more information.

For any point $p\in \M_I$, let $J^{-}(p)$ be the causal past and let
$\N^{-}(p)$ and $\I^{-}(p)$ denote respectively the null boundary
and the interior. For each $0<\tau<i_*$ with $i_*$ defined by
(\ref{3.8.4}), let $\J^{-}(p, \tau)$, $\N^{-}(p, \tau)$ and
$\I^{-}(p, \tau)$ denote the portions of $\J^{-}(p)$, $\N^{-}(p)$
and $\I^{-}(p)$ in the time slab $[t(p)-\tau, t(p)]$ respectively.
Let $\Phi$ be the integral curve of $\bT$ through $p$ with
$\Phi(t(p))=p$. According to the definition of $i_*$, all the null
cones $\N^{-}(\Phi(t), \tau+t-t(p))$, with $t(p)-\tau\le t\le t(p)$
and $\tau<i_*$, are disjoint and their union forms $\N^{-}(p,
\tau)$. We now define $u$ to be the function, constant on each
$\N^{-}(\Phi(t), t+\tau-t(p))$, such that
$$
u(\Phi(t))=\int_{t_0}^t n(\Phi(t')) d t'.
$$
Such $u$, which will be called an optical function, is a
well-defined smooth function on $\J^{-}(p, \tau)$ and satisfies the
eikonal equation
$$
\bg^{\a\b}\p_\a u\p_\b u=0.
$$
It is clear that the level sets $C_u$ of $u$ are the incoming null
cones in the time slab $[t(p)-\tau, t(p)]$ with vertices on $\Phi$, and
$\bT (u)=1$ on $\Phi$. Moreover, the null geodesic vector $L$
defined before can be written as
$$
L=\bg^{\a\b} \p_\b u\p_\a.
$$

For each $t\in [t(p)-\tau,t(p)]$, we define $u_M(t)$ and $u_m(t)$
respectively to be the largest and smallest values of $u$ for which
the part of the cone $C_u$ that lies in the future of $\Sigma_t$ is
contained in $\J^{-}(p)$, i.e.
$$
u_M(t)=u(p) \qquad \mbox{and}\qquad u_m(t)=u(\Phi(t)).
$$
For each $u(\Phi(t(p)-\tau))\le u\le u(p)$, we also define $t_M(u)$
and $t_m(u)$ to be the largest and smallest value of $t$ for which
$\Sigma_t$ intersects $C_u$ respectively. It is clear that $t_M(u)$
is the value of $t$ at the vertex of $C_u$ and $t_m(u)=t(p)-\tau$.
Note that both $u_M$ and $t_m$ are independent of $t$.

We set
\begin{equation*}
S_{t,u}:=C_u\cap \Sigma_t
\end{equation*}
which is a smooth surface for each $t(p)-\tau\le t\le t(p)$ and
$u_M\le u<u_m(t)$. The corresponding radius function is defined as
$$
r(t,u):=\sqrt{(4\pi)^{-1} |S_{t,u}|},
$$
where $|S_{t, u}|$ denotes the area of $S_{t,u}$ with respect to the metric $\gamma$.

The following result follows immediately from Lemma \ref{comp1}
and the definition of $u$.

\begin{proposition}\label{comr}
Under the bootstrap assumptions $({\bf BA1})$--$({\bf BA3})$ on
$\N^{-}(p,\tau)$ for all $p\in \M_I$, there hold
\begin{equation}\label{comr1}
C^{-1} \le \frac{t_M(u)-t}{r(t,u)} \le C
\end{equation}
and
\begin{equation}\label{commr1}
C^{-1}\le \frac{u-u_m(t)}{r(t,u)}\le C
\end{equation}
for all $t(p)-\min\{i_*, \tau, \delta_*\}<t<t(p)$, where $C$ and
$\delta_*$ are two positive universal constants.
\end{proposition}

In view of the above notations, it is clear that
$$
\N^{-}(p, \tau)=\bigcup_{t\in [t(p)-\tau, t(p)]} S_{t, u_M}.
$$
Let $\mbox{Int}(S_{t, u_M})$ be the interior of $S_{t, u_M}$ in $\Sigma_t$, then
$$
\mbox{Int}(S_{t, u_M})=\bigcup_{u\in [u_M, u_m(t)]} S_{t, u} \quad
\mbox{and} \quad \J^{-}(p,\tau)=\bigcup_{t\in [t(p)-\tau, t(p)]}
\mbox{Int}(S_{t, u_M}).
$$

The following simple result can be found in \cite{KC}.

\begin{lemma}\label{sint}
For any scalar $f$ satisfying
$$
\lim_{u \rightarrow
u_m(t)}\int_{S_{t,u}} f d\mu_\ga=0,
$$
there holds
\begin{equation*}
\int_{S_{t,u_M}}f d\mu_{\ga} =-\int_{u_m(t)}^{u_M}\int_{S_{t,
u}}(\nabla_N f+\tr{\theta} f) a d\mu_{\ga_u} du,
\end{equation*}
where $N$ denotes the unit inward normal to $S_{t,u}$ in $\Sigma_t$, and $\theta$ denotes
the corresponding second fundamental form.
\end{lemma}

\subsection{\bf Trace estimates}

We will rely on the following trace inequality.

\begin{lemma}[Trace inequality]\label{tr1}
Under the bootstrap assumptions $({\bf BA1})$--$({\bf BA3})$ on
$\N^{-}(p,\tau)$ with $\ee \tau\le 1$, for any $\Sigma_t$ tangent tensor field $F$ there
holds
\begin{equation*}
\|r^{-1/2} F\|_{L^2(S_{t})}\les  \|\nab F\|_{L^2(\Sigma_t)}
+\|F\|_{L^2(\Sigma_t)},
\end{equation*}
where $S_t:=\N^{-}(p,\tau)\cap \Sigma_t$ and $r:=\sqrt{(4\pi)^{-1}
|S_t|}$.
\end{lemma}

The proof of Lemma \ref{tr1} can be seen in Appendix.
Using Lemma \ref{tr1},  we are able to derive the following

\begin{proposition}\label{sint2}
Let the bootstrap assumptions $({\bf BA1})$--$({\bf BA3})$ hold on
$\N^{-}(p,\tau)$ with $\ee \tau\le 1$. Then for any $\Sigma_t$ tangent tensor field $F$
there hold
\begin{align}
&\|
F\|_{L^2(S_{t})}^2\les\|F\|_{H^1(\Sigma_t)}\|F\|_{L^2(\Sigma_t)},\label{sint3}\\
&\|F\|_{L^4(S_{t})}\les
\|F\|_{H^1(\Sigma_t)}\label{sint5}
\end{align}
for all $t(p)-\tau \le t<t(p)$.
\end{proposition}

\begin{proof}
Let $\phi(u)$ be a smooth cut-off function verifying $0\le \phi\le
1$, $\phi(u_M)=1$ and $\mbox{supp}(\phi)\subset[\frac{u_m+u_M}{2},
u_M]$. It then follows from Lemma \ref{sint} that
\begin{equation}\label{sint4}
\| F\|_{L^2(S_{t})}^2= -\int_{\itt(S_{t})} \left( \nabla_{N}|\phi F|^2 +
\tr{\theta} |\phi F|^2\right)  a d\mu_{\ga}du' =I_1+I_2,
\end{equation}
where
\begin{align*}
I_1&= -2\int_{\itt (S_{t})}\left( \phi^2 F\cdot\nab_N F
+\phi \nab_N \phi |F|^2\right)  a  d\mu_{\ga}du', \\
I_2&=-\int_{\itt (S_{t})} \tr{\theta} |\phi F|^2 a d\mu_{\ga}du'.
\end{align*}
Since the bootstrap assumption ({\bf BA1}) implies $1/2\le a\le
3/2$, it is easy to see that
\begin{equation*}
\left|\int_{\itt (S_{t})} \phi^2 F\cdot \nab_N F a d\mu_{\ga}du'
\right|\les \|\nab_N F\|_{L^2(\Sigma_t)} \|F\|_{L^2(\Sigma_t)}
\end{equation*}
and
\begin{align*}
\left|\int_{\itt (S_{t})} \phi \nab_N \phi |F|^2  a d \mu_\ga du'
\right|\les
 \frac{1}{u_M-u_m} \int_{\frac{u_m+u_M}{2}}^{u_M}
\int_{S_{t, u'}} |F|^2  d\mu_\gamma du'.
\end{align*}
It follows from Lemma \ref{tr1} that
\begin{align*}
\int_{S_{t, u'}} |F|^2  d\mu_\gamma &\les \|r^{-1/2}
F\|_{L^2(S_{t, u'})} \|F\|_{L^2(S_{t, u'})} r^{1/2}\\
&\les \|F\|_{H^1(\Sigma_t)} \|F\|_{L^2(S_{t, u'})} r^{1/2},
\end{align*}
where $r:=r(t, u')$. From Proposition \ref{comr} it follows
that $r(t, u')\les u'-u_m$. Thus
\begin{align*}
&\left|\int_{\itt (S_{t})} \phi \nab_N \phi |F|^2 a d \mu_\ga
du'\right| \\
&\qquad\quad \les \frac{1}{u_M-u_m} \|F\|_{H^1(\Sigma_t)}
\|F\|_{L^2(\Sigma_t)}
\left(\int_{\frac{u_m+u_M}{2}}^{u_M} (u'-u_m) du'\right)^{1/2}\\
&\qquad\quad \les \|F\|_{H^1(\Sigma_t)} \|F\|_{L^2(\Sigma_t)}.
\end{align*}
We therefore obtain
$$
|I_1|\les \|F\|_{H^1(\Sigma_t)} \|F\|_{L^2(\Sigma_t)}.
$$

In order to estimate the term $I_2$, we recall that $\tr \theta= - a
\tr \chi +\delta^{AB} k_{AB}$. Since the bootstrap assumption
({\bf BA2}) implies $|\tr \chi -2/s|\le \ee$ on each $S_{t, u'}$ and Proposition
\ref{comr} implies that $s$, $t(p)-t$ and $r$ are comparable, we have
\begin{align*}
|I_2|&\les \left(\ee \tau+1\right) \int_{u_m}^{u_M}\int_{S_{t,u'}}
r^{-1} |\phi F|^2 d\mu_\gamma d u'+\int_{u_m}^{u_M}\int_{S_{t,u'}}
|k| |\phi F|^2
d\mu_\gamma du'\\
&\les \int_{u_m}^{u_M}\int_{S_{t,u'}} r^{-1}
|\phi F|^2 d\mu_\gamma d u' + \|k\|_{L^3(\Sigma_t)}
\|F\|_{L^3(\Sigma_t)}^2 .
\end{align*}
Recall that $\|k\|_{L^3(\Sigma_t)}\le C$ from Lemma \ref{first}
and apply Lemma \ref{sob.02} to $\|F\|_{L^3(\Sigma_t)}^2$ we
obtain
$$
|I_2|\les \|F\|_{H^1(\Sigma_t)}\|F\|_{L^2(\Sigma_t)}+
\int_{\frac{u_m+u_M}{2}}^{u_M}\int_{S_{t,u'}} r^{-1}  |F|^2
d\mu_\gamma d u'.
$$
Now we use Lemma \ref{tr1} again and note that Proposition
\ref{comr} implies $r(t,u')^{-1}\les (u'-u_m)^{-1}$, we have
\begin{align*}
\int_{\frac{u_m+u_M}{2}}^{u_M} \int_{S_{t,u'}} r^{-1} |F|^2
d\mu_\gamma  du' &\les \|F\|_{H^1(\Sigma_t)} \|F\|_{L^2(\Sigma_t)}
\left(\int_{\frac{u_m+u_M}{2}}^{u_M} (u'-u_m)^{-1} du'\right)^{1/2}\\
&\les \|F\|_{H^1(\Sigma_t)} \|F\|_{L^2(\Sigma_t)}.
\end{align*}
Therefore
\begin{align*}
|I_2|\les  \|F\|_{H^1(\Sigma_t)} \|F\|_{L^2(\Sigma_t)}.
\end{align*}
The proof of (\ref{sint3}) is complete.

Applying (\ref{sint4}) with $|F|$ replaced by $|F|^2$, combined with
Sobolev embedding, we can obtain (\ref{sint5}) in the similar fashion.
\end{proof}

As a consequence, we obtain

\begin{proposition}\label{P1}
Let the bootstrap assumptions $({\bf BA1})$--$({\bf BA3})$ hold on
$\N^{-}(p, \tau)$ with $\ee \tau\le 1$. Let
$S_{t}:=\N^{-}(p,\tau)\cap \Sigma_t$ and let $r$ be defined by
(\ref{3.9.1}). Let $\pi_0$ denote the tensor $-\nabla \log n$.
\begin{enumerate}
\item[]
\begin{enumerate}
\item[$(a)$] Let $\pib$ denote either $k$ , $\pi_0$ or $\bd_0 \log n$,
then for $t(p)-\tau\le t\le t(p)$
\begin{align}
\|\pib\|_{L^4(S_{t})} &\le C, \label{l10}\\
\|r^{-1/2}\pib\|_{L^2(S_{t})} &\le C. \label{l21}
\end{align}

\item[$(b)$] Let $F$ denotes either $n^{-1}\nab^2 n$ or $n^{-2}\nab\nd$, then
\begin{equation}\label{l7}
\|F\|_{L^2(\N^{-}(p,\tau))}\le C.
\end{equation}

\item[$(c)$] For $\pi_0$, there holds
\begin{equation}\label{dlpi}
\|\nab_L \pi_0\|_{L^2(\N^{-}(p,\tau))}+\|\bd_0
\pi_0\|_{L^2(\N^{-}(p,\tau))}+\|\nab\pi_0\|_{L^2(\N^{-}(p,\tau))}\le C
\end{equation}
\end{enumerate}
\end{enumerate}
\end{proposition}

\begin{proof}
(a) From Lemma \ref{first}, Proposition \ref{lapse1} and Lemma
\ref{L2.1.1} it follows that $\|\pib\|_{H^1(\Sigma_t)}\le C$.
Thus (\ref{l10}) follows from (\ref{sint5}) in Proposition
\ref{sint2} and (\ref{l21}) follows from Lemma \ref{tr1}.

(b) For $F=(n^{-1} \nabla^2 n, n^{-2} \nabla \dot{n})$ it follows from
Proposition \ref{lapse1}, Proposition \ref{Ba22}, Lemma \ref{L2.1.1}
and Proposition \ref{phit01} that
$$
\|\nab F\|_{L_t^{1} L_x^2(\M_*)}\le C \quad \mbox{and}\quad
\|F\|_{L_t^\infty L_x^2(\M_*)}\le C.
$$
Applying (\ref{sint3}) to $F$ yields
\begin{align*}
\|F\|^2_{L^2(\N^{-}(p,\tau))} \les \|F\|_{L_t^1 H_x^1(\M_*)}
\|F\|_{L_t^\infty L_x^2(\M_*)}\les C.
\end{align*}

(c) By straightforward calculation, symbolically we have
\begin{align*}
 {\bf D}_0 \pi_0 &=-n^{-2}\nab \nd+\underline{\pi}\c \pi_0,\\
 \nab\pi_0&=-n^{-1}\nab^2 n+\underline{\pi}\c \pi_0,\\
\nab_{L}\pi_0&=a^{-1} n^{-2}\nab \nd-a^{-1} \nab\pi_0-a^{-1}
\underline{\pi} \c \pi_0.
\end{align*}
Therefore, (\ref{dlpi}) follows immediately from (\ref{l10}) and (\ref{l7}).
\end{proof}

\section{\bf Estimates on the null cones }
\setcounter{equation}{0}

\subsection{\bf Structure equations on the null cones}

In Section 4 we introduced the null pair $L$, $\Lb$ on the null cone
$\N^{-}(p,\tau)$ and define the null second fundamental forms
$\chi$, $\chib$ and the Ricci coefficients $\zeta$ and $\zb$. For
the null frame $(e_A)_{A=1,2}$, $e_3=\Lb$, $e_4=L$, there hold
\begin{eqnarray}
\bd_A \Lb=\chib_{AB}e_B+\zeta_A \Lb, && \qquad \bd_A
L=\chi_{AB}e_B-\zeta_A L,\nn\\
\bd_\Lb \Lb=2\underline{\xi}_A e_A+2 \omega \Lb, && \qquad  \bd_\Lb
L=2\zeta_A e_A -2 \omega L, \nn\\
\bd_L \Lb=2\zb_A e_A, &&\qquad  \bd_L L=0 \label{llb}
\end{eqnarray}
and
\begin{align}
\bd_B e_A &=\sn_B e_A+\frac{1}{2}\chi_{AB}
e_3+\frac{1}{2}\chib_{AB} e_4, \label{angba}\\
\bd_4 e_A&=\sn_4 e_A +\zb_A e_4, \label{de4a} \\
\bd_3 e_A &= \sn_3 e_A+\zeta_A e_3+\xi_A e_4,\nn
\end{align}
where $\sn$ denotes the covariant differentiation on $S_t$.

Let $\a$, $\beta$, $\rho$ and $\sigma$ be the null components of
$\bR$ defined in (\ref{f14}). There hold the following structure
equations on null cones (see \cite[p.351--360]{KC}.)
\begin{align}
&\frac{d\tr \chi}{ds}+\frac{1}{2}(\tr\chi)^2=-|\chih|^2, \label{s1}\\
&\frac{d\chih_{AB}}{ds}+\tr\chi\chih_{AB}=\a_{AB}, \label{s2}\\
&\frac{d}{ds} \zeta_A=-\chi_{AB}\zeta_B+\chi_{AB}\zb_B-\b_A, \label{s3}\\
&\frac{d}{ds}\tr\chib+\frac{1}{2}\tr\chi\tr\chib=2\div\zb-\chih\c\chibh+2|\zb|^2+2\rho. \label{s5}
\end{align}
Moreover, $\zeta$ verifies the following Hodge system
\begin{align}
&\div\zeta=-\mu-\rho+\frac{1}{2} \chih \cdot \chibh-|\zeta|^2
-\frac{1}{2}a\delta \tr\chi -a\lambda \tr \chi,
\label{hdg1}\\
&\curl\zeta=\sigma-\frac{1}{2} \chih\wedge \chibh, \label{hdg2}
\end{align}
where $\mu$ and $\und\mu$ are the mass aspect functions defined by
\begin{align}
\mu&=-\frac{1}{2} \bd_3 \tr\chi+\frac{a^2}{4}(\tr\chi)^2-\omega
\tr\chi, \label{m1}\\
\und{\mu}&=\bd_4 \tr\chib+\frac{1}{2} \tr\chi\c \tr\chib, \label{m2}\\
\omega &=\frac{1}{2}(\bd_3\log a+a k_{NN}-a\pi_{0N}). \label{c20}
\end{align}

Let $N$ be the unit inward normal to $S_t$ in $\Sigma_t$ and let
$\theta$ be the second fundamental form of $S_t$, i.e. $\theta_{AB}=
g(\nabla_A N, e_B)$. Then there hold
\begin{align}
\nab_N e_A &=\sn_N e_A +a^{-1} \sn_A a N, \label{t5}\\
\nab_A N&=\theta_{AB} e_B,  \label{t6}\\
\nab_B e_A&=\sn_B e_A-\theta_{AB}N, \label{t7}\\
\nab_N N&=-a^{-1}\sn_A a e_A\label{t8}.
\end{align}

We introduce the new null pair $L':=\bT+N$, $\Lb':=\bT-N$. Then
$L=-a^{-1}  L'$ and $\Lb=-a \Lb'$.
Let $\chi', \chib', \zeta', \zb', \nu$ be the Ricci coefficients
corresponding to the null frame $(e_A)_{A=1,2}$, $e_3'=\Lb', e_4'=L'$.
Then
\begin{eqnarray*}
\chi=-a^{-1}\chi', \qquad \chib=-a\chib',\qquad
\zeta=\zeta', \qquad  \zb=\zb'
\end{eqnarray*}
and
\begin{align}
&\chi'_{AB}=\theta_{AB}-k_{AB}, \label{c1}\\
&\chib'_{AB}=-\theta_{AB}-k_{AB}, \label{c2}\\
&\zeta'_A=\sn_A \log a+\ep_A, \label{c7}\\
&\zb'_A=\sn_A \log n-\ep_A, \label{c4}\\
&\nu=-\sn_N\log n+\delta-\lambda. \label{c5}
\end{align}

\subsection{\bf Proof of Theorem \ref{d2laps}}

The main purpose of this subsection is to prove Theorem \ref{d2laps}
concerning the boundedness of $\N_1[\slashed \pi]$ under the
bootstrap assumptions ({\bf BA1})--({\bf BA3}) on $\N^{-}(p,\tau)$
with $0<\tau\le i_*$ and $\ee \tau\le 1$ for any $p\in \M_I$, where $\slashed \pi$ is
defined by (\ref{psl}) and the Sobolev norm $\N_1[F]$ for any $S_t$
tangent tensor field $F$ is defined by (\ref{n1sob}). We can restate
Theorem \ref{d2laps}  in the following form, since the estimates for $\lambda$
are trivial.

\begin{proposition}\label{1r}
Let $\slashed\pi$ be the $S_t$ tangent tensor field defined in (\ref{psl}), and
let $\bar\pi:=(k, -\nabla \log n)$. Then, under the bootstrap
assumptions $({\bf BA1})$--$({\bf BA4})$ with $\ee \tau\le 1$,
there hold
\begin{align}
\|r^{-1} \bar\pi\|_{L^2(\N^{-}(p,\tau))} &\le C, \label{1r1}\\
\|\sn \slashed \pi\|_{L^2(\N^{-}(p,\tau))} &\le C, \label{lrr}\\
\|\sn_L \sl{\pi}\|_{L^2(\N^{-}(p,\tau)) } &\le C. \label{lr2}
\end{align}
\end{proposition}

We have obtained in Theorem \ref{flux.01} and (\ref{dlpi}) that
\begin{equation}\label{flxd}
\|\sn \bar\pi\|_{L^2(C_u)}+\|\nab_L \bar \pi\|_{L^2(C_u)}\le C.
\end{equation}
In view of (\ref{t6}), (\ref{t7}) and (\ref{llb}), (\ref{de4a}), we can symbolically write
\begin{equation}\label{kprj4}
\sn \slashed{\pi}= \sn \bar\pi+\tr \theta\c
\slashed{\pi}+\hat\theta\c \slashed{\pi}
\end{equation}
and also in view of $\frac{dt}{ds}=-(an)^{-1}$,
\begin{equation}\label{lk1}
\sn_L \sl{\pi}=\nab_L \bar\pi+\sl{\pi}\c\zb + (an)^{-1}.
\end{equation}

In order to show Proposition \ref{1r}, we need three auxiliary
lemmas. We will use the following norms for $\Sigma_t$ tangent
tensor fields $F$ on null cones $\N^{-}(p,\tau)$
\begin{align*}
\|F\|_{L_x^q L_t^\infty(\N^{-}(p,\tau))}^q &:=\int_{{\Bbb S}^2}
\sup_{t\in \Ga_\omega} \left( v_t
|F|_g^q \right) d\mu_{{\Bbb S}^2},\\
\|F\|_{L_\omega^q L_t^\infty(\N^{-}(p,\tau))}^q &:= \int_{{\Bbb
S}^2} \sup_{t\in \Gamma_\omega} |F|_g^q d\mu_{{\Bbb S}^2}.
\end{align*}
where $v_t$ is defined by (\ref{vt1}), and $\Ga_\omega$, $\omega\in
{\Bbb S^2}$, denotes the portion of an incoming null geodesic
initiating from $p$ in the time slab $[t(p)-\tau, t(p)]$. In the
following argument we will suppress $\N^{-}(p,\tau)$ in these norms
for simplicity.

\begin{lemma}
For any $S_t$ tangent tensor field $F$, there hold the estimates
\begin{equation}\label{sob3}
\|r^{-1/2} F\|_{L_x^2 L_t^\infty}+\|F\|_{L_x^4 L_t^\infty}\les
\N_1[F],
\end{equation}
\begin{equation}\label{intp1}
\| F\|_{L_x^4 L_t^\infty}^2\les \left(\|\sn_L F\|_{L^2} + \|r^{-1}
F\|_{L^2}\right)\|F\|_{L_\omega^\infty L_t^2}.
\end{equation}
\end{lemma}

\begin{proof}
We refer to \cite{KR1,Qwang} for the proof of (\ref{sob3}). In the
following we will prove (\ref{intp1}). Let $v_t$ be defined by
(\ref{vt1}). We first integrate along any past null geodesic
initiating from $p$ to get
\begin{equation}\label{intp2}
v_t|F|^4=\lim_{t\rightarrow t(p)} (v_t |F|^4) -\int_t^{t(p)}
\frac{d}{dt'} (v_{t'} |F|^4) dt'.
\end{equation}
For the estimate of the first term on the right of (\ref{intp2}), we
proceed as follows. Let $\varphi$ be a smooth cut-off function defined
on $[t(p)-\tau, t(p)]$ verifying $0\le \varphi\le 1$, $\varphi(t(p))=1$ and
$\mbox{supp}\varphi \subset [t(p)-\tau/2, t(p)]$. Then
\begin{equation}\label{intp3}
\lim_{t\rightarrow t(p)} v_t |F|^4 = \int_{t(p)-\tau}^{t(p)}
\left(\frac{d}{dt}(v_t |F|^4)\varphi^4+ 4 v_t |F|^4
\varphi^3\frac{d}{dt} \varphi \right) dt.
\end{equation}
Since $|\frac{d}{dt} \varphi| \les (t(p)-t)^{-1}$, we have from
Lemma \ref{comp1} that $|\frac{d}{dt} \varphi| v_t^{\f12}\les 1$.
Using $0\le \varphi\le 1$, it then follows from (\ref{intp2}) and
(\ref{intp3}) that
\begin{equation}\label{3.12.1}
\|F\|_{L_x^4 L_t^\infty}^4 =\int_{{\Bbb S}^2} \sup_{t(p)-\tau\le
t\le t(p)}(v_t |F|^4) \les I+II,
\end{equation}
where
\begin{align*}
I=\int_{\Bbb S^2}  \int_{t(p)-\tau}^{t(p)} \left|\frac{d}{dt} (v_t
|F|^4)\right| dt,\qquad II=\int_{{\Bbb S}^2}\int_{t(p)-\tau}^{t(p)}
v_t^{1/2} |F|^4.
\end{align*}
Since
\begin{equation*}
\frac{d}{dt} (v_t |F|^4)=-na\left(\tr\chi v_t |F|^4+4 v_t |F|^2
\sn_L F\c F\right),
\end{equation*}
we have
\begin{align*}
I &\les \left(\|v_t^{1/2} \sn_L F\|_{L_\omega^2 L_t^2} +\|\tr\chi
v_t^{1/2} F\|_{L_\omega^2L_t^2} \right)\|F\|_{L_\omega^\infty
L_t^2}
\|v_t^{1/2} |F|^2\|_{L_\omega^2 L_t^\infty}\\
&\les \left(\|\sn_L F\|_{L^2} +\|\tr\chi F\|_{L^2} \right)
\|F\|_{L_\omega^\infty L_t^2} \|F\|_{L_x^4 L_t^\infty}^2.
\end{align*}
By the bootstrap assumption ({\bf BA2}) and Lemma \ref{comp1} we
have
\begin{align*}
\| \tr\chi F\|_{L^2}&\les
\left\|\tr\chi-\frac{2}{s}\right\|_{L^\infty}
\tau \|r^{-1} F\|_{L^2}+\|r^{-1} F\|_{L^2}\\
&\les (\ee \tau+1) \|r^{-1} F\|_{L^2}\lesssim \|r^{-1} F\|_{L^2}.
\end{align*}
Therefore
\begin{align*}
I &\les \left(\|\sn_L F\|_{L^2} +\|r^{-1} F\|_{L^2} \right)
\|F\|_{L_\omega^\infty L_t^2} \|F\|_{L_x^4 L_t^\infty}^2.
\end{align*}
It is easy to see that
\begin{align*}
|II| &\les \|F\|_{L_\omega^2 L_t^2}\|F\|_{L_\omega^\infty L_t^2}
\|v_t^{1/2} |F|^2\|_{L_\omega^2 L_t^\infty} \les \|r^{-1} F\|_{L^2}
\|F\|_{L_\omega^\infty L_t^2} \|F\|_{L_x^4 L_t^\infty}^{2}.
\end{align*}
Combining the estimates for $I$ and $II$ with (\ref{3.12.1}) gives
(\ref{intp1}).
\end{proof}

\begin{lemma}\label{tsp2}
For any $S_t$  tangent tensor field $F$ verifying
\begin{equation}\label{tsp1}
\sn_L F+\frac{m}{2} \tr\chi F= G\c F+ H
\end{equation}
with $m\ge 1$ an integer and $G$ a tensor field of suitable type, if
$\lim_{t\rightarrow t(p)} r(t)^m F=0$ and $\sup_{\omega \in {\Bbb
S}^2}\int_{t(p)-\tau}^{t(p)} na |G|^2 dt \le \Delta_0^2$, the
following estimates hold
\begin{align}
\|F\|_{L_\omega^2L_t^2}&\les e^{C \Delta_0
\tau^{1/2}}\|H\|_{L^2}, \label{tsp3}\\
\|r^{\f12} F\|_{L_\omega^2 L_t^\infty}&\les e^{C \Delta_0
\tau^{1/2}}\|H\|_{L^2}\label{tsp5}.
\end{align}
\end{lemma}

\begin{proof}
In what follows, we will use Lemma \ref{comp1} to compare
$v_t^{1/2}$, $r$, $s$ and $t(p)-t$ if necessary. Since
$\frac{d}{dt}v_t = -na \tr\chi v_t$, along any past null
geodesic initiating from $p$ we have
\begin{align*}
\frac{d}{dt}( v_t^m |F|^2)= -2na v_t^m \l H+F\c G, F\r
\end{align*}
With the help of the $\lim_{t\rightarrow t(p)} r^{m}|F|=0$, it
follows for $t(p)-\tau\le t\le t(p)$ that
$$
v_t^m |F|^2 =2 \int_t^{t(p)} na v_{\tt}^m \l H+F\cdot G, F\r \le 2
\int_t^{t(p)} na v_{\tt}^m \left(|F||H| +|F|^2 |G|\right).
$$
By a simple argument we can derive
\begin{equation*}
v_t^{m/2} |F|\le  \exp\left(\int_t^{t(p)}  |G| na \right)
\int_t^{t(p)}  na v_{t'}^{m/2} |H| \exp\left(-\int_{t'}^{t(p)} n a
|G| \right) dt'.
\end{equation*}
In view of $\sup_{\omega \in {\Bbb S}^2}\int_{t(p)-\tau}^{t(p)} na
|G|^2 dt \le \Delta_0^2$, we have $\exp(\int_t^{t(p)} na |G|)\le
e^{C \Delta_0 \tau^{1/2}}$. Thus by using Lemma \ref{comp1} and $m\ge
1$, we have
\begin{align}\label{tran.2}
|F| &\le  e^{C \Delta_0 \tau^{1/2}} v_t^{-m/2}\int_t^{t(p)}  v_{t'}^{m/2} |H| na dt'\nn\\
&\les e^{C \Delta_0 \tau^{1/2}} (t(p)-t)^{-1} \int_t^{t(p)} r|H| dt'.
\end{align}

To derive (\ref{tsp3}), we integrate the above inequality along a
null geodesic initiating from vertex $p$. By the Hardy-Littlewood
inequality
$$
\left\|\frac{1}{s}\int_0^s |f|\right\|_{L_s^2}\les \|f\|_{L_s^2}
$$
it follows that
\begin{align}\label{l23}
\| F\|_{L_t^2} &\les  e^{C \Delta_0 \tau^{1/2}}
\left\|\frac{1}{t(p)-t} \int_t^{t(p)} r|H|\right\|_{L_t^2} \nn\\
&\les e^{C \Delta_0 \tau^{1/2}}  \|r H\|_{L_t^2}.
\end{align}
Integrating (\ref{l23}) with respect to the angular variable $\omega\in
{\Bbb S}^2$ yields (\ref{tsp3}).

Next we multiply (\ref{tran.2}) by $r^{\f12}$ to obtain
\begin{equation*}
\sup_{t(p)-\tau\le t\le t(p) } r^{\f12} |F|\les  e^{C \Delta_0
\tau^{1/2}} \|r H\|_{L_t^2} ,
\end{equation*}
which, by taking  the $L_\omega^2$ norm,  gives (\ref{tsp5}).
\end{proof}

In view of (\ref{s2}) and Lemma \ref{tsp2}, we are able to prove the
following estimates for $\chih$.

\begin{lemma}\label{L3.11.1}
For $\chih$ there hold the estimates
\begin{equation}\label{chih1}
\|r^{-1}\chih\|_{L^2}+\|r^{1/2}\chih\|_{L_\omega^2
L_t^\infty}+\|\sn_L \chih\|_{L^2}\le C,
\end{equation}
\begin{equation}\label{chih2}
\|\chih\|_{L_x^4 L_t^\infty}\le C \ee^{1/4}.
\end{equation}
\end{lemma}

\begin{proof}
We will use the transport equation (\ref{s2}), i.e.
\begin{equation}\label{2.8.1}
\sn_L \chih +\tr \chi \chih =\a.
\end{equation}
Recall that $r \chih\rightarrow 0$ as $t\rightarrow t(p)$, see
\cite{Qwang}. Recall also that $\|\a\|_{L^2} \le C$, see Theorem
\ref{flux.01}. It then follows from Lemma \ref{tsp2} that
$$
\|r^{1/2} \chih\|_{L_\omega^2 L_t^\infty} +\|\chih\|_{L_\omega^2 L_t^2}\le C.
$$
Next we use (\ref{2.8.1}) again to estimate $\|\sn_L \chih\|_{L^2}$.
With the help of the bootstrap assumption ({\bf BA2}) and the
comparability of $r$, $s$ and $t(p)-t$ given in Lemma \ref{comp1},
we have
\begin{align*}
\|\tr\chi \, \chih\|_{L^2}&\les
\left\|\tr\chi-\frac{2}{s}\right\|_{L^\infty} \|r \chih\|_{L_t^2
L_\omega^2} +\| r^{-1} \chih\|_{L^2} \le C.
\end{align*}
Thus, from (\ref{2.8.1}) it follows
$$
\|\sn_L \chih \|_{L^2} \les \|\tr \chi \chih\|_{L^2}
+\|\a\|_{L^2}\le C.
$$
We therefore complete the proof of (\ref{chih1}).

By making use of (\ref{intp1}) and (\ref{chih1}) together with the
bootstrap assumption ({\bf BA3}) we obtain
\begin{align*}
\|\chih\|_{L_x^4 L_t^\infty}&\les (\|\sn_L \chih\|_{L^2}
+\|r^{-1}\chih\|_{L^2})^{\f12}\|\chih\|_{L_\omega^\infty
L_t^2}^{\f12} \le C \ee^{1/4}
\end{align*}
which gives (\ref{chih2}).
\end{proof}

Now we are ready to complete the proof of Proposition \ref{1r}.

\begin{proof}[Proof of Proposition \ref{1r}]
We first prove (\ref{1r1}). Let $|\bar\pi|:=|\bar \pi|_g$. It is
easy to check
$$
\sn_L (s^{-1}|\bar\pi|^2) +\tr\chi s^{-1} |\bar\pi|^2
=s^{-1}(\tr\chi -\frac{2}{s})|\bar\pi|^2+s^{-2} |\bar\pi|_g^2+2
s^{-1} \nab_L \bar \pi\c \bar\pi.
$$
We integrate the above equation along the null cone $\N^-(p, \tau)$.
By Lemma \ref{comp1}, it is easy to see $\int_{S_t} s^{-1} |\bar \pi|^2
\rightarrow 0$ as $t\rightarrow t(p)$. Therefore, by integration by
parts we obtain
\begin{align}
\int_{\N^{-}(p,\tau)} \left(s^{-2} |\bar \pi|^2 +s^{-1}
(\tr\chi-\frac{2}{s})|\bar\pi|^2+2s^{-1} \nab_L \bar\pi\c \bar \pi\right)
na d\mu_\ga dt=\int_{S_{t(p)-\tau}}s^{-1}|\bar \pi|^2\nn.
\end{align}
By Lemma \ref{comp1} and (\ref{l21}) in Proposition \ref{P1} we have
\begin{equation*}
\left|\int_{S_{t(p)-\tau}} s^{-1} |\bar \pi|^2 \right| \les
 \|r^{-1/2} \bar \pi\|_{L^2(S_{t(p)-\tau})}^2 \le C.
\end{equation*}
By ({\bf BA2}), Lemma \ref{comp1} and (\ref{l21}),
\begin{equation*}
\left|\int_{\N^-(p,\tau)}na s^{-1} (\tr\chi-\frac{2}{s})|\bar\pi|^2
d\mu_\ga dt \right|\le C\ee \tau\le C.
\end{equation*}
By (\ref{flxd}) we have
\begin{align*}
\left|\int_{\N^{-}(p,\tau)}s^{-1}\nab_L \bar \pi\c \bar \pi n a
d\mu_\ga dt\right| & \les \|\nab_L \bar \pi\|_{L^2}\|s^{-1} \bar
\pi\|_{L^2}\le C\|s^{-1} \bar \pi\|_{L^2}.
\end{align*}
Therefore
\begin{equation}\label{chi9}
\|s^{-1} \bar\pi\|_{L^2}^2 \le C+C\|s^{-1} \bar
\pi\|_{L^2}
\end{equation}
which implies $\|s^{-1}\bar{\pi}\|_{L^2} \le C$.
Consequently, in view of Lemma \ref{comp1}, (\ref{1r1}) follows.
As a byproduct, we have from ({\bf BA2}) and Lemma \ref{comp1} that
\begin{equation}
\|\tr\chi \bar\pi\|_{L^2} \les
\|s^{-1}\bar\pi\|_{L^2}+\left\|\tr\chi-\frac{2}{s}\right\|_{L^\infty}\tau
\|s^{-1}\bar\pi\|_{L^2}
\le C(1+\ee \tau)\le C.\label{hk4}
\end{equation}

Next we will show (\ref{lrr}). we will use the equation
(\ref{kprj4}), i.e.
\begin{equation}\label{3.11.20}
 \sn \slashed{\pi}= \sn \bar \pi+\tr
\theta\c \slashed{\pi}+\hat\theta\c \slashed{\pi}.
\end{equation}
Using $\theta_{AB}=- a \chi_{AB} +k_{AB}$, we have from (\ref{l10})
and (\ref{chih2}) that
\begin{align*}
\|\hat\theta\c \slashed \pi\|_{L^2}&\les \|\slashed
\pi\|_{L^4}\left(\|k\|_{L^4}+\|\chih\|_{L^4}\right) \le C
\left(\ee^{1/4}+1\right)\tau^{1/2}\le C.
\end{align*}
Since $\tr\theta=-a\tr\chi+\delta^{AB}k_{AB}$,  we have from
(\ref{l10}) and (\ref{hk4}) that
\begin{equation*}
\|\tr\theta \slashed \pi\|_{L_t^2 L_x^2}\les\|k\|_{L^4}
\|\slashed\pi\|_{L^4}+ \|\tr\chi\slashed\pi\|_{L^2}\le C.
\end{equation*}
Consequently, in view of (\ref{flxd}) and (\ref{3.11.20}), (\ref{lrr})
follows immediately.

In view of (\ref{lk1}) and (\ref{c4}), (\ref{lr2}) follows
immediately from (\ref{flxd}) and (\ref{l10}).
\end{proof}

\subsection{\bf  Estimates for Ricci coeffients }

\begin{lemma} For the Ricci coefficient $\zeta$ and the null lapse $a$ there hold
\begin{align}
\|r^{\f12}\zeta\|_{L_\omega^2
L_t^\infty}+\|r^{-1}\zeta\|_{L^2}&+\|\sn_L \zeta\|_{L^2} \le C, \label{nta2}\\
\|r^{\f12} \sn\log a\|_{L_\omega^2 L_t^\infty}+\|r^{-1}\sn \log
a\|_{L^2}& +\|\sn_L \sn\log a\|_{L^2} \le C.\label{nta1}
\end{align}
\end{lemma}

\begin{proof}
From the transport equation (\ref{s3}) we have
\begin{equation}\label{s3.new}
\sn_L \zeta +\frac{1}{2} \tr \chi \cdot \zeta =-\chih \cdot \zeta +\chi\cdot \zb-\beta.
\end{equation}
Since ({\bf BA3}) implies $\|\chih\|_{L_\omega^\infty L_t^2} \le
\ee^{1/2}$ with $\ee \tau\le 1$, it follows from Lemma \ref{tsp2} and
the relation $\chi=\chih +\frac{1}{2} \tr \chi \, \gamma$ that
\begin{align*}
\|r^{\frac{1}{2}}\zeta\|_{L_\omega^2
L_t^\infty}+\|r^{-1}\zeta\|_{L^2} &\les \|\beta\|_{L^2} + \|\chih \cdot \zb\|_{L^2}
+\|\tr\chi \c \zb\|_{L^2}
\end{align*}
From Theorem \ref{flux.01} we have $\|\beta\|_{L^2}\le C$. Recall
that $\zb=\sn \log n -\ep$ which is a combination of terms in
$\slashed \pi$. By (\ref{hk4}) we have $\|\tr \chi \, \zb\|_{L^2}
\le C$. Therefore
\begin{align*}
\|r^{1/2} \zeta\|_{L_\omega^2 L_t^\infty} +\|r^{-1} \zeta\|_{L^2}
&\le C\left(\ee \tau+1\right) +\|\chih \cdot \zb\|_{L^2}.
\end{align*}
In view of (\ref{l10}) in Proposition \ref{P1}, (\ref{chih2})
in Lemma \ref{L3.11.1}, and $\ee \tau\le 1$, we have
\begin{align*}
\|r^{\frac{1}{2}}\zeta\|_{L_\omega^2
L_t^\infty}+\|r^{-1}\zeta\|_{L^2}
&\le C+ \tau^{1/2}\|\chih\|_{L_x^4 L_t^\infty} \|\zb\|_{L_t^\infty L_x^4}
\le C.
\end{align*}
Consequently, it follows from (\ref{s3.new}), ({\bf BA2}) and ({\bf
BA3}) that $\|\sn_L \zeta\|_{L^2} \le C$. We thus
obtain (\ref{nta2}).

In order to show (\ref{nta1}), we use the relation $\zeta=\sn\log
a+\ep$. By Proposition \ref{1r},
$$
\|r^{\frac{1}{2}}\ep\|_{L_\omega^2 L_t^\infty}+\|\ep\|_{L_\omega^2
L_t^2}+\|\sn_L \ep\|_{L^2}\le C.
$$
Thus, the estimates for $\sn \log a $ follows.
\end{proof}

\begin{lemma}\label{mub}
For the $\underline{\mu}$ defined by (\ref{m2}) there holds
$\|\underline{\mu}\|_{L^2}\le C$ on $\N^-(p,\tau)$.
\end{lemma}

\begin{proof}
Recall that by (\ref{s5}), $\underline{\mu}=2 \mbox{div} \zb
-\chih\cdot \underline{\chih} +2|\zb|^2 +2\rho$. We have from
Theorem \ref{flux.01}, Proposition \ref{P1} and Theorem \ref{d2laps} that
\begin{align*}
\|\underline{\mu}\|_{L^2} &\les \|\sn \zb\|_{L^2} +\|\zb\|_{L^4}^2
+\|\rho\|_{L^2} +\|\chih\cdot
\underline{\chih}\|_{L^2} \les C+\|\chih \cdot \underline{\chih}\|_{L^2}.
\end{align*}
Recall also the relation $\chibh'=-\chih' -2 \eh$, we have from
(\ref{chih2}) and Proposition \ref{P1} that
\begin{align*}
\|\underline{\mu}\|_{L^2} \les C+
\|\chih\|_{L^4}\left(\|\chih\|_{L^4}+ \|k\|_{L^4}\right)\le C.
\end{align*}
The proof is thus complete.
\end{proof}

In the following we summarize the estimates obtained so far in this
section.

\begin{proposition}\label{P3}
There exists universal constants $\delta_0>0$ and $C_*>0$ such that,
under the bootstrap assumptions $({\bf BA1})$--$({\bf BA3})$ with $\ee \tau\le 1$,
if $\tau<\min\{i_*, \delta_0\}$ then there hold
\begin{align}
\|r^{-\frac{1}{2}} \pib\|_{L^2(S_{t,u})} &\le  C, \label{rtr}\\
\|\pib\|_{L^4(S_{t,u})} & \le C, \label{r9} \\
\N_1[\slashed\pi](p,\tau) & \le C ,\label{r7}\\
\|n^{-1}\nab^2 n, n^{-2}\nab\nd\|_{L^2} &\le C,\label{r10}\\
\|r^{\f12}(\chih, \bpi,\zeta, \sn \log a, \hat \theta)\|_{L_\omega^2 L_t^\infty}
   &\le C, \label{c6}\\
\|(\chih,\bpi,\zeta,\sn \log a, \hat\theta)\|_{L_t^2 L_\omega^2}
   &\le C, \label{r4}\\
\|\sn_L (\chih, \zeta, \sn \log a, \hat \theta)\|_{L^2} & \le C, \label{r11}
\end{align}
where $\pib=(n^{-1}\p_t\log n, \bpi)$.
\end{proposition}

The above estimates provide the intermediate steps toward
the proof of Theorem \ref{cltmin6}. The complete proof however requires
more estimates on $\chih$, $\zeta$ and $\zb$ as follows. Since the arguments
are rather lengthy, we will report them in \cite{tqwang}.

\begin{proposition}\label{P4}
There exists universal constants $\delta_0>0$ and $C_*>0$ such that,
under the bootstrap assumptions $({\bf BA1})$--$({\bf BA4})$ with $\ee \tau\le 1$,
if $\tau<\min\{i_*, \delta_0\}$ then there hold
\begin{align}
\left\|\tr\chi-\frac{2}{s}\right\|_{L^\infty} &\le  C_*, \label{r1}\\
\|\chih\|_{L_\omega^\infty L_t^2}+ \|\zeta\|_{L_\omega^\infty
L_t^2} &\le C_*,\label{r14}\\
\|\nu\|_{L_\omega^\infty L_t^2}+\|\zb\|_{L_\omega^\infty
L_t^2} &\le C_*,\label{r2}\\
\N_1[\chih, \zeta,\sn\log a,\hat \theta](p,\tau) & \le C_*, \label{r3}\\
\|r^{\f12}(\sn\tr\chi,\mu)\|_{L_x^2
L_t^\infty}+\|(\sn\tr\chi, \mu)\|_{L^2} &\le C_*,\label{derim}
\end{align}
on the null cone $\N^{-}(p,\tau)$ for all $p\in \M_I$.
\end{proposition}

The estimates in Proposition \ref{P3} and Proposition \ref{P4} gives Theorem
\ref{cltmin6}. Thus, we may use a bootstrap argument, as explained in Section 4, to conclude
that all the estimates in the above two propositions hold on the null cones
$\N^{-}(p, \tau)$ for all $p\in \M_I$ with $\tau=\min\{i_*, \delta_*\}$ for
some universal constant $\delta_*>0$.

We conclude this section with an application to estimate
$\|\pib\|_{L_u^2 L_\omega^2(\mbox{Int}(S_{t,u}))}$, where,
for any $\Sigma$ tangent tensor F,
$$
\|F\|_{L_u^2 L_\omega^2(\itt S_{t,u})}^2=\int_{u_m}^u
\int_{{S_{t,u'}}}{r'}^{-2} |F|_g^2 a d\mu_{\ga} du'
$$
with $r'=r(t,u')$.

\begin{proposition}\label{P2.6.10}
For $\pib=(n^{-1}\p_t\log n, \bpi)$, there holds
\begin{equation}\label{2phd}
\|\pib\|_{L_u^2 L_\omega^2(\mbox{Int}(S_{t,u}))}\le C.
\end{equation}
\end{proposition}

\begin{proof}From (\ref{s1}), (\ref{m1}), (\ref{c20}) and (\ref{hdg1}), we can
derive
\begin{equation}\label{ntrx}
\sn_N \tr\chi'+\frac{1}{2}(\tr\chi')^2
=-\frac{1}{2}\delta\tr\chi'+2\lambda
\tr\chi' -\chih'(\chih'+\eh)-(\div \zeta+|\zeta|^2+\rho),
\end{equation}
which, multiplied by $|\pib|:=|\pib|_g$, implies
\begin{align*}
\nab_N & (\tr\chi'{|\pib|_g}^2) +\tr\theta(\tr\chi'|\pib|_g^2)
-\frac{1}{2}|\tr\chi' \pib|_g^2\\
&=\left\{-\frac{3}{2}\delta\tr\chi'-\chih'(\chih' +\eh)-(\div
\zeta+|\zeta|^2+\rho) \right\} |\pib|^2+ 2 \tr\chi' \nab_N \pib\c \pib,
\end{align*}
 In view of
Lemma \ref{sint}, integrating the above equation over $\itt
(S_{t,u})$ gives
\begin{align}
\frac{1}{2} \int_{u_m}^{u}&\int_{S_{t,u'}} (\tr\chi')^2|\pib|^2 a d\mu_{\ga}d u' \nn\\
&=-\int_{S_{t, u}} \tr\chi'|\pib|^2+\int_{u_m}^{u}
\int_{S_{t,u'}} \left(-2 \nab_N \pib \c\tr\chi' \pib+\rho|\pib|^2\right) a d\mu_{\ga} du'\nn\\
&+\int_{u_m}^{u}\int_{S_{t,u'}} \left(\frac{3}{2}\delta
\tr\chi'+|\zeta|^2+\chih'(\chih'+\eh)\right) |\pib|^2 a d\mu_{\ga} du'\nn\\
&+\int_{u_m}^{u}\int_{S_{t,u'}} -\zeta \cdot \sn(|\pib|^2 a)
d\mu_{\ga}du'\label{lze}
\end{align}
By ({\bf BA2}), Lemma \ref{comp1} and (\ref{l21}),
\begin{align*}
\left|\int_{S_{t, u}} \tr\chi'|\pib|^2 d\mu_{\ga}\right|\les \|r^{-1/2}
\pib\|_{L^2(S_{t,u})}^2 \le C.
\end{align*}
By Lemma \ref{first},  Proposition \ref{lapse1} and (\ref{phdin}),
\begin{align*}
&\left|\int_{u_m}^{u} \int_{S_{t,u'}} \nab_N \pib\c \tr\chi' \pib a d\mu_\ga du'\right| \les
\|\nab_N \pib\|_{L^2(\Sigma_t)} \|\tr\chi' \pib\|_{L^2(\Sigma_t)}\le
C \|\tr\chi' \pib\|_{L^2(\Sigma_t)}
\end{align*}
and
\begin{align*}
\left|\int_{u_m}^{u'}\int_{S_{t,u}}\frac{3}{2}\delta \tr\chi'|\pib|^2
a d\mu_{\ga} d u' \right|
&\les\|k\|_{L^6(\Sigma_t)}\|\pib\|_{L^6(\Sigma_t)}^2
+\|\pib\|_{L^3(\itt (S_{t,u}))}^3\\
&\les (\|\nab k\|_{L^2(\Sigma_t)}+\|\pib\|_{H^1(\Sigma_t)})\|\pib\|_{H^1(\Sigma_t)}^2\\
& \le C.
\end{align*}
By Lemma \ref{es1} and (\ref{l10}),
\begin{align*}
&\left|\int_{u_m}^{u} \int_{S_{t,u'}}\rho |\pib|^2 a d\mu_{\ga} du'\right| \les
\|\rho\|_{L^2(\Sigma_t)} \|\pib\|_{L^4(\itt S_{t,u})}^2\le C
(u-u_m)^{1/2}.
\end{align*}
By (\ref{c7}) we have
\begin{align*}
&\left|\int_{u_m}^{u}\int_{S_{t,u'}} \zeta\sn(a |\pib|^2)d\mu_{\ga}d u'\right|
 =\left|\int_{u_m}^u \int_{S_{t,u'}} (\sn \log a |\pib|^2\zeta
+\sn|\pib|^2\zeta ) a d\mu_{\ga}du'\right|\\
&\qquad\qquad \les  \|\sn \pib\|_{L^2(\itt (S_{t,u}))} \sup_{u_m\le u'\le u}
\left(\|\pib\|_{L^4(S_{t,u'})} \|\zeta\|_{L^4(S_{t,u'})}\right)
(u-u_m)^{1/2} \\
&\qquad\qquad \quad\, +\int_{u_m}^{u}\int_{S_{t,u'}}
(|\zeta|^2|\pib|^2+|\zeta||\pib|^3) a d\mu_{\ga}d u'.
\end{align*}
In view of Lemma \ref{first}, Proposition \ref{lapse1} and
(\ref{phdin}) we derive
$$
\|\sn \pib\|_{L^2(\itt (S_{t,u}))}\le\|\nab \pib\|_{L^2(\Sigma_t)} \le C,
$$
while in view of (\ref{r3}), (\ref{sob3}) and (\ref{l10}) we have
$$
\sup_{u_m\le u'\le u} \|\zeta\|_{L^4(S_{t,u'})}\le C,
\qquad \sup_{u_m\le u'\le u} \|\pib\|_{L^4(S_{t,u'})}\le C.
$$
Consequently,
\begin{align*}
\int_{u_m}^{u}\int_{S_{t,u'}}
\left(|\zeta|^2|\pib|^2+|\zeta||\pib|^3\right)a d\mu_{\ga} du'
& \les \sup_{u_m\le u'\le u} \left(\|\zeta\|_{L^4(S_{t,u'})}^2\|\pib\|_{L^4(S_{t,u'})}^2\right) (u-u_m)\\
&+\sup_{u_m\le u'\le u} \left(\|\zeta\|_{L^4(S_{t,u'})} \|\pib\|_{L^4(S_{t,u'})}^3\right) (u-u_m)\\
&\le C(u-u_m).
\end{align*}
Therefore, we obtain
\begin{align*}
\left|\int_{u_m}^{u}  \int_{S_{t,u'}} \zeta\sn(a |\pib|_g
^2)d\mu_{\ga}d u'\right|\le C(1+(u-u_m)^{1/2})(u-u_m)^{1/2}.
\end{align*}
In view of (\ref{r3}), (\ref{sob3}) and (\ref{l10}), by a similar
argument we obtain
\begin{align*}
&\left|\int_{u_m}^{u}\int_{S_{t,u'}} (|\zeta|^2+\chih'(\chih'+\eh))|\pib|^2
a d\mu_{\ga} d u'\right| \\
&\qquad \qquad \les
\int_{u_m}^{u}\int_{S_{t,u'}}(|\pib|^2(|\chih|^2+|\zeta|^2)
+|\chih|\c|\pib|^3 )d\mu_{\ga}d u'\\\
& \qquad\qquad \le C (u-u_m)
\end{align*}
Combining all the above estimates with (\ref{lze}) and noting that
$u-u_m\les \tau\les 1$, it yields
\begin{equation*}
\|\tr\chi \pib\|_{L^2(\itt (S_{t,u}))}^2\le C+ C \|\tr\chi \pib\|_{L^2(\itt(S_{t,u}))}
\end{equation*}
which implies $\|tr \chi \pib\|_{L^2(\itt(S_{t,u}))}\le C$. This
together with ({\bf BA2}) implies the desired inequality.
\end{proof}

\section{\bf Proof of Theorem \ref{fina3}}
\setcounter{equation}{0}

In this section we will complete the proof of Theorem \ref{fina3}.
For any $p\in \M_I$, let $\Phi(t)$ be the integral curve of
$\bT$ through $p$ with $\Phi(t(p))=p$. For each $p_t:=\Phi(t)$, we will
represent $k(p_t)$ in terms of a Kirchoff-Sobolev formula over a past
null cone with vertex $p_t$. We then use the
estimates established in the previous section to obtain
$\int_{t(p)-\tau}^{t(p)} |k(\Phi(t))|^2 n dt \le C$ for some universal
constant $C$.

\subsection{\bf Derivation of Kirchoff Parametrix}

We first revisit the formulation of Kirchoff Parametrix in
\cite{Ksob}. We define $\bA$ to be a $\Sigma_t$ tangent 2-tensor
verifying
\begin{equation}\label{tran.1}
(\bd_L \bA)_{ij}+\frac{1}{2}\tr\chi \bA_{ij}=0\,\,\,\mbox{ on }
\N^-(p, \tau),  \qquad \lim_{t\rightarrow t(p)^{-}} (t(p)-t)
\bA_{ij}=J_{ij},
\end{equation}
where $J\in T_p \Sigma_{t(p)}$ and $|J|_g=1$. This ${\bf A}$ is
similar to the one defined in \cite{KR2} but with the modification
that ${\bf A}$ is $\Sigma_t$ tangent. Since we have obtained in
Propositions \ref{P3} and \ref{P4} the estimates on
\begin{align*}
\left\|\tr\chi-\frac{2}{s}\right\|_{L^\infty}, \,\,\|\sn
\tr\chi\|_{L^2}, \, \, \|r^{\f12}\sn \tr\chi\|_{L_x^2 L_t^\infty},
\,\, \|r^{-1}(\zeta+\zb)\|_{L^2}, \,\, \|\chih, \nu,
\zb\|_{L_\omega^\infty L_t^2}, \,\, \R(p,\tau)
\end{align*}
on the null cone $\N^{-}(p,\tau)$, we may adapt the proof in
\cite{KR2} to obtain the following estimates on ${\bf A}$.

\begin{proposition}\label{P2.6.1}
For the tensor ${\bf A}$ defined by (\ref{tran.1}) there hold
\begin{equation}\label{rec1}
\|\sn \bA\|_{L^2(\N^-(p,\tau))}+ \|r^{\f12} \sn \bA\|_{L_x^2
L_t^\infty(\N^{-}(p,\tau))} +\|r\bA\|_{L^\infty(\N^-(p, \tau))}\le C,
\end{equation}
where $C$ is a universal constant.
\end{proposition}

Now we revisit the Kirchoff-Sobolev formula for any $\Sigma_t$
tangent 2-tensor $\Psi_{I}$, $I=\{\mu,\nu\}$, see \cite{Ksob,Shao}.
According to the definition of $\Box \Psi_{I}$, we have
under the null frame $(e_A)_{A=1,2}$, $e_3=\Lb$, $e_4=L$ that
\begin{equation*}
\Box\Psi_{I}=-\frac{1}{2}\bd_{43}\Psi_{I}-\frac{1}{2}\bd_{34}\Psi_{I}
+\delta^{AB}\bd_{AB}\Psi_{I}.
\end{equation*}
By (\ref{llb}),
\begin{equation}
\bd_{43}\Psi_{I}=\bd_4(\bd_3 \Psi)_{I}-2\zb^A \bd_A \Psi_{I}.
\end{equation}
It is easy to see
$$
\bd_{34}\Psi_{I}-\bd_{43}\Psi_{I}=\bR_{\mu}{}^\a{}_{34} \Psi_{\a
\nu} + \bR_{\nu}{}^\a{}_{34} \Psi_{\mu\a}.
$$
By (\ref{angba}), we obtain
\begin{equation*}
\delta^{AB}\bd_{AB}\Psi_{I}=\delta^{AB}\sn_A \sn_B \Psi_{I}-\f12\tr
\chib \bd_4 \Psi_{I} -\f12\tr \chi \bd_3 \Psi_{I}.
\end{equation*}
Therefore
\begin{align*}
\Box \Psi_{I}&=-\bd_4 ({\bd_3\Psi})_{I} +2 \zb^A\bd_A \Psi_{I}
-\frac{1}{2} \tr\chib \bd_4 \Psi_{I} -\frac{1}{2}\tr\chi \bd_3
\Psi_{I}\\
&\quad\,  +\delta^{AB}\sn_A \sn_B \Psi_{I}
-\frac{1}{2} \bR_{\mu}{}^\a{}_{34} \Psi_{\a \nu}
-\frac{1}{2} \bR_{\nu}{}^\a{}_{34} \Psi_{\mu\a}.
\end{align*}
We multiply the above equation by ${\bf A}_{I}$ and integrate over
$\N^{-}(p,\tau)$ to obtain
\begin{align}
\int_{\N^-(p,\tau)} \Box \Psi_{I} \bA^{I}&=\Xi_1+ \Xi_2
+\int_{\N^-(p,\tau)} \left( 2 \zb^A  \bd_A \Psi_{I} \c \bA^{I}
+\delta^{AB}\sn_A
\sn_B \Psi_{I} \c \bA^{I}\right) \nn\\
&\quad \, -\frac{1}{2} \int_{\N^{-}(p,\tau)}  \left(
\bR_{\mu}{}^\a{}_{34} \Psi_{\a \nu} +\bR_{\nu}{}^\a{}_{34}
\Psi_{\mu\a}\right){\bf A}^{\mu\nu}.   \label{kirl2}
\end{align}
where
\begin{align*}
\Xi_1&=\int_{\N^-(p,\tau)}\left(-\bd_4
(\bd_3 \Psi)_{I} \c \bA^{I}-\frac{1}{2}\tr\chi \bd_3 \Psi_{I} \c \bA^{I}\right),\\
\Xi_2&=-\frac{1}{2}\int_{\N^-(p,\tau)} \tr\chib \bd_4 \Psi_{I} \c
\bA^{I}.
\end{align*}
For $\Xi_1$, integrating  by parts gives
\begin{align*}
\Xi_1&=\int_{\N^-(p,\tau)}\left(-\bd_4 (\bd_3 \Psi)_{I} \c \bA^{I}
-\tr\chi \bd_3 \Psi_{I} \c \bA^{I} +\frac{1}{2}\tr\chi \bd_3 \Psi_{I} \c \bA^{I} \right)\\
&=-\int_{S_{t(p)-\tau}} \bd_3 \Psi_{I} \c \bA^{I} +
\lim_{t\rightarrow
t(p)} \int_{S_t} \bd_3 \Psi_{I} \c \bA^{I}\\
&\quad\,  +\int_{\N^-(p,\tau)}\left(\bd_4 \bA^{I}+\frac{1}{2}\tr\chi
\bA^{I}\right)\c\bd_3 \Psi_{I}
\end{align*}
Since $\lim_{t\rightarrow t(p)} (t(p)-t)^2\bA=0$, we have in view of (\ref{tran.1}) that
\begin{equation*}
\Xi_1= -\int_{S_{t(p)-\tau}} \bd_3 \Psi_I\c
\bA^I+\int_{\N^-(p,\tau)} \Omega_1(\Psi),
\end{equation*}
where
\begin{equation*}
\Omega_1(\Psi)=\bd_4 \bA^{0i}\c \bd_3 \Psi_{0i}+\bd_4 \bA^{i0}\c
\bd_3 \Psi_{i0}.
\end{equation*}

For the term $\Xi_2$, in view of  (\ref{tran.1}) and the fact
that $\Psi$ is  $\Sigma_t$ tangent, we first have
 \begin{align*}
-\frac{1}{2}\tr\chib\bd_4 \Psi_I\c
\bA^I&=-\frac{1}{2}\left(\bd_4(\Psi_I\c \bA^I \tr\chib)-\bd_4
\bA^I\c \tr\chib\c \Psi_I-\bd_4 \tr\chib \c \Psi_I\c \bA^I\right)\\
&=-\frac{1}{2}\left(\bd_4(\Psi_I\c \bA^I \tr\chib)
+\frac{1}{2}\tr\chi\tr\chib \bA^I\c \Psi_I
-\bd_4 \tr\chib\c \Psi_I\c \bA^I\right),
\end{align*}
thus integration by parts yields
\begin{align*}
\Xi_2&=\int_{\N^-(p,\tau)}\frac{1}{2}\underline\mu \bA^I\c \Psi_I
-\frac{1}{2}\left(\int_{S_{t(p)-\tau}}\Psi_I\c \bA^I
\tr\chib-\lim_{t\rightarrow t(p)} \int_{S_t} \Psi_I \c
\bA^I\tr\chib\right),
\end{align*}
where  $\underline\mu$ is  defined in (\ref{m2}).

In view of $\tr\chib'=-\tr\chi'-2\delta^{AB}k_{AB}$, we have
\begin{align*}
\lim_{t\rightarrow t(p)} \frac{1}{2}\int_{S_t} \Psi_I\c \bA^I
\tr\chib=-4\pi n(p) \l \Psi,J\r,
\end{align*}
Hence
\begin{align*}
\Xi_2=\int_{\N^{-}(p,\tau)} \frac{1}{2} \underline\mu \bA^I\c
\Psi_I-\frac{1}{2}\int_{S_{t(p)-\tau}} \Psi_I\c \bA^I \tr\chib-4\pi
n(p) \l\Psi, J\r.
\end{align*}
Therefore we derive
\begin{align}
4\pi n(p) \langle \Psi, J\rangle &=\int_{\N^{-}(p, \tau)}\left(-\Box \Psi_I \c
\bA^I+\frac{1}{2}\underline \mu \Psi_I\c\bA^I +\Omega_1(\Psi)\right)\nn\\
&\quad\, -\int_{S_{t(p)-\tau}} \left(\bd_3 \Psi_I \c \bA^I
+\frac{1}{2}\tr\chib\Psi_I\c \bA^I\right) \nn\\
& \quad \, +\int_{\N^{-}(p,\tau)} \left(2 \zb^B \bd_B \Psi_I \c
\bA^I-\sn_B\Psi_I \c
\sn^B \bA^I\right)\nn\\
&\quad\, -\frac{1}{2} \int_{\N^-(p,\tau)} \left(\bR_{i}{}^\a{}_{34}
\Psi_{\a j}
 + \bR_{j}{}^\a{}_{34} \Psi_{i \a}\right) {\bf A}^{ij}.
\label{kirchoff}
\end{align}
We apply (\ref{kirchoff}) to the tensor
field $\Psi=k$ and obtain

\begin{theorem}\label{T2.6.1}
Let $p\in \M_I$, let $\Phi(t)$ be the integral curve of ${\bf T}$
through $p$ with $\Phi(t(p))=p$, and let $p_t=\Phi(t)$. Let ${\bf
A}$ be a $\Sigma_t$ tangent 2-tensor on $\J^{-}(p,\tau)$ verifying
(\ref{tran.1}) on each null cone $C_u:=\N^-(p_t, t-t(p)+\tau)$,
where $u=u(t)=\int_{t_0}^t n|_\Phi dt$ for $t_m:=t(p)-\tau\le t\le
t(p)$. Then there holds
\begin{align}\label{kirch}
4 \pi n(p_t) \l k(p_t), J\r =I(p_t)+J(p_t)+K(p_t)+L(p_t)+\mathfrak{E}(p_t)+\int_{C_u}
\Omega_1(k),
\end{align}
where
\begin{align*}
I(p_t)&=-\int_{C_u} {\bf A} \c\Box k, \nn\\
J(p_t)&=-\frac{1}{2} \int_{C_u}{\bf A} \c {\bf R}(\cdot, \cdot, \Lb, L)\c
k, \nn\\
K(p_t)&=\int_{C_u} \left(-\sn^B {\bf A} \c \sn_B k+2\zb^B\c \sn_B k\c {\bf A}\right), \nn\\
L(p_t)&=\frac{1}{2}\int_{C_u}\underline\mu {\bf A} \c k, \nn\\
{\mathfrak{E}(p_t)}&=-\int_{S_{t_m,u}} \left({\bf D}_3 k \c {\bf
A}+\frac{1}{2}\tr\chib k\c {\bf A}\right).
\end{align*}
\end{theorem}

\subsection{\bf Main estimates}

In the following we will use the representation formula given in
Theorem \ref{T2.6.1} to show that
$$
\int^{t(p)}_{t(p)-\tau} |k(p_t)|^2 n dt \le C
$$
for some universal constant $C$. We proceed as follows.

$\bullet$ {\it Estimate on $I(p_t)$:} We will use the expression of
$\Box k$ given in Proposition \ref{wavep}, which symbolically can be
written as
$$
\Box k=-n^{-3} \dot n \nabla^2 n + n^{-2} \nabla^2 \dot n +\pi \cdot \pi \cdot
\pi + k\cdot \nabla^2 n + k\cdot Ric +\pi \cdot \nabla k -n^{-1}
k.
$$
It then follows from Proposition \ref{P2.6.1} that
\begin{align*}
|I(p_t)| &\les \int_{C_u} r^{-1} \left(|\dot n \nabla^2 n|
+|\nabla^2 \dot n|+|\pi|^3 +|k||\nabla^2 n| +|k||Ric|
+|\pi||\nabla k| +|k|\right) \\
&\les \|\nabla^2 n\|_{L^2(C_u)} \|r^{-1} \dot n\|_{L^2(C_u)}+
\|r^{-1} \nabla^2 \dot n\|_{L^1(C_u)} +\int_{C_u}
r^{-1} |\pi|^3\\
&\quad \, + \|r^{-1} k\|_{L^2(C_u)} \|\nabla^2  n\|_{L^2(C_u)}
+\|Ric\|_{L^2(C_u)} \|r^{-1} k\|_{L^2(C_u)} \\
&\quad\, + \|r^{-1} \pi \|_{L^2(C_u)} \|\nabla k\|_{L^2(C_u)}
+\|r^{-1} k\|_{L^1(C_u)}.
\end{align*}
Therefore, with the help of Proposition \ref{P1} and Proposition
\ref{1r}, we have
\begin{align*}
|I(p_t)| &\les \|r^{-1} \dot n\|_{L^2(C_u)}+\|r^{-1} \nabla^2 \dot
n\|_{L^1(C_u)} +\int_{C_u} r^{-1} |\pi|^3  \\
&\quad \, +\|Ric\|_{L^2(C_u)} + \|\nabla k\|_{L^2(C_u)} + C.
\end{align*}

Now we consider $\int_{t_m}^{t(p)} |I(p_t)|^2 dt$. Using $\frac{d
u}{d t}=n$, we have from Proposition \ref{P2.6.10}  that
\begin{align*}
\int_{t_m}^{t(p)} \|r^{-1} \dot n\|_{L^2(C_{u(t)})}^2 n dt
&=\int_{u(t_m)}^{u(t(p))} \|r^{-1} \dot n\|_{L^2(C_u)}^2 du\\
&=\int_{u(t_m)}^{u(t(p))} \int_{t_m}^{t_M(u)} \int_{S_{\tt, u}}
r^{-2} |\dot n|^2 na d\mu_\gamma d\tt du\\
&=\int_{t_m}^{t(p)} \int_{u(\tt)}^{u(t(p))} \int_{S_{\tt, u}}
r^{-2} |\dot n|^2 na d\mu_\gamma du d\tt\\
&\les\int_{t_m}^{t(p)} \|r^{-1} \dot n\|_{L^2(Int(S_{\tt,
u(t(p))}))}^2 d\tt\\
&\le C\tau.
\end{align*}
By similar argument, we have from Lemma \ref{first} that
$$
\int_{t_m}^{t(p)} \left(\|Ric\|_{L^2(C_u)}^2 + \|\nabla
k\|_{L^2(C_u)}^2\right) n dt\le C\tau.
$$
Therefore
$$
\int_{t_m}^{t(p)} |I(p_t)|^2 n dt\les C\tau + \int_{t_m}^{t(p)}
\|r^{-1} \nabla^2 \dot n\|_{L^1(C_u)}^2 n dt + \int_{t_m}^{t(p)}
\left(\int_{C_u} r^{-1} |\pi|^3\right)^2 n dt.
$$
By using the Minkowski inequality and Proposition  \ref{phit01} we
have
\begin{align*}
&\left(\int_{t_m}^{t(p)} \|r^{-1} \nabla^2 \dot n\|_{L^1(C_u)}^2  n
dt\right)^{1/2}\\
&\qquad\qquad =\left(\int_{u(t_m)}^{u(t(p))}
\left(\int_{t_m}^{t_M(u)} r^{-1} \|an \nabla^2 \dot n\|_{L^1(S_{\tt,
u})} d\tt\right)^2 du\right)^{1/2} \\
&\qquad\qquad \le \int_{t_m}^{t(p)} \left(\int_{u(\tt)}^{u(t(p))} r^{-2}
\|an \nabla^2\dot n\|_{L^1(S_{\tt, u})}^2 du\right)^{1/2} d\tt\\
&\qquad\qquad  \les \int_{t_m}^{t(p)} \|\nabla^2 \dot n\|_{L^2(Int(S_{\tt,
u(t(p))}))} d\tt  \le C.
\end{align*}
Finally, we have from Proposition \ref{P1} and (\ref{1r1}) that
$$
\int_{C_u} r^{-1} |\pi|^3 \les \int_{t_m}^{t_M(u)} \|r^{-1} \pi
\|_{L^2(S_{\tt, u})}\|\pi \|_{L^4(S_{\tt, u})}^2 d\tt \le
C(t_M(u)-t_m)^{1/2}.
$$
Thus, by Lemma \ref{comp1} we obtain
\begin{equation*}
\int_{t_m}^{t(p)} \left(\int_{C_u} r^{-1}
|\pi|^3\right)^2 n dt \le C\tau^2.
\end{equation*}
Combining the above estimates we therefore obtain
$$
\int_{t_m}^{t(p)} |I(p_t)|^2 n dt \le C+C \tau^2\les C.
$$

$\bullet$ {\it Estimate on $J(p_t)$:} It follows from Proposition
\ref{P2.6.1}, Theorem \ref{flux.01} and Proposition \ref{1r} that
$$
|J(p_t)|\les \|r{\bf A}\|_{L^\infty(C_u)} \|r^{-1} k\|_{L^2(C_u)}
\R(p_t, \tau+t-t(p))\le C.
$$
Thus
$$
\int_{t_m}^{t(p)} |J(p_t)|^2 n dt \le C(t(p)-t_m)\le C\tau\le C.
$$

$\bullet$ {\it Estimate on $K(p_t)$:} It follows from the
H\"{o}lder inequality that
$$
|K(p_t)|\les \|\sn {\bf A}\|_{L^2(C_u)} \|\sn k\|_{L^2(C_u)} + \|r
{\bf A}\|_{L^\infty(C_u)}\|r^{-1} \zb\|_{L^2(C_u)} \|\sn k
\|_{L^2(C_u)}.
$$
Thus, we obtain from Proposition \ref{P2.6.1}, Theorem
\ref{flux.01},  and Proposition \ref{1r} that
$|K(p_t)|\le C$ which gives
$$
\int_{t_m}^{t(p)} |K(p_t)|^2 n dt \le C(t(p)-t_m)\le C\tau\le C.
$$

$\bullet$ {\it Estimate on $L(p_t)$:} It follows from Proposition
\ref{P2.6.1} and Proposition \ref{1r} that
$$
|L(p_t)|\le \|r{\bf A}\|_{L^\infty(C_u)} \|r^{-1} \bar\pi
\|_{L^2(C_u)} \|\underline{\mu}\|_{L^2(C_u)}\les
\|\underline{\mu}\|_{L^2(C_u)}.
$$
From Lemma \ref{mub} we then obtain $|L(p_t)|\le C$. Therefore
$$
\int_{t_m}^{t(p)} |L(p_t)|^2 n dt \le C(t(p)-t_m)\le C\tau\le C.
$$

$\bullet$ {\it Estimate on $\mathfrak{E}(p_t)$:} We first have
from Proposition \ref{P2.6.1} that
\begin{align*}
|\mathfrak{E}(p_t)|\les r^{-1} \| {\bf D}_3 k\|_{L^1(S_{t_m, u})}
+ r^{-1} \| \tr \chib k\|_{L^1(S_{t_m, u})}.
\end{align*}
Using the definition of $r$ we then obtain
$$
|\mathfrak{E}(p_t)|\les \|{\bf D}_3 k\|_{L^2(S_{t_m, u})} + r^{-1}
\|\tr \chib k\|_{L^1(S_{t_m, u})}.
$$
Recall
$$\tr\chib'=-\tr\chi'-2\delta^{AB}k_{AB}.$$
 Since
({\bf BA1}) implies $1/2\le a\le 3/2$. Thus, with the help of
({\bf BA2}), it yields
\begin{align*}
\|\tr\chib k\|_{L^1(S_{t_m, u})} &\les \left\|\tr \chi
-\frac{2}{s}\right\|_{L^\infty(C_u)} \|k\|_{L^1(S_{t_m,u})} +
r^{-1} \|k\|_{L^1(S_{t_m,u})} +\|k\|_{L^2(S_{t_m,u})}^2\\
&\les r^{-1} \|k\|_{L^1(S_{t_m,u})} +\|k\|_{L^2(S_{t_m, u})}^2\\
&\les \|k\|_{L^2(S_{t_m,u})} +r \|k\|_{L^4(S_{t_m,u})}^2.
\end{align*}
Consequently
$$
|\mathfrak{E}(p_t)|\les \|{\bf D}_3 k\|_{L^2(S_{t_m, u})} +r^{-1}
\|k\|_{L^2(S_{t_m, u})} +\|k\|_{L^4(S_{t_m, u})}^2.
$$
Therefore, using $\frac{d u}{d t} =n$, we have
\begin{align*}
\int_{t_m}^{t(p)} |\mathfrak{E}(p_t)|^2 dt  &\les
\int_{u(t_m)}^{u(t(p))} |\mathfrak{E}(p_t)|^2 du \\
&\les \|{\bf D}_3 k\|_{L^2(\Sigma_{t_m})}^2 + \|r^{-1}
k\|_{L^2(Int(S_{t_m, u}))}^2 +\|k\|_{L^4(\Sigma_{t_m})}^4
\end{align*}
It follows from Lemma \ref{first} and
Proposition \ref{P2.6.10} that
\begin{align*}
\int_{t_m}^{t(p)} |\mathfrak{E}(p_t)|^2 dt \les \|{\bf D}_3
k\|_{L^2(\Sigma_{t_m})}^2 + C.
\end{align*}
Recall that $\Lb=-a(\bT-N)$. So ${\bf D}_3 k=-a ({\bf D}_0
k-\nabla_N k)$. Recall also that ${\bf D}_0 k=-n^{-1} \nabla^2 n
+Ric +k\Tr k $. Thus
\begin{align*}
\|{\bf D}_3 k\|_{L^2(\Sigma_{t_m})}\les \|\nabla^2
n\|_{L^2(\Sigma_{t_m})} +\|Ric\|_{L^2(\Sigma_{t_m})}
+\|k\|_{L^4(\Sigma_{t_m})}^2 +\|\nabla k\|_{L^2(\Sigma_{t_m})}.
\end{align*}
It follows from Lemma \ref{first} and Proposition \ref{lapse1}
that $\|{\bf D}_3 k\|_{L^2(\Sigma_{t_m})}\le C$. Therefore
$$
\int_{t_m}^{t(p)} |\mathfrak{E}(p_t)|^2 n dt \le C.
$$

$\bullet$ {\it Estimate on $\int_{C_u} \Omega_1(k)$:}  By straightforward
calculation we have $\Omega_1(k)={\bf A}\c \bpi\c \bpi\c \bpi$. It follows from
Proposition \ref{P2.6.1} that
$$
|\Omega_1(k)|\les \int_{C_u} r^{-1} |\bpi|^3.
$$ Therefore, one can use the similar argument in the
estimate of $I(p_t)$ to get
$$
\int_{t_m}^{t(p)} |\Omega_1(k)|^2 n dt \les \int_{t_m}^{t(p)}
\left|\int_{C_u} r^{-1} |\bpi|^3\right|^2 n dt \le C\tau^2 \le C.
$$

\section{\bf Proof of main theorem I}
\setcounter{equation}{0}

In this section, based on Theorem \ref{thm4}, we will follow the idea in \cite{KR2} to give
the proof of Theorem \ref{thm3}. According to the local existence
theorem given in \cite[Proposition 6.1]{KR2}, see also \cite[Theorem 10.2.1]{KC},
it suffices to show that the quantity
\begin{align}\label{R0}
\R_*: =\|Ric\|_{H^2(\Sigma_t)} +\|k\|_{H^3(\Sigma_t)}
\end{align}
on each slice $\Sigma_t$ with $t_0\le t<t_*$ is uniformly bounded.

Since $(\bM, \bg)$ is a vacuum space-time, by virtue of the Bianchi identity ${\bf R}$
verifies a wave equation of the form
\begin{equation}\label{boxr}
\Box {\bf R}={\bf R} \star {\bf R},
\end{equation}
Based on higher energy estimates it is standard to show that
\begin{equation}\label{R1}
\|\bd \bR(t)\|_{L^2}^2\les \|\bd \bR(t_1)\|_{L^2}^2 + \int_{t_1}^t \|\bR(t')\|_{L^\infty}^2 d t'
\end{equation}
and
\begin{equation}\label{R2}
\|\bd^2 \bR(t)\|_{L^2}^2 \les \|\bd^2 \bR(t_1)\|_{L^2}^2
+\int_{t_1}^t \|\bd \bR(t')\|_{L^2}^2 \|\bR(t')\|_{L^\infty}^2 d t'
\end{equation}
for all $t_0\le t_1\le t<t_*$.
The derivation has been given in \cite{KR2} under the assumption (\ref{tn1}),
the argument however depends only on the condition ({\bf A1}).

Thus, the derivation of the $L^\infty$ bound of ${\bf R}$ is a
crucial step. As in \cite{Ksob} one can represent ${\bf R}(p)$, for
each $p\in \M_*$, by a Kirchoff-Sobolev formula over the null cone
$\N^{-}(p,\tau)$, where $\tau>0$ is a universal constant such that
$i_*(p,t)\ge \tau$ whose existence is guaranteed by Theorem
\ref{thm4}. One can then follow the delicate argument in \cite{KR2}
to derive that
\begin{equation}\label{R3}
\|\bR(t)\|_{L^\infty}\les \tau^{-1} \sup_{t'\in [t-2\tau, t-\tau/2]}
\left(\|\bR(t')\|_{L^2} +\|\bd \bR(t')\|_{L^2} +\|\bd^2
\bR(t')\|_{L^2}\right).
\end{equation}
The derivation of (\ref{R3}) requires the estimates on
\begin{align*}
\R(p,\tau),&\quad \left\|\tr \chi
-\frac{2}{s}\right\|_{L^\infty(\N^{-}(p,\tau))},
\quad \|\chih, \nu, \zeta, \zb\|_{L_\omega^\infty L_t^2(\N^{-}(p,\tau))},\\
\|\mu, \sn \tr\chi\|_{L^2(\N^{-}(p,\tau))},& \quad \|r^{1/2} \sn \tr
\chi\|_{L_x^2 L_t^\infty(\N^{-}(p,\tau))}, \quad
\|r^{-1}(\zeta+\zb)\|_{L^2(\N^{-}(p,\tau))}
\end{align*}
which are provided by Proposition \ref{P3} and Proposition \ref{P4} under the condition ({\bf A1}).
Combining the estimates (\ref{R1})--(\ref{R3}) gives
$$
\|\bR(t)\|_{H^2}\les \tau^{-1} \sup_{t'\in [t-\tau, t-\tau/2]}
\|\bR(t')\|_{H^2}.
$$
Iterating this estimate as many times as needed, in steps of size
$\tau/2$, yields
\begin{equation}\label{R4}
\sup_{t\in [t_0,t_*)} \|\bR(t)\|_{H^2}\le C,
\end{equation}
where $C$ is a positive constant depending only on $Q_0$, $\k$, $|\Sigma_0|$, $t_*$, $I_0$
and the initial data $\|\bR(t_0)\|_{H^2}$.

Now we are ready to show that the quantity $\R_*$ defined by (\ref{R0}) is uniformly bounded
for all $t_0\le t<t_*$. Although the argument is standard, we will include the details
for completeness.

We have defined in (\ref{em1}) the electric and magnetic parts $E$, $H$ of the
curvature tensor $\bR$. It is known that
\begin{align}
\nabla_i k_{jm}-\nabla_j k_{im} ={\ep_{ij}}^l H_{lm},\label{3.14.1}\\
R_{ij}-k_{ia}k^{aj}+\Tr k \, k_{ij}=E_{ij}.\label{3.14.2}
\end{align}
From Lemma \ref{es1} and Lemma \ref{first} it follows that
\begin{equation}\label{Rkeh1}
\|Ric\|_{L^2}+\|k\|_{H^1}+\|E\|_{L^2}+\|H\|_{L^2}\le C,
\end{equation}
where and in the following all the norms are taken over a fixed slice
$\Sigma_t$ which is suppressed for simplicity.

In order to obtain the derivative estimates, by straightforward calculation we have
symbolically
\begin{align}
&\nab_m E_{ij}=\bd_m \bR_{0i0j}-k\cdot H, \label{nabe1}\\
&\nab_m H_{ij}=\bd_m {}^\star\bR_{0i0j}-k\cdot E,\label{nabe2}\\
\nab^2_{mn} E_{ij}&=\bd^2_{mn}\bR_{0i0j}-k_{mn} \bd_0
\bR_{0i0j}-\nab(k\cdot H), \label{2de}\\
\nab^2_{mn} H_{ij}&=\bd^2_{mn}{}^\star \bR_{0i0j}-k_{mn} \bd_0
{}^\star \bR_{0i0j}-\nab(k\cdot E).\label{2dh}
\end{align}
From (\ref{nabe1}) and (\ref{nabe2}) it follows that
\begin{align*}
\|\nab E\|_{L^2}&\le
\|\bd\bR\|_{L^2}+\|k\|_{L^6}\|H\|_{L^3}\\
\|\nab H\|_{L^2}&\le
\|\bd\bR\|_{L^2}+\|k\|_{L^6}\|E\|_{L^3}
\end{align*}
Applying Lemma \ref{sob.02} to $\|E\|_{L^3}$ and $\|H\|_{L^3}$,
and using (\ref{R4}) and (\ref{Rkeh1}), we obtain
$$
\|\nabla E\|_{L^2} +\|\nabla H\|_{L^2}
\le C+ C \left(\|\nabla E\|_{L^2}^{1/2} +\|\nabla H\|_{L^2}^{1/2}\right)
$$
which implies
\begin{equation}\label{lex1}
\|\nab E\|_{L^2}+\|\nab H\|_{L^2}\le C.
\end{equation}

Next we will derive the estimate for $\|\nabla^2 k\|_{L^2}$.
It follows from $\mbox{div} k=0$ and (\ref{3.14.1}) that
$\triangle k= Ric \cdot k+\nabla H$. Differentiating it and commuting
$\nabla$ with $\Delta$ gives
$$
\triangle \nab k=Ric \cdot \nab k +\nab Ric \cdot k +\nabla^2 H
$$
which together with (\ref{3.14.2}) implies
\begin{equation}\label{dfhk2}
\triangle \nabla k=k\cdot k \cdot \nabla k +E\cdot \nabla k + \nabla E\cdot k +\nabla^2 H.
\end{equation}
Multiplying (\ref{dfhk2}) by $\nab k$ and integrating over $\Sigma_t$ yields
\begin{align*}
\int_{\Sigma_t} |\nabla^2 k|^2 &\les \int_{\Sigma_t} \left(|k|^2 |\nabla k|^2
+|E||\nabla k|^2 +|\nabla E| |k| |\nabla k|+ |\nabla H| |\nabla^2 k|\right) \nn \\
&\les \|k\|_{L^6}^2 \|\nabla k\|_{L^3}^2 + \|E\|_{L^6} \|\nabla
k\|_{L^{12/5}}^2
+\|\nabla E\|_{L^2} \|\nabla k\|_{L^3} \|k\|_{L^6} \nn \\
&\quad \,  +\|\nabla H\|_{L^2} \|\nabla^2 k\|_{L^2}.
\end{align*}
With the help of Lemma \ref{sob.02}, (\ref{Rkeh1}) and (\ref{lex1}), we have
$$
\|\nabla^2 k\|_{L^2}^2 \le C\left(\|\nabla^2 k\|_{L^2} +\|\nab^2 k\|_{L^2}^{1/2}\right)
$$
which implies $\|\nabla^2 k \|_{L^2} \le C$. By the Sobolev embedding we obtain
\begin{equation}\label{k2}
\|k\|_{L^\infty}+\|k\|_{H^2}\le C.
\end{equation}
Using (\ref{k2}) and (\ref{R4}), it follows easily from (\ref{3.14.2}), (\ref{2de})
and (\ref{2dh}) that
\begin{equation*}
\|\nabla Ric\|_{L^2}+\|\nab^2 Ric\|_{L^2} + \|\nab^2
E\|_{L^2}+\|\nab^2 H\|_{L^2}\le C.
\end{equation*}

Finally we derive the estimate on $\|\nabla^3 k\|_{L^2}$. By
differentiating (\ref{dfhk2}), commuting $\nabla$ with $\triangle$
and using (\ref{3.14.2}) we obtain
\begin{align*}
\triangle \nabla^2 k&= k\cdot k\cdot \nabla^2 k + k\cdot \nabla k
\cdot \nabla k +E\cdot \nabla^2 k +\nabla E\cdot \nabla k + \nabla^2
E\cdot k +\nabla^3 H.
\end{align*}
Multiplying this equation by $\nabla^2k$ and integrating over $\Sigma_t$ it follows
\begin{align*}
\int_{\Sigma_t}  |\nabla^3 k|^2&\les \int_{\Sigma_t} \left(|k|^2 |\nabla^2 k|^2
+ |k||\nabla k|^2 |\nabla^2 k|+ |E||\nabla^2 k|^2 +|\nabla E| |\nabla k| |\nabla ^2 k|\right)\\
&\quad +\int_{\Sigma_t} \left(|k||\nabla^2 E| |\nabla^2 k|+|\nabla^2 H| |\nabla^3 k|\right)\\
&\les \|k\|_{L^\infty}^2 \|\nabla^2 k\|_{L^2}^2
+\|k\|_{L^\infty}\|\nabla k\|_{L^4}^2 \|\nabla^2 k\|_{L^2}
+\|E\|_{L^\infty} \|\nabla^2 k\|_{L^2}^2\\
&\quad\, + \|\nabla E\|_{L^4}\|\nabla k\|_{L^4} \|\nabla^2 k\|_{L^2}
+\|k\|_{L^\infty} \|\nabla^2 E\|_{L^2} \|\nabla^2 k\|_{L^2} \\
&\quad \, +\|\nabla^2 H\|_{L^2} \|\nabla^3 k\|_{L^2}\\
&\le C+ C\|\nabla^3 k\|_{L^2}.
\end{align*}
Therefore $\|\nabla^3 k\|_{L^2}\le C$. The proof is thus complete.

\section{\bf Appendix}\label{ap}

\setcounter{equation}{0}

In this appendix we give the proof of Lemma \ref{tr1}. It suffices
to consider the case that $F$ is an arbitrary smooth function on
$\Sigma_t.$

On $\itt S_{t,u_M(t)}=\cup_{u_m(t)\le u\le
u_M(t)}S_{t,u}$,  we have the family of diffeomorphisms
$$\Psi_{u,t}: {\Bbb S}^2\rightarrow S_{t,u},\,
\Psi_{u,t}(\omega)={\mathcal G}_{\Phi(t_M(u))}(t,\omega), $$ where
${\mathcal G}$ is defined in (\ref{gp}).
Relative to this radial foliation,
the metric $g$ on $\itt S_{t,u_M}$ can
be written as
$$a^2 du^2+\ga_{AB} d\omega_A d\omega_B,
$$
where $\ga$ is the restriction of $g$ on $S_{t,u}$.
By  Lemma \ref{comp1}, $\ga_{S_{t,u}}\approx r(t,u)^2 \ga_{{\Bbb S}^2}$.
Moreover, $F(x)$, $x\in S_{t,u}$, can be reparametrized by $$F(x)=F(u,
\omega):=F\circ \Psi_{u,t}( \omega), \, \omega\in {\Bbb S}^2.$$
Due to Lemma \ref{comp1} and (\ref{commr1}), for a
scalar function $f$,
\begin{equation}\label{el2}
\|f\|_{L^2(S_{t, u})}^2\approx \int_{{\Bbb S}^2} |f|^2 (u-u_m)^2
d\mu_{{\Bbb S}^2}.
\end{equation}

For a fixed leaf $S_{t,u_0}$ with $u_m(t)<u_0\le u_M(t)$ and any $x
\in S_{t, u_0}$, $F(x)=F(u_0, \omega)$ with $\omega\in {\Bbb S}^2$, we define
with $u_m:=u_m(t)$
\begin{align*}
m(x)&:=\frac{2}{u_0-u_m}\int_0^{\frac{u_0-u_m}{2}} F(-z+u_0,
\omega)
dz\\
G(x)&:=m(x)-F(x).
\end{align*}
Lemma \ref{tr1} can be proved by
establishing the following estimates
\begin{align}
&r^{-1/2}\|G\|_{L^2(S_{t,u_0})}\les \|\nab
F\|_{L^2(\Sigma_t)}\label{ftr1}\\& r^{-1/2}
\|m\|_{L^2(S_{t,u_0})}\les \|F\|_{H^1(\Sigma_t)}\label{ftr2},
\end{align}
where $r=r(t,u_0)$.

To see (\ref{ftr1}), according to definition, we have
\begin{align}
G(x)
&=\frac{2}{{u_0-u_m}}\int_0^{\frac{u_0-u_m}{2}} \int_0^1
\frac{d}{d\ell} F( u_0-\ell z, \omega) d\ell dz.\label{el3}
\end{align}
It is easy to see
\begin{equation*}
\frac{d}{d\ell} F(u_0-\ell z,\omega)=-z\c\p_u F(u_0-\ell z,\omega).
\end{equation*}
In view of  ({\bf BA1}) and $N=-a^{-1}\p_u$,  it follows that
\begin{equation}
\left|\frac{d}{d\ell} F(u_0-\ell z,\omega)\right|\les z|\nab_N
 F|(u_0-\ell z,\omega) .\label{el1}
\end{equation}
Since $0<z\le \frac{u_0-u_m}{2}$, we have from (\ref{commr1}) that
$z\les r'$, where $r'=r(t,u_0-\ell z).$ Thus, by combining (\ref{el3}) with (\ref{el1})
and setting $v(y):=\|r'\nab F(-y+u_0, \cdot)\|_{L^2_\omega}$ it yields
\begin{align*}
r^{-1/2}\|G\|_{L^2(S_{t,u_0})}&\les r^{-1/2} \int_0^1 d\ell
\int_0^{\frac{u_0-u_m}{2}}r'\||\nab F(-\ell z+u_0, \cdot)|_g\|_{L^2_\omega}d z\\
&\les r^{-1/2} \int_0^1 \ell^{-1} \int_0^{\frac{\ell(u_0-u_m)}{2}}
v(y) dy d\ell\\
&\les r^{-1/2}\int_0^{
\frac{u_0-u_m}{2}}\int_{\frac{2 y}{u_0-u_m}}^1 \ell^{-1} d\ell v(y) d y\nn\\
&\les  (u_0-u_m)^{-\frac{1}{2}}\int_0^{\frac{u_0-u_m}{2}}
\ln\left(\frac{u_0-u_m}{2y}\right) v(y) dy.
\end{align*}
By H\"older inequality,
\begin{align*}
r^{-1/2} \|G\|_{L^2(S_{t,u_0})}&\les \left(\int_0^1 (\ln \s)^2
d\s\right)^{1/2} \left(\int_0^{\frac{u_0-u_m}{2}} |v(y)|^2 dy\right)^{1/2}\\
& \les
\|\nab F\|_{L^2(\itt S_{t ,u_0})}.
\end{align*}
This proves (\ref{ftr1}). Using (\ref{commr1}), with
$r'':=r(t,u_0-z)\approx u_0-u_m-z$,
\begin{align*}
\|m\|_{L^2(S_{t,u_0})}&\les  \int_0^{\frac{u_0-u_m}{2}}
\|{r''}^{1/3}F\|_{L_\omega^6} {r''}^{-\frac{1}{3}} du'\les r^{\f12}
\|F\|_{L^6(\itt S_{t,u_0})}.
\end{align*}
By Sobolev embedding, (\ref{ftr2}) follows.





\end{document}